\documentclass[12pt,twoside]{preprint}
\usepackage{amssymb}
\usepackage{amsmath}

\usepackage{hyperref}
\usepackage{breakurl}

\usepackage{mhenvs}
\usepackage{mhequ}
\usepackage{mhsymb}
\usepackage{booktabs}
\usepackage{tikz}
\usepackage{mathrsfs}
\usepackage[osf]{newtxtext}
\usepackage{microtype}
\usepackage{comment}

\makeatletter
\pgfdeclareshape{crosscircle}
{
  \inheritsavedanchors[from=circle] 
  \inheritanchorborder[from=circle]
  \inheritanchor[from=circle]{north}
  \inheritanchor[from=circle]{north west}
  \inheritanchor[from=circle]{north east}
  \inheritanchor[from=circle]{center}
  \inheritanchor[from=circle]{west}
  \inheritanchor[from=circle]{east}
  \inheritanchor[from=circle]{mid}
  \inheritanchor[from=circle]{mid west}
  \inheritanchor[from=circle]{mid east}
  \inheritanchor[from=circle]{base}
  \inheritanchor[from=circle]{base west}
  \inheritanchor[from=circle]{base east}
  \inheritanchor[from=circle]{south}
  \inheritanchor[from=circle]{south west}
  \inheritanchor[from=circle]{south east}
  \inheritbackgroundpath[from=circle]
  \foregroundpath{
    \centerpoint%
    \pgf@xc=\pgf@x%
    \pgf@yc=\pgf@y%
    \pgfutil@tempdima=\radius%
    \pgfmathsetlength{\pgf@xb}{\pgfkeysvalueof{/pgf/outer xsep}}%
    \pgfmathsetlength{\pgf@yb}{\pgfkeysvalueof{/pgf/outer ysep}}%
    \ifdim\pgf@xb<\pgf@yb%
      \advance\pgfutil@tempdima by-\pgf@yb%
    \else%
      \advance\pgfutil@tempdima by-\pgf@xb%
    \fi%
    \pgfpathmoveto{\pgfpointadd{\pgfqpoint{\pgf@xc}{\pgf@yc}}{\pgfqpoint{-0.707107\pgfutil@tempdima}{-0.707107\pgfutil@tempdima}}}
    \pgfpathlineto{\pgfpointadd{\pgfqpoint{\pgf@xc}{\pgf@yc}}{\pgfqpoint{0.707107\pgfutil@tempdima}{0.707107\pgfutil@tempdima}}}
    \pgfpathmoveto{\pgfpointadd{\pgfqpoint{\pgf@xc}{\pgf@yc}}{\pgfqpoint{-0.707107\pgfutil@tempdima}{0.707107\pgfutil@tempdima}}}
    \pgfpathlineto{\pgfpointadd{\pgfqpoint{\pgf@xc}{\pgf@yc}}{\pgfqpoint{0.707107\pgfutil@tempdima}{-0.707107\pgfutil@tempdima}}}
  }
}
\makeatother

\colorlet{symbols}{black}      
\colorlet{testcolor}{green!60!black}

\def\symbol#1{\textcolor{symbols}{#1}}

\def\1{\mathbf{\symbol{1}}}

\usetikzlibrary{shapes.misc}
\usetikzlibrary{shapes.symbols}
\usetikzlibrary{shapes.geometric}
\usetikzlibrary{snakes}
\usetikzlibrary{decorations}
\usetikzlibrary{decorations.markings}

\def\drawx{\draw[-,solid] (-3pt,-3pt) -- (3pt,3pt);\draw[-,solid] (-3pt,3pt) -- (3pt,-3pt);}
\tikzset{
	root/.style={circle,fill=testcolor,inner sep=0pt, minimum size=2mm},
	dot/.style={circle,fill=black,inner sep=0pt, minimum size=1mm},
	var/.style={circle,fill=black!10,draw=black,inner sep=0pt, minimum size=2mm},
	dotred/.style={circle,fill=black!50,inner sep=0pt, minimum size=2mm},
	generic/.style={semithick,shorten >=1pt,shorten <=1pt},
	dist/.style={ultra thick,draw=testcolor,shorten >=1pt,shorten <=1pt},
	testfcn/.style={ultra thick,testcolor,shorten >=1pt,shorten <=1pt,<-},
	testfcnx/.style={ultra thick,testcolor,shorten >=1pt,shorten <=1pt,<-,
		postaction={decorate,decoration={markings,mark=at position 0.6 with {\drawx}}}},
	kprime/.style={semithick,shorten >=1pt,shorten <=1pt,densely dashed,->},
	kprimex/.style={semithick,shorten >=1pt,shorten <=1pt,densely dashed,->,
		postaction={decorate,decoration={markings,mark=at position 0.4 with {\drawx}}}},
	kernel/.style={semithick,shorten >=1pt,shorten <=1pt,->},
	multx/.style={shorten >=1pt,shorten <=1pt,
		postaction={decorate,decoration={markings,mark=at position 0.5 with {\drawx}}}},
	kernelx/.style={semithick,shorten >=1pt,shorten <=1pt,->,
		postaction={decorate,decoration={markings,mark=at position 0.4 with {\drawx}}}},
	kernel1/.style={->,semithick,shorten >=1pt,shorten <=1pt,postaction={decorate,decoration={markings,mark=at position 0.45 with {\draw[-] (0,-0.1) -- (0,0.1);}}}},
	kernel2/.style={->,semithick,shorten >=1pt,shorten <=1pt,postaction={decorate,decoration={markings,mark=at position 0.45 with {\draw[-] (0.05,-0.1) -- (0.05,0.1);\draw[-] (-0.05,-0.1) -- (-0.05,0.1);}}}},
	kernelBig/.style={semithick,shorten >=1pt,shorten <=1pt,decorate, decoration={zigzag,amplitude=1.5pt,segment length = 3pt,pre length=2pt,post length=2pt}},
	rho/.style={dotted,semithick,shorten >=1pt,shorten <=1pt},
	renorm/.style={shape=circle,fill=white,inner sep=1pt},
	labl/.style={shape=rectangle,fill=white,inner sep=1pt},
cumu2n/.style={inner sep=3pt},
cumu2/.style={draw=red!50,fill=red!20},
cumu3/.style={regular polygon, regular polygon sides=3,draw=red!50,rounded corners=3pt,fill=red!20,minimum size=5mm},
cumu4/.style={regular polygon, regular polygon sides=4,draw=red!50,rounded corners=3pt,fill=red!20,minimum size=7mm},
cumu5/.style={regular polygon, regular polygon sides=5,draw=red!50,rounded corners=3pt,fill=red!20,minimum size=5mm},
	xi/.style={circle,fill=symbols!10,draw=symbols,inner sep=0pt,minimum size=1.2mm},
	xix/.style={crosscircle,fill=symbols!10,draw=symbols,inner sep=0pt,minimum size=1.2mm},
	xib/.style={circle,fill=symbols!10,draw=symbols,inner sep=0pt,minimum size=1.6mm},
	xibx/.style={crosscircle,fill=symbols!10,draw=symbols,inner sep=0pt,minimum size=1.6mm},
	not/.style={circle,fill=symbols,draw=symbols,inner sep=0pt,minimum size=0.5mm},
	>=stealth,
	}
\makeatletter
\def\DeclareSymbol#1#2#3{\expandafter\gdef\csname MH@symb@#1\endcsname{\tikz[baseline=#2,scale=0.15,draw=symbols]{#3}}\expandafter\gdef\csname MH@symb@#1s\endcsname{\scalebox{0.7}{\tikz[baseline=#2,scale=0.15,draw=symbols]{#3}}}}
\def\<#1>{\csname MH@symb@#1\endcsname}
\makeatother

\DeclareSymbol{Xi22}{0.5}{\draw (0,0) node[xi] {} -- (-1,1) node[not] {} -- (0,2) node[xi] {};}

\DeclareSymbol{Xi2}{-2}{\draw (0,-0.25) node[xi] {} -- (-1,1) node[xi] {};}
\DeclareSymbol{Xi3}{0}{\draw (0,0) node[xi] {} -- (-1,1) node[xi] {} -- (0,2) node[xi] {};}
\DeclareSymbol{Xi4}{2}{\draw (0,0) node[xi] {} -- (-1,1) node[xi] {} -- (0,2) node[xi] {} -- (-1,3) node[xi] {};}
\DeclareSymbol{Xi2X}{-2}{\draw (0,-0.25) node[xi] {} -- (-1,1) node[xix] {};}
\DeclareSymbol{XXi2}{-2}{\draw (0,-0.25) node[xix] {} -- (-1,1) node[xi] {};}

\DeclareSymbol{IXi2}{0}{\draw (0,-0.25) node[not] {} -- (-1,1) node[xi] {} -- (0,2) node[xi] {};}
\DeclareSymbol{IXi^2}{-1}{\draw (-1,1) node[xi] {} -- (0,0) node[not] {} -- (1,1) node[xi] {};}

\DeclareSymbol{XiX}{-2.8}{\node[xibx] {};}
\DeclareSymbol{Xi}{-2.8}{\node[xib] {};}
\DeclareSymbol{IXiX}{-1}{\draw (0,-0.25) node[not] {} -- (-1,1) node[xix] {};}

\DeclareSymbol{Xi3b}{-1}{\draw (-1,1) node[xi] {} -- (0,0) node[xi] {} -- (1,1) node[xi] {};}

\DeclareSymbol{IXi3}{2}{\draw (0,-0.25) node[not] {} -- (-1,1) node[xi] {} -- (0,2) node[xi] {} -- (-1,3) node[xi] {};}
\DeclareSymbol{IXi}{-2}{\draw (0,-0.25) node[not] {} -- (-1,1) node[xi] {};}
\DeclareSymbol{XiI}{-2}{\draw (0,-0.25) node[xi] {} -- (-1,1) node[not] {};}

\DeclareSymbol{Xi4b}{-1}{\draw(0,1.5) node[xi] {} -- (0,0); \draw (-1,1) node[xi] {} -- (0,0) node[xi] {} -- (1,1) node[xi] {};}
\DeclareSymbol{Xi4b'}{-1}{\draw(0,1.5) node[xi] {} -- (0,-0.2); \draw (-1,1) node[xi] {} -- (0,-0.2) node[not] {} -- (1,1) node[xi] {};}
\DeclareSymbol{Xi4c}{0}{\draw (0,1) -- (0.8,2.2) node[xi] {};\draw (0,-0.25) node[xi] {} -- (0,1) node[xi] {} -- (-0.8,2.2) node[xi] {};}
\DeclareSymbol{Xi4d}{-4.5}{\draw (0,-1.5) node[not] {} -- (0,0); \draw (-1,1) node[xi] {} -- (0,0) node[xi] {} -- (1,1) node[xi] {};}
\DeclareSymbol{Xi4e}{0}{\draw (0,2) node[xi] {} -- (-1,1) node[xi] {} -- (0,0) node[xi] {} -- (1,1) node[xi] {};}
\DeclareSymbol{Xi4e'}{0}{\draw (0,2) node[xi] {} -- (-1,1) node[xi] {} -- (0,-0.2) node[not] {} -- (1,1) node[xi] {};}

\newtheorem{assumption}[lemma]{Assumption}

\def\s{\mathfrak{s}}

\def\K{\mathfrak{K}}

\def\TT{\mathscr{T}}

\def\${|\!|\!|}

\def\Z{\mathbb{Z}}

\def\R{\mathbb{R}}
\def\N{\mathbb{N}}

\def\cB{\mathcal{B}}

\def\cR{\mathcal{R}}

\def\CK{\mathcal{K}}
\def\sD{\mathscr{D}}
\def\bfc{\mathbf{c}}
\def\dbar\K{\bar{\bar\K}}

\def\I{\mathrm{I}}
\def\II{\mathrm{I\kern-0.1emI}}
\def\III{\mathrm{I\kern-0.1emI\kern-0.1emI}}
\def\IV{\mathrm{I\kern-0.1emV}}

\def\id{\mathrm{id}}

\DeclareSymbol{X}{-2.4}{\node[dot] {};}
\DeclareSymbol{1}{0}{\draw[white] (-.4,0) -- (.4,0); \draw (0,0)  -- (0,1.2) node[dot] {};}
\DeclareSymbol{2}{0}{\draw (-0.6,1.3) node[dot] {} -- (0,0) -- (0.6,1.3) node[dot] {};}

\DeclareSymbol{3}{0}{\draw (0,0) -- (0,1.2) node[dot] {}; \draw (-.7,1) node[dot] {} -- (0,0) -- (.7,1) node[dot] {};}

\DeclareSymbol{31}{-3}{\draw (0,0) -- (0,-1) -- (1,0) node[dot] {}; \draw (0,0) -- (0,1.2) node[dot] {}; \draw (-.7,1) node[dot] {} -- (0,0) -- (.7,1) node[dot] {};}
\DeclareSymbol{30}{-3}{\draw (0,0) -- (0,-1); \draw (0,0) -- (0,1.2) node[dot] {}; \draw (-.7,1) node[dot] {} -- (0,0) -- (.7,1) node[dot] {};}
\DeclareSymbol{32}{-3}{\draw (0,0) -- (0,-1) -- (1,0) node[dot] {}; \draw (0,0) -- (0,-1) -- (-1,0) node[dot] {}; \draw (0,0) -- (0,1.2) node[dot] {}; \draw (-.7,1) node[dot] {} -- (0,0) -- (.7,1) node[dot] {};}
\DeclareSymbol{22}{-3}{\draw (0,0.3) -- (0,-1) -- (1,0) node[dot] {}; \draw (0,0.3) -- (0,-1) -- (-1,0) node[dot] {};\draw (-.7,1) node[dot] {} -- (0,0.3) -- (.7,1) node[dot] {};}
\DeclareSymbol{20}{-3}{\draw (0,0) -- (0,-1);\draw (-.7,1) node[dot] {} -- (0,0) -- (.7,1) node[dot] {};}
\DeclareSymbol{12}{-3}{\draw (0,0.3) -- (0,-1) -- (1,0) node[dot] {}; \draw (0,0.3) -- (0,-1) -- (-1,0) node[dot] {};\draw (-.7,1) node[dot] {} -- (0,0.3);}
\DeclareSymbol{10}{-3}{\draw (0,0.3) -- (0,-1);\draw (-.7,1) node[dot] {} -- (0,0.3);}
\DeclareSymbol{21}{-3}{\draw (0,0.3) -- (0,-1) -- (1,0) node[dot] {};\draw (-.7,1) node[dot] {} -- (0,0.3) -- (.7,1) node[dot] {};}
\DeclareSymbol{21a}{-3}{\draw (0,0.3) -- (0,-1) -- (1,0) node[dot] {};\draw  (0,0.3) -- (.7,1) node[dot] {};}
\DeclareSymbol{211}{3}{
	\draw (0,0) -- (-2,3) node[dot] {};
	\draw (0,0)  -- (1,1) node[dot] {};
	\draw (-0.67,1)  -- (.33,2) node[dot] {};
	\draw (-1.33,2)  -- (-0.33,3) node[dot] {};
}
\DeclareSymbol{4}{-3}{
	\draw (-1,0)  -- (0,-1.2) -- (1,0) ;
	\draw (-1.5,1) node[dot] {} -- (-1,0) -- (-0.5,1) node[dot] {};  
	\draw (1.5,1) node[dot] {} -- (1,0) -- (0.5,1) node[dot] {};
}

\definecolor{Red}{rgb}{1,0,0}
\definecolor{Blue}{rgb}{0,0,1}
\definecolor{Olive}{rgb}{0.41,0.55,0.13}
\definecolor{Yarok}{rgb}{0,0.5,0}
\definecolor{Green}{rgb}{0,1,0}
\definecolor{MGreen}{rgb}{0,0.8,0}
\definecolor{DGreen}{rgb}{0,0.55,0}
\definecolor{Yellow}{rgb}{1,1,0}
\definecolor{Cyan}{rgb}{0,1,1}
\definecolor{Magenta}{rgb}{1,0,1}
\definecolor{Orange}{rgb}{1,.5,0}
\definecolor{Violet}{rgb}{.5,0,.5}
\definecolor{Purple}{rgb}{.75,0,.25}
\definecolor{Brown}{rgb}{.75,.5,.25}
\definecolor{Grey}{rgb}{.7,.7,.7}
\definecolor{Black}{rgb}{0,0,0}

\def\diam{\mathop{\mathrm{diam}}}

\begin{document}

\title{Discretisation of regularity structures}
\author{Dirk Erhard$^1$ and Martin Hairer$^2$}
\institute{University of Warwick, UK, \email{D.Erhard@warwick.ac.uk} 
\and University of Warwick, UK, \email{M.Hairer@Warwick.ac.uk}}

\maketitle

\begin{abstract}
We introduce a general framework allowing to apply the theory of regularity structures to
discretisations of stochastic PDEs. The approach pursued in this article is that we do not
focus on any one specific discretisation procedure. Instead, we assume that we are given a 
scale $\eps > 0$ and a ``black box'' describing the behaviour of our discretised objects
at scales below $\eps$.
\end{abstract}

\setcounter{tocdepth}{2}
\tableofcontents

\section{Introduction}
\label{S1}

The theory of regularity structures is a framework developed by the second author in~\cite{Regularity} which
allows to renormalise stochastic PDEs of the form
\begin{equation}
\label{eq:generalpde}
\CL u = F(u,\nabla u,\xi),
\end{equation}
that are ill-posed in the classical sense. Here, $\CL$ is typically a parabolic differential operator, $\xi$ is a very irregular random input and $F$ is some local non-linearity. 
The na\"{i}ve approach to study well-posedness of~\eqref{eq:generalpde} would be to consider a sequence of smooth approximations
\begin{equation}
\label{eq:smoothgeneralpde}
\CL u_{\eps} =F(u_\eps,\nabla u_\eps,\xi_\eps),
\end{equation}
and simply declare the limit of this sequence to be the solution to~\eqref{eq:generalpde}.
However, in many cases it turns out that the sequence $u_\eps$ either does not possess a limit or the limit is trivial.
A way to circumvent this is to renormalise, which can be interpreted as a kind of ``recentering'' of the equation,
see \cite{BHZalg}. 

This amounts to allowing some of the constants appearing in the definition of $F$ appearing in the right hand side of~\eqref{eq:smoothgeneralpde} to be $\eps$-dependent and more specifically to diverge as $\eps$ tends to zero.
The theory of regularity structures is very successful when applied to approximations $u_\eps$ that are 
functions defined on $\R^d$,
it however presently does not in general apply to discrete approximations of stochastic PDEs or to equations where 
the operator $\CL$ itself is perturbed by a higher-order term, in such a way that its
scaling properties are different at small scales (think for instance of $\CL$ given by $\partial_t-\Delta$ and $\CL_\eps= \partial_t-\Delta + \eps^2 \Delta^2$).
In the present article, we develop a general framework that is able to deal with these cases.
The guiding principle is a separation of scales, i.e., above a certain scale $\eps$ we show that the theory of regularity structures still applies, whereas scales below $\eps$ are treated as a black box.
As a result, our framework is flexible and does not rely on any specific discretisation procedure. More precisely
the results in the article show that, given a stochastic PDE such that the theory of regularity structures 
can be applied to renormalise it, our present framework can be used to treat a large class of natural discretisations
for it.

\subsection{Discussion}

Let us first comment on our motivation for this work as well as its relation to some previous work.
One of the guiding questions in the study of particle systems is the question of universality, i.e., what 
are the characterising features for a class of particle systems that converge to the same limit 
under appropriate rescaling and how can one prove this convergence?
The bulk of the literature is concerned with scaling limits of a \textit{fixed} system which necessarily
implies that the resulting limit is itself a scale-invariant object.
The present article however is motivated by the situation where one considers a \textit{family} of systems
indexed by some parameter and one simultaneously tunes this parameter as one rescales the system.
In this way, one typically obtains scaling limits that are \textit{not} scale invariant themselves.
One of the insights of the theory presented in \cite{Regularity} is that they are however \textit{locally}
described by linear combinations of objects that are scale-invariant, but with different scaling exponents.

So far, the state of the art for answering such questions (in the second case where the limiting object
can, at least formally, be described by a singular stochastic PDE) is to 
heavily rely on special features of the model(s) under consideration. In the type of situation of interest to
us, some standard techniques consist of
\begin{claim}
\item extending the discrete equation to an equation defined on $\R^d$, see for instance the articles of Gubinelli and Perkowski~\cite{KPZreloaded}, Mourrat and Weber~\cite{MourratWeber}, Shen and Weber~\cite{ShenWeber}, Zhu and Zhu~\cite{ZhuZhuNS,ZhuZhuPhi};
\item linearising the problem via a Hopf-Cole transformation, see for example the articles of Bertini and Giacomin~\cite{MR1462228}, Corwin and Shen~\cite{CorwinShen}, Corwin, Shen and Tsai~\cite{CorwinShenTsai}, Corwin and Tsai~\cite{CorwinTsai}, Dembo and Tsai~\cite{DemboTsai}, Labb\'e~\cite{Labbe} in the setting of the KPZ equation;
\item transforming the equation into a martingale problem, see for instance the series of articles by Gon\c{c}alvez and Jara~\cite{GoncalvesJaraKPZ}, Gubinelli and Jara~\cite{GubinelliJara}, Gubinelli and Perkowski~\cite{GubinelliPerkowskienergy}, in which the concept of energy solution is developed to study the stochastic Burgers / KPZ equation.
\end{claim} 
When applied in the correct context these techniques can be very powerful. The drawback however is that they are often quite sensitive to small perturbations of the model. Another problem of analytical techniques like regularity structures
or paraproducts is that while they provide a rather clean and general-purpose toolbox in the continuum, their
extensions to the discrete setting require purpose-built modifications that are much less reusable.
It is therefore desirable to have a robust theory at hand which is insensitive to the details of the
underlying discrete setting, and it is the goal of the present article to make a first attempt at 
developing such a theory.

Another motivation to introduce the framework developed here stems from the fact that a common way to derive properties of a stochastic PDE is to approximate it by discrete systems for which the desired property holds and then show
that it remains stable under passage to the limit. This methodology was for instance successfully applied by Hairer and Matetski~\cite{MatetskiDiscrete} to prove that the $\Phi_3^4$-measure built in \cite{GlimmJaffe,Fel74}
is invariant for the $\Phi_3^4$-equation.
In the specific case where the solution to the stochastic PDE at hand is a function of time, \cite{MatetskiDiscrete} developed a framework adapting the theory of regularity structures to allow for certain spatial discretisations. 

\subsection{Strategy} 
This article aims to justify the philosophy that small scales actually do not matter much. The way we formulate this here is that we 
assume some reasonable (albeit technical) assumptions on scales smaller than $\eps$, that are almost independent of the specific discretisation one deals with, and we show that this implies desirable estimates on all scales up to order one. The precise strategy to do so then results in a series of statements that are described further below.

The main idea to accommodate a large class of discretisations of stochastic PDEs is to consider the behaviour below scale $\eps$
as encoded in a ``black box''. In order to describe this, our main ingredient
is a sequence of linear spaces $\CX_\eps$ that can be viewed as subspaces of $\CD'(\R^d)$, the space of distributions, and 
that possess a natural family of (extended) seminorms. We then work with stochastic PDEs of the type~\eqref{eq:smoothgeneralpde} 
such that $u_\eps\in\CX_\eps$. 
Examples for $\CX_\eps$ include $\CD'(\R^d)$ (thus recovering the original setting in~\cite{Regularity}), but also the space of functions defined on a discrete grid of mesh size $\eps$, or simply some space of smooth functions.
To cast this into the framework of regularity structures we first work with an abstract version of $u_\eps$, i.e., we write $u_\eps$ as a generalised Taylor expansion. Here, the monomials may represent the classical Taylor monomials or they may represent abstract expressions that are functions of $\xi$. The control of these monomials typically amount to the control of explicit stochastic objects and heavily depend on the choice of the equation at hand.

The ``Taylor coefficients'' can be thought of as the ``derivatives'' of an abstract version of a $\CC^\gamma$-function and they are given by 
the solution of an abstract fixed point problem. We denote the class of abstract functions just described by $\CD_\eps^{\gamma}$. The subindex $\eps$ indicates that the way regularity is measured above and below scale $\eps$ may differ. Indeed, above scale $\eps$ regularity is measured as in the continuous setting~\cite{Regularity}, whereas below scale $\eps$, besides some natural constraints, the way to measure regularity is not further determined. The idea of introducing $\CD^\gamma$-spaces that depend on a parameter $\eps$ is not 
new and already appeared in the works~\cite{KPZJeremy,Phi4Weijun}. In a sense the current article generalises some of the ideas developed there.

Note that the solution to~\eqref{eq:smoothgeneralpde} is a random space-time function / distribution, whereas the approach outlined above gives rise to an abstract ``modelled distribution'' $f\in\CD_\eps^{\gamma}$. To link the abstract object with a concrete object we construct a  ``reconstruction map'' $\CR^\eps$. Unlike in~\cite{Regularity}, 
we do not think of $\CR^\eps f$ as a distribution, but rather as an element of $\CX_\eps$, which could for
example represent a space of functions defined on a discrete grid at scale $\eps$.  The way to construct $\CR^\eps$ is to \emph{postulate} the existence of a map $\CR^\eps$ satisfying certain estimates on scales below $\eps$ and to then
show that analogous estimates automatically hold on all larger scales. Note that there may be several candidates for $\CR^\eps$ all satisfying the same quantitative estimates on small scales, so that in our context the reconstruction
map is in general not uniquely defined. To proceed, we further need to define operations on $\CD_\eps^{\gamma}$ in order to actually construct abstract solutions to~\eqref{eq:smoothgeneralpde}.
In particular we define an abstract notion of convolution $\CK_\gamma^{\eps}$ against the Green function $K^\eps$ of $\CL$ in~\eqref{eq:smoothgeneralpde} (which is an operator mapping into $\CX_\eps$!) and we show that $\CK_\gamma^{\eps}$ satisfies a certain Schauder estimate. It turns out that unlike in the continuous setting~\cite{Regularity} it is not necessarily true that the convolution operator $\CK_\gamma^{\eps}$ can be defined in a natural way so that 
it intertwines with $\CR^\eps$ in the sense that
\begin{equation}
\label{eq:commute}
\CR^\eps\CK_\gamma^{\eps}= K^\eps \CR^\eps.
\end{equation}
We however argue that in many cases one can enforce~\eqref{eq:commute} by tweaking the definition of $\CK_\gamma^{\eps}$.
So far the strategy outlined above appears to only allow us to describe solutions to~\eqref{eq:smoothgeneralpde}, 
but, as shown in~\cite{Regularity} and~\cite{BHZalg}, the encoding of renormalisation procedures in the theory
of regularity structures is of purely algebraic nature and does not depend on any specific discretisation procedure, see also Remark~\ref{rem:algebra} below.

We finish with some concluding remarks that comment on how one does apply the theory developed in this article.
\begin{remark}
\label{rem:algebra}
To apply the theory developed in this article to a concrete problem, an important ingredient is the construction of a 
suitable regularity structure and a suitable renormalised model. In the usual (continuous) case, a framework was 
built in~\cite[Sec.~8]{Regularity}, and further refined in \cite{BHZalg,Ajay} that automatises this construction.
To perform an analogous construction in the present context, one then needs an algebra structure on $\CX_\eps$,
as well as a representation of the Taylor polynomials, which is often the case.
The general analytical results of \cite{Ajay} however fall out the scope of this article and would have to be adapted. 
\end{remark}
\begin{remark}
It was shown in~\cite[Thm~10.7]{Regularity} (see also \cite{NonGaussShen,Ajay}) that in order to obtain 
convergence of a sequence of models to a limiting model it is enough to obtain suitable bounds on sufficiently high moments 
for its terms of negative homogeneity.
An important ingredient of the proof is the recursive definition of the regularity structure and the model, provided by the algebraic framework alluded to above. This is the same in the current case.
Another ingredient is the identity $\Pi_z \tau =\CR f_{\tau,z}$, where $f_{\tau,z}(y)= \Gamma_{yz} \tau-\tau$ for any $\ell\in A$ and $\tau\in\CT_{\ell}$ with positive homogeneity and $f_{\tau,z}(y)= \Gamma_{yz} \tau$ in case $\tau$ has negative homogeneity.
Provided that the norm $\$\cdot\$_{\gamma;\K;\eps}$ of Definition~\ref{def:admissible} below is chosen such that $\$z\mapsto \Gamma_{zz'}^\eps \tau\$_{\ell;\K;\eps}=0$
for any $\ell\in A$, any $\tau\in\CT_{\ell}$ and any $z'\in\R^d$, the aforementioned identity follows from Assumption~\ref{a:rec}.
A further ingredient is the characterisation of the spaces of distributions under consideration by a wavelet basis. 
On scales larger than $\eps$ a similar characterisation holds in the current case, see Definition~\ref{def:model} below. 
On small scales, the situation depends on the choice of seminorms on $\CX_\eps$ introduced in Definition~\ref{def:admissible},
but, in many cases of interest, we expect to have a choice of these seminorms which makes the corresponding bounds simple to verify.
Finally, the last ingredient is the extension theorem~\cite[Thm~5.14]{Regularity}. The analogous statement in the present setup is the content of Theorem~\ref{thm:extension}.
Thus, under the appropriate assumptions~\cite[Thm~10.7]{Regularity} also holds in the framework constructed in this article.
\end{remark}

\subsection{Structure of the article}
In Section~\ref{S2} we develop the framework we will work with, in particular we introduce the sequence of spaces $\CX_\eps$, we define the notion of discrete models and the $\CD_\eps^{\gamma}$-spaces.
In Sections~\ref{S3}--\ref{S6} we explain the main operations on these spaces.

\subsection{Notation} 

Throughout this work $\s=(\s_1,\ldots, \s_d)\in \N_{\geq 1}^d$ denotes a scaling of $\R^d$ and we associate 
to it the metric $d_\s$ on $\R^d$ given by
\begin{equation}
d_\s(y,z)=\sup_{i\in\{1,\ldots, d\}}|y_i-z_i|^{1/\s_i}.
\end{equation}
We sometimes use the notation $\|y-z\|_\s$ in place of $d_\s(y,z)$.
Moreover, we let $|\s|=\s_1+\cdots +\s_d$ and for a multiindex $k$ we use the notation $|k|_\s=\sum_{i}\s_ik_i$. Given a set $B\subset \R^d$ and $z\in\R^d$ we denote the distance of $z$ to $B$ with respect to the metric $d_\s$ by $d_\s(z,B)$.
Given $\delta>0$ and $\varphi:\R^d\to\R$ we set $\CS_\s^{\delta}(z_1,\ldots, z_d)= (\delta^{-\s_1}z_1,\ldots, \delta^{-\s_d}z_d)$ and $(\CS_{\s,z}^{\delta}\varphi)(y)=\delta^{-|\s|}\varphi(\CS_\s^{\delta}(y-z))$. We also use occasionally the notation $\varphi_z^{\delta}$. Moreover, we also write $\varphi_z^n$ in place of $\varphi_z^{2^{-n}}$. For a compact subset $\K$ of $\R^d$ (also written as $\K\Subset\R^d$) we denote
its $1$-fattening and $2$-fattening by $\bar{\K}$ and $\dbar\K$, respectively.
We further denote the $d_\s$-ball of radius $\delta$ around $z\in\R^d$ by
$\CB_\s(z,\delta)$.

We occasionally use the notation $[\eta]$ to denote the support of the function $\eta$.
In this article we use various notions of norms, seminorms and metrics on various spaces. To improve readability we inserted a norm index in the appendix listing all these distances.

Given a distribution $\xi$ and a test function $\phi$, we use interchangeably the notations $\xi(\phi)$,
$\scal{\xi,\phi}$, and $\int \phi(x)\,\xi(dx)$ for the corresponding pairing.

\subsection*{Acknowledgements}

{\small
MH gratefully acknowledges financial support from the
 Leverhulme Trust through a leadership award
 and from the European Research Council through a consolidator grant, project 615897.
}

\section{A framework for discretisations}
\label{S2}

\textbf{Convention:} From now on we assume that we are given a scaling $\s$ of $\R^d$ and
a regularity structure $\TT=(A,\CT, G)$ containing the canonical $\s$-scaled polynomials in $d$ indeterminates. 
We also fix a value $r> |\min A\,|$ and we denote by $\Phi$
the set of all functions $\phi\in \CC^r$ with $\|\phi\|_{\CC^r}\leq 1$ such that $\supp \phi\subset \CB_{\s}(0,1)$
and we simply say that ``$\phi$ is a test function'' whenever 
$\phi\in\Phi$.
Given an element $\tau \in \CT = \bigoplus_{\alpha \in A}\CT_\alpha$, we write $\|\tau\|_\alpha$ for
the norm of its component in $\CT_\alpha$ and $\|\tau\|$ for its norm in $\CT$.
Moreover $\CQ_\alpha$ denotes the projection onto $\CT_\alpha$.

We build a general framework for allowing for discretisations of models and their convergence
to continuous limiting models. Our construction relies crucially on the following notion.
\begin{definition}
\label{def:admissible}
A \textit{discretisation} for the regularity structure $\TT$ and scaling $\s$ on $\R^d$
consists of the following data.
\begin{claim}
\item[1.] 
A collection of linear spaces $\CX_\eps$ with $\eps \in (0,1]$ endowed
with inclusion maps $\iota_\eps \colon \CX_\eps \hookrightarrow \CD'(\R^d)$, so that 
elements of $\CX_\eps$ can be interpreted as distributions.
\item[2.] Each $\CX_\eps$ admits a family of extended seminorms $\|\cdot\|_{\alpha;\K_\eps;z;\eps}$ 
(i.e., we do \emph{not} exclude the possibility that $\|f\|_{\alpha;\K_\eps;z;\eps}=\infty$ for some $f\in\CX_\eps$) with
$\alpha \in \R$, the $\K_\eps$ are compact subsets of $\R^d$ of diameter at most $2\eps$ and $z\in\R^d$.
We moreover assume that these seminorms are local in the sense that if 
$f,g \in \CX_\eps$ and $(\iota_\eps f)(\phi) = (\iota_\eps g)(\phi)$ for every $\phi\in \CC^r$ supported in 
$\K_\eps$, then
$\|f-g\|_{\alpha;\K_\eps;z;\eps} = 0$.
\item[3.] 
Uniformly over all $f\in\CX_\eps$, $z\in\R^d$, $\alpha\in\R$, and $\phi \in \Phi$, one has the bound
\begin{equation}
\label{eq:relationtoseminorm}
|(\iota_\eps f)(\phi_z^{\eps})|\lesssim \eps^{\alpha}\| f\|_{\alpha; [\phi_z^\eps]; z;\eps}.
\end{equation}
\item[4.] For any function $\Gamma\colon \R^d \times \R^d \to G$, there exists 
a family of extended seminorms 
$\$\cdot\$_{\gamma;\K;\eps}$ on the space of functions $f:\R^d\to \CT_{<\gamma}$ with $\gamma\in\R$, $\K\Subset\R^d$ and $\eps\in (0,1]$.
For the same set of indices, and given two functions $\Gamma_1,\Gamma_2:\R^d\times\R^d\to G$, there is a family of
seminorms $\$\cdot;\cdot\$_{\gamma;\K;\eps}$ on the space of pairs $(f,g)$, with $f,g:\R^d\to\CT_{<\gamma}$.
Both families of seminorms are assumed to depend only on the values of $f$ (and $g$ respectively) in a neighbourhood of size $\bfc\eps$ around $\K$, for $\bfc\geq 0$ a fixed constant.
\end{claim}
\end{definition}


\begin{remark}
The above way of introducing the seminorms in the fourth item may be a bit misleading. As we will see in the examples below, they indeed depend on the functions $\Gamma,\Gamma_1$, and $\Gamma_2$ introduced above, so that the correct way of denoting them would be
$\$\cdot\$_{\gamma;\Gamma;\K;\eps}$ and $\$\cdot;\cdot\$_{\gamma;\Gamma_1,\Gamma_2,\K;\eps}$. However, for the sake of 
readability and since they will always be clear from context, we omit these additional indices.
\end{remark}
\begin{remark}
Assume that in the fourth item in the above definition one has $\Gamma_1=\Gamma_2$. In practice one then often has the relation $\$f;g\$_{\gamma;\K;\eps} = \$f-g\$_{\gamma;\K;\eps}$. See for instance the four examples introduced further below,
where this is indeed the case.
\end{remark}

\begin{remark}
In practice we are not interested in all possible functions $\Gamma_1$ and $\Gamma_2$ as above, but in specific sequences of maps $(\Gamma^\eps)_{\eps \in (0,1]}$ that depend on the underlying model, see Definition~\ref{def:model} for an explanation of that terminology.
\end{remark}
To illustrate the setup described above, we mention the following four examples to which we will refer frequently in this work. We leave it as an exercise to verify that these are indeed examples of discretisations in the above sense.
In all examples below $\K_\eps$ denotes a compact subset of $\R^d$ with diameter bounded by $2\eps$.

\begin{itemize}
\item[1.]\textbf{The purely discrete case.} Let $\Lambda_\eps^d\subset\R^d$ be a locally finite set 
such that, for some constants $c, C \in (0,1]$, it holds that for every $z \in \Lambda_\eps^d$,
$\inf_{z \neq z' \in \Lambda_\eps^d} \|z-z'\|_\s \in [c\eps,C\eps]$.
We then set $\CX_\eps = \R^{\Lambda_\eps^d}$ with 
\begin{equ}
\label{eq:discrXeps}
\bigl(\iota_\eps f\bigr)(\phi) = \eps^{|\s|}\sum_{z \in \Lambda_\eps^d} f(z)\,\phi(z)\;,\quad
\|f\|_{\alpha;\K_\eps;z;\eps} = \eps^{-\alpha} \sup_{y\in\K_\eps\cap\Lambda_\eps^d} |f(y)|\;.
\end{equ}
Moreover, we set
\begin{equation}
\$ f\$_{\gamma;\K;\eps}= \sup_{y,z\in\K \cap \Lambda_\eps^d\,:\, \|y-z\|_\s<\eps}\sup_{\beta<\gamma} \eps^{\beta-\gamma}
\|f(z)-\Gamma_{zy}^{\eps}f(y)\|_{\beta}\;,
\end{equation}
and \begin{equation}
\$f;\bar{f}\$_{\gamma;\K;\eps}=  \sup_{y,z\in\K\cap\Lambda_\eps^d\,:\, \|y-z\|_\s<\eps}\sup_{\beta<\gamma} \eps^{\beta-\gamma}
\|f(z)-\Gamma_{zy}^{\eps}f(y)-\bar{f}(z)+\bar\Gamma_{zy}^\eps\bar{f}(y)\|_{\beta}\;,
\end{equation}
for all $\gamma\in\R$, $\eps\in (0,1]$, $\K\Subset\R^d$ and $f,\bar{f}:\R^d\to \CT_{<\gamma}$.
\item[2.]\textbf{The semidiscrete case.} Let $\Lambda_\eps^{d-1}\subset\R^{d-1}$ be as above. 
We let $\CX_\eps= L^\infty(\R \times \Lambda_\eps^{d-1})$ and we write 
$\s=(\s_1,\bar\s)$. The inclusion map $\iota_\eps$ and the family of (extended) seminorms 
$\|\cdot\|_{\alpha;\K_\eps;z;\eps}$ are then defined by
\begin{equation}
\begin{aligned}
(\iota_\eps f)(\varphi) &= \eps^{|\bar{\s}|}\sum_{x\in\Lambda_\eps^{d-1}}\int_{\R}
f(t,x)\varphi(t,x)\, dt, \\
\|f\|_{\alpha;\K_\eps;z; \eps}&=\eps^{-\alpha}\sup_{y\in\K_\eps\cap\R\times\Lambda_\eps^{d-1}}|f(y)|.
\end{aligned}
\end{equation}
Furthermore,
\begin{equation}
\$ f\$_{\gamma;\K;\eps}=\sup_{\substack{y,z\in\K\cap\R\times\Lambda_\eps^{d-1}\\ \|y-z\|_\s<\eps}}\sup_{\beta <\gamma}
\eps^{\beta-\gamma}\| f(z)-\Gamma^{\eps}_{zy}f(y)\|_{\beta},
\end{equation} and
\begin{equation}
\$f;\bar{f}\$_{\gamma;\K;\eps}= \sup_{\substack{y,z\in\K\cap\R\times\Lambda_\eps^{d-1}\\ \|y-z\|_\s<\eps}}\sup_{\beta <\gamma}
\eps^{\beta-\gamma}\| f(z)-\Gamma^{\eps}_{zy}f(y)-\bar f(z)+\bar\Gamma_{zy}^\eps\bar f(y)\|_{\beta},
\end{equation}
for all $\gamma\in\R$, $\eps\in(0,1]$, $\K\Subset\R$ and $f,\bar{f}:\R^d\to \CT_{<\gamma}$.
\item[3.]\textbf{The continuous case.} In this case, we set $\CX_\eps=\CC(\R^d,\R)$ with $\iota_\eps$
given by the canonical identification between continuous functions and distributions.
We set 
\begin{equ}[e:defNormcont]
\|f\|_{\alpha;\K_\eps;z;\eps} = \eps^{-\alpha} \sup_{y\in\K_\eps} |f(y)|\;,
\end{equ}
as well as
\begin{equ}
\$ f\$_{\gamma;\K;\eps}= \sup_{y,z\in\K, \|y-z\|_\s<\eps}\sup_{\beta<\gamma} \eps^{\beta-\gamma}
\|f(z)-\Gamma^{\eps}_{zy}f(y)\|_{\beta}\;,
\end{equ}
and 
\begin{equation}
\$f;\bar{f}\$_{\gamma;\K;\eps}=  \sup_{y,z\in\K, \|y-z\|_\s<\eps}\sup_{\beta<\gamma} \eps^{\beta-\gamma}
\|f(z)-\Gamma^{\eps}_{zy}f(y)-\bar f(z)+\bar\Gamma_{zy}^\eps\bar f(y)\|_{\beta}\;,
\end{equation}
for $\gamma\in\R,\K\Subset \R^d, \eps\in(0,1]$ and $f,\bar f:\R^d\to \CT_{<\gamma}$.
\item[4.]\textbf{The transparent case.}
We let $\CX_\eps=\CD'(\R^d)$ with $\iota_\eps$ given by the identity. Let $\lambda\in (0,\eps]$ and $z\in\K_\eps$, we write $\Phi_{\eps,z}^{\lambda}$ for the set of all $\varphi\in\Phi$ such that $[S_{\s,z}^{\lambda}\varphi]\subset \K_\eps$. We endow $\CX_\eps$ with 
\begin{equ}
\|f\|_{\alpha;\K_\eps;z;\eps} = \sup_{\lambda \in (0,\eps]}\sup_{\phi \in \Phi_{\eps,z}^{\lambda}} \lambda^{-\alpha} |f(\phi_z^\lambda)|\;.
\end{equ}
Finally, we define 
\begin{equation}
\$ f\$_{\gamma;\K;\eps}= \sup_{y,z\in\K, \|y-z\|_\s<\eps}\sup_{\beta<\gamma}
\frac{\|f(z)-\Gamma^{\eps}_{zy}f(y)\|_{\beta}}{\|y-z\|_{\s}^{\gamma-\beta}},
\end{equation}
and
\begin{equation}
\$f;\bar f\$_{\gamma;\K;\eps}= \sup_{y,z\in\K, \|y-z\|_\s<\eps}\sup_{\beta<\gamma}
\frac{\|f(z)-\Gamma^{\eps}_{zy}f(y)-\bar f(z)+\bar\Gamma_{zy}^\eps\bar f(y)\|_{\beta}}{\|y-z\|_{\s}^{\gamma-\beta}},
\end{equation}
for $\gamma\in\R,\K\Subset \R^d, \eps\in(0,1]$ and $f,\bar{f}:\R^d\to \CT_{<\gamma}$.
\end{itemize}
\begin{remark}
Note that in the first three examples the extended seminorms  $\|\cdot\|_{\alpha;\K_\eps;z;\eps}$ are always finite and thus proper seminorms. This is not true anymore in the last example. The fact that in the former examples the index $z$ does not appear in the respective definitions is not a typo. Thus when considering a concrete equation that falls into the framework of one of these three examples one may also simply omit this index.   
\end{remark}
\begin{remark}
The transparent case is actually the setting of the original theory of regularity structures developed in~\cite{Regularity}. In particular it is possible to check that the various assumptions stated in this article are satisfied in the original setup.
Another example in which our setup could be applied is the space-discrete $\Phi^4_3$-model studied in~\cite{MatetskiDiscrete}, which falls into the semidiscrete case.
In the forthcoming article~\cite{PamExclusion} we investigate the  parabolic Anderson model in discrete space, where the noise term is given by the fluctuation field of the simple symmetric exclusion process, and we let the mesh of the space tend to zero. The framework of the present article allows to study that problem (it indeed served as the original motivation for this article).
\end{remark}
We denote by $\CL(\CT,\CX_\eps)$ the space of all linear maps from $\CT$ to $\CX_\eps$. We then have the following definition.

\begin{definition}
\label{def:model}
A discrete model for a given regularity structure $\TT=(A,\CT,G)$ consists of a collection of maps $z \mapsto \Pi^{\eps}_z\in \CL(\CT, \CX_\eps)$ and $(x,y) \mapsto \Gamma_{xy}^\eps \in G$ such that
\begin{claim}
\item $\Gamma^\eps_{zz}= \id$, the identity operator, and $\Gamma^\eps_{xy}\Gamma^\eps_{yz}=\Gamma^\eps_{xz}$ for $x,y,z\in\R^d$.
\item One has $\Pi_z^\eps= \Pi_y^\eps\Gamma^\eps_{yz}$ for all $y,z\in\R^d$.
\end{claim}
Furthermore, for any compact set $\K\subset\R^d$ and every $\gamma>0$ one has the analytical estimates
\begin{equ}
\label{eq:Pi}
\bigl|\bigl(\iota_\eps\Pi^{\eps}_z \tau\bigr)(\CS_{\s,z}^{\lambda}\phi)\bigr| \lesssim \|\tau\|\lambda^{|\tau|}\;,\quad 
\|\Pi^{\eps}_z \tau\|_{|\tau|;\K_\eps;z;\eps} \lesssim \|\tau\|\;,
\end{equ} and, setting $f_z^{\tau, \Gamma^\eps}(y) = \Gamma^{\eps}_{yz}\tau-\tau$,
\begin{equation}
\label{eq:Gamma}
\|\Gamma^\eps_{zz'}\tau\|_m\lesssim \|\tau\|\|z-z'\|_\s^{|\tau|-m}\;,\quad
\$ f_z^{\tau, \Gamma^\eps}\$_{|\tau|;\K; \eps}\lesssim \|\tau\|,
\end{equation}
uniformly over $\lambda \in (\eps,1]$, test functions $\phi\in\Phi$, all homogeneous $\tau\in \CT$ with
$|\tau| < \gamma$, all $m < |\tau|$, and uniformly over $z,z'\in\K$ such that $\|z-z'\|_\s\in (\eps,1]$.
In the second bound in~\eqref{eq:Pi} we moreover considered compact sets $\K_\eps$ of diameter bounded by $2\eps$.
\end{definition}
\begin{remark}
In the above definition, the proportionality constants are allowed to depend on $\eps$. In practice however one is interested 
in obtaining convergence to a limit as $\eps$ tends to zero, and this usually requires them to be independent of $\eps$
in order to obtain useful statements. 
The same remark applies to all estimates stated in this article. In particular if in any of the assumptions stated in this 
article the corresponding proportionality constants do not depend on $\eps$, then the same is true for any derived estimate.	
\end{remark}

\begin{remark}
For a compact set $\K\subset \R^d$ we let $\|\Pi^\eps\|_{\gamma;\K}^{(\eps)}$ and $\|\Gamma^\eps\|_{\gamma;\K}^{(\eps)}$ be the smallest proportionality constants such that~\eqref{eq:Pi} and~\eqref{eq:Gamma} hold respectively.
Let $Z^\eps=(\Pi^\eps,\Gamma^\eps)$ be a model, we set $\$Z^\eps\$_{\gamma,\K}^{(\eps)}=\|\Pi^\eps\|_{\gamma;\K}^{(\eps)}+ 
\|\Gamma^\eps\|_{\gamma;\K}^{(\eps)}$ and for a second model $\bar{Z}^\eps=(\bar{\Pi}^\eps,\bar{\Gamma}^\eps)$ we denote
$\$Z^\eps;\bar{Z}^\eps\$_{\gamma;\K}^{(\eps)}= \|\Pi^\eps-\bar{\Pi}^\eps\|_{\gamma;\K}^{(\eps)}+ 
\|\Gamma^\eps;\bar{\Gamma}^\eps\|_{\gamma;\K}^{(\eps)}$. Here, $\|\Gamma^\eps;\bar{\Gamma}^\eps\|_{\gamma;\K}^{(\eps)}$ denotes the sum of the suprema of $\|(\Gamma_{zz'}^{\eps}-\bar{\Gamma}_{zz'}^{\eps})\tau\|_m/\|z-z'\|_\s^{|\tau|-m}$
and $\$f_z^{\tau, \Gamma^\eps} ; f_z^{\tau, \bar \Gamma^\eps}\$_{|\tau|;\K; \eps}$ over the same set as in Definition~\ref{def:model}.
\end{remark}
\begin{remark}

\label{rem:comparecontdiscr}
In~\cite{Regularity} the definition of a model required $\Pi^\eps$ to map into the space of distributions $\CD'(\R^d)$ 
(this means $\CX_\eps=\CD'(\R^d)$ in the current framework). Moreover, only the first inequalities in~\eqref{eq:Pi} 
and~\eqref{eq:Gamma} were imposed. They were however imposed for all $\lambda\in (0,1]$. We refer to such models as 
``continuous models'' and we denote them by $(\Pi,\Gamma)$. The ``transparent case'' above shows that the notion of 
``discrete model'' is a (strict) generalisation of that of a continuous model. A natural way to compare a 
continuous model to a discrete model $(\Pi^\eps,\Gamma^\eps)$ is via the quantities
\begin{equation}
\label{eq:comparecontdisc}
\begin{aligned}
&\|\Pi-\Pi^\eps\|_{\gamma;\K;\geq \eps}=\sup_{\lambda\in (\eps,1]}\sup_{\substack{\ell\in A\\ \ell <\gamma}}\sup_{\tau\in\CT_\ell}\sup_{\varphi \in\Phi}\sup_{z\in\K}\lambda^{-|\tau|}\bigl|\bigl(\Pi_z\tau- \iota_\eps\Pi^{\eps}_z\tau\bigr)(\CS_{\s,z}^{\lambda}\phi)\bigr|\;,\quad\mbox{and}\\
&\|\Gamma-\Gamma^{\eps}\|_{\gamma;\K;\geq \eps}= 
\sup_{\substack{\ell\in A\\ \ell <\gamma}}\sup_{m<\ell}\sup_{\tau\in\CT_\ell}\sup_{\substack{z,z'\in\K\\ \|z-z'\|_\s\in (\eps,1]}}\|z-z'\|_\s^{m-\ell}\|\Gamma_{zz'}\tau-\Gamma_{zz'}^\eps\tau\|_m.
\end{aligned}
\end{equation}
\end{remark}
With Remark~\ref{rem:comparecontdiscr} at hand we have the following definition.
\begin{definition}
\label{def:comparecontdics}
Let $Z=(\Pi,\Gamma)$ be a continuous model and let $Z^\eps=(\Pi^\eps,\Gamma^\eps)$ be a discrete model. 
For every compact set $\K\subset\R^d$, we define a distance between $\Pi$ and $\Pi^\eps$ via
\begin{equation}
\| \Pi;\Pi^\eps\|_{\gamma;\K}= \|\Pi-\Pi^\eps\|_{\gamma;\K;\geq \eps}+ \sup_{\ell, z, \tau,\K_\eps} \|\Pi_z^\eps\tau\|_{|\tau|;\K_\eps;z;\eps}+
\sup_{\lambda,\varphi,\ell,\tau,z} \lambda^{-|\tau|}\bigl|\bigl(\Pi_z\tau\bigr)(\CS_{\s,z}^{\lambda}\phi)\bigr|,
\end{equation}
and between $\Gamma$ and $\Gamma^\eps$ via
\begin{equation}
\|\Gamma;\Gamma^\eps\|_{\gamma;\K}= \|\Gamma-\Gamma^{\eps}\|_{\gamma;\K;\geq\eps} + \sup_{z,\tau} \$ f_z^{\tau, \Gamma^\eps}\$_{|\tau|;\K; \eps}+ \sup_{m,z,z',\tau} \|z-z'\|_\s^{m-|\tau|}\|\Gamma_{zz'}\tau \|_{m}.
\end{equation}
Here, when referring to the discrete model $Z^\eps$ the supremum is taken over the same set of indices as in Definition~\ref{def:model}, whereas when referring to the model $Z$ the supremum is additionally taken over $\lambda\in(0,\eps]$ and $\|z-z'\|_\s\leq \eps$, respectively.
We finally define the distance between $Z$ and $Z^\eps$ by
\begin{equation}
\$Z;Z^\eps\$_{\gamma;\K}= \|\Pi;\Pi^\eps\|_{\gamma;\K}+ \|\Gamma;\Gamma^\eps\|_{\gamma;\K}.
\end{equation}
\end{definition}

We have the following definition.
\begin{definition}
\label{def:Dgamma}
Let $\gamma\in\R$ and fix a discrete model $(\Pi^{\eps},\Gamma^{\eps})$. The space $\CD_\eps^\gamma$ consists of all maps $f:\R^d\to \CT_{<\gamma}$ such that for every compact set $\K\subset\R^d$ one has
\begin{equation}
\label{eq:Dgamma}
\$ f \$^{(\eps)}_{\gamma;\K} \overset{\text{def}}{=}  
\sup_{\substack{(y,z)\in \K\\
\eps \leq \|y-z\|_\s\leq 1}}\sup_{\beta<\gamma} \frac{\|f(z)-\Gamma^{\eps}_{zy}f(y)\|_{\beta}}{\|y-z\|_\s^{\gamma-\beta}}
+\$ f\$_{\gamma;\K;\eps} <\infty.
\end{equation}
We say that $f\in\CD_\eps^{\gamma}$ is a modelled distribution.
\end{definition}
\begin{remark}
Note that we could have equipped the $\CD_\eps^{\gamma}$-spaces with a norm that involves the expression
\begin{equation}
\sup_{z\in\K}\sup_{\beta <\gamma}\|f(z)\|_{\beta},
\end{equation}
as was done in~\cite{Regularity}.
However, since this term will not be part of any bound that we state in this article, we decided to omit this term in Definition~\ref{def:Dgamma}.
\end{remark}
\begin{remark}
In some situations we only consider elements of $\CD^\gamma_\eps$ taking values in some given sector $V$ of $\CT$. In this case we also use the notation $\CD^\gamma_\eps(V)$. In cases where $V$ is of regularity $\alpha$ we also write $\CD_{\alpha,\eps}^{\gamma}$. Occasionally we want to emphasize the dependence on a given model $Z^\eps=(\Pi^\eps,\Gamma^\eps)$, and we use the notation $\CD^\gamma_\eps(Z^\eps)$ or $\CD^\gamma_\eps(\Gamma^\eps)$ to do so.
\end{remark}
\begin{remark}
Given two models $(\Pi^\eps,\Gamma^\eps)$, $(\bar{\Pi}^\eps,\bar{\Gamma}^\eps)$ and two functions $f\in\CD^\gamma_\eps(\Gamma^\eps)$ and $\bar{f}\in\CD^\gamma_\eps(\bar{\Gamma}^\eps)$ we define a distance between $f$ and $\bar{f}$ via
\begin{equation}
\$f;\bar{f}\$_{\gamma;\K}^{(\eps)}= 
 \sup_{\substack{(y,z)\in \K\\
\eps \leq \|y-z\|_\s\leq 1}}\sup_{\beta<\gamma} \frac{\|f(z)-\bar{f}(z)-\Gamma^{\eps}_{zy}f(y)+\bar{\Gamma}^\eps_{zy}\bar{f}(y)\|_{\beta}}{\|y-z\|_\s^{\gamma-\beta}}
+\$f;\bar{f}\$_{\gamma;\K;\eps}.
\end{equation}
\end{remark}
\begin{remark}
\label{rem:Dgammacont}
Fix a continuous model $(\Pi,\Gamma)$. One can define the space of functions $\CD^\gamma(\Gamma)$ as all maps $f:\R^d\to\CT_{<\gamma}$ such that
\begin{equation}
\label{eq:Dgammacont}
\$ f\$_{\gamma;\K}:= \sup_{\substack{(y,z)\in \K\\
\|y-z\|_\s\leq 1}}\sup_{\beta<\gamma} \frac{\|f(z)-\Gamma_{zy}f(y)\|_{\beta}}{\|y-z\|_\s^{\gamma-\beta}} <\infty.
\end{equation}
Note that this definition differs from~\cite[Def.~3.1]{Regularity}. However, the finiteness of the norm in~\eqref{eq:Dgammacont} implies the finiteness of the corresponding norm in~\cite{Regularity}, so that the current setup indeed agrees with the one in~\cite{Regularity}. It is also consistent with \eqref{eq:Dgamma} in the sense that the
two expressions agree in the transparent case.
\end{remark}
\begin{definition}
\label{def:Dgammacontdisc}
Given a continuous model $(\Pi,\Gamma)$ and a discrete model $(\Pi^\eps,\Gamma^\eps)$ and two functions $f\in\CD^{\gamma}(\Gamma)$ and $f^\eps\in \CD_\eps^{\gamma}(\Gamma^\eps)$, we define the distance between $f$ and $f^\eps$ via
\begin{equation}
\label{eq:Dgammacontdisc}
\begin{aligned}
\$ f;f^\eps\$_{\gamma;\K}=
 &\sup_{\substack{(y,z)\in \K\\
\|y-z\|_\s < \eps}}\sup_{\beta<\gamma} \frac{\|f(z)-\Gamma_{zy}f(y)\|_{\beta}}{\|y-z\|_\s^{\gamma-\beta}}
+ \$ f^\eps \$_{\gamma;\K;\eps}\\
&+ \sup_{\substack{(y,z)\in \K\\
\eps\leq \|y-z\|_\s\leq 1}}\sup_{\beta<\gamma} \frac{\|f(z)-f^\eps(z)-\Gamma_{zy}f(y)+\Gamma_{zy}^\eps f^\eps(y)\|_{\beta}}{\|y-z\|_\s^{\gamma-\beta}}.
\end{aligned}
\end{equation}
In plain words, at scales larger than $\eps$ we compare $f$ and $f^\eps$ in the natural way, and at scales smaller than $\eps$ we simply add the bits describing $f$ and $f^\eps$ at these scales.
\end{definition}

\noindent
\textbf{Convention:} For the rest of this article the symbol $\K_\eps$ is reserved for a compact subset of $\R^d$ with diameter bounded by $2\eps$.


\section{The reconstruction theorem}
\label{S3}

In this section we prove that given an element $f$ in $\CD_\eps^{\gamma}$ and an operator $\CR^\eps:\CD_\eps^{\gamma}\to\CX_\eps$ such that $\CR^\eps f$ is close to $\Pi_z^\eps f(z)$ on a local scale, then they are also close globally.  
The significance of it is that $\Pi_z^\eps f(z)$ is a local object, whereas $\CR^\eps f$ can be thought of as a global object.
The following two assumptions are key for this and they are assumed to hold throughout this article.
\begin{assumption}
\label{a:rec}
Let $\gamma>0$ and fix a discrete model $(\Pi^\eps,\Gamma^\eps)$. We assume that there is a linear map $\cR^{\eps}:\CD_\eps^\gamma(\Gamma^\eps)\to \CX_\eps$ such that for every $z\in\R^d$ and every compact set $\K_\eps\subset \R^d$ of diameter at most $2\eps$ containing $z$,
\begin{equation}
\label{eq:rec}
\|\cR^{\eps} f - \Pi^{\eps}_z f(z)\|_{\gamma;\K_\eps;z;\eps}\lesssim
\|\Pi^\eps\|^{(\eps)}_{\gamma;\bar\K_\eps} \$ f\$_{\gamma;\K_\eps;\eps},
\end{equation} 
locally uniformly over all $z$, over all such compact sets $\K_\eps$, and $\eps>0$. 
Any map $\CR^\eps$ satisfying~\eqref{eq:rec} is called a reconstruction operator.
\end{assumption}

\begin{remark}
Recall that the norm $\$\cdot\$_{\gamma;\K_\eps;\eps}$ is allowed to depend on a neighbourhood of size $\bfc\eps$ of $\K_\eps$ for some fixed value of $\bfc$. In practice it is often~\eqref{eq:rec} that determines the choice of $\bfc$, i.e., one first chooses a candidate for $\CR^\eps$ and then one verifies for which choice of $\bfc$, given a reasonable candidate for $\$\cdot\$_{\gamma;\K_\eps;\eps}$, a bound of the type~\eqref{eq:rec} is satisfied.
\end{remark}
\begin{remark}\label{rem:rec}
	A common choice for $\CR^\eps$ is given by $\CR^\eps f(z)= (\Pi_z^\eps f(z))(z)$, provided of course that this is a 
	meaningful expression, which is the case in the first three examples given in the previous section.
	To verify Assumption~\ref{a:rec} in the purely discrete case, fix $z\in\Lambda_\eps^d$ and a compact set $\K_\eps$ as required above. Then for every $y\in\K_\eps$, we can write
	\begin{equation}
	\begin{aligned}
	\eps^{-\gamma}|(\Pi_y^\eps f(y))(y)-(\Pi_z^\eps f(z))(z)|
	&=\eps^{-\gamma}|\Pi_y^\eps[f(y)-\Gamma_{yz}^\eps f(z)](y)|\\
	&\leq\sum_{\ell <\gamma}\eps^{-\gamma}|\Pi_y^\eps \CQ_\ell [f(y)-\Gamma_{yz}^\eps f(z)](y)|\\
	&\lesssim \|\Pi^\eps\|^{(\eps)}_{\gamma;\bar\K_\eps}
	\sum_{\ell <\gamma}\eps^{-\gamma +\ell}\|f(y)-\Gamma_{yz}^\eps f(z)\|_\ell,
	\end{aligned}
	\end{equation}
	and, since there are only finitely many terms of homogeneity smaller than $\gamma$,~\eqref{eq:rec} readily follows from the definition of $\$\cdot\$_{\gamma;\K_\eps;\eps}$.
	Essentially the same computation can be done in the semidiscrete and continuous case.
	In the transparent case, it is a consequence of~\cite[Theorem 3.10]{Regularity} that the reconstruction operator satisfies Assumption~\ref{a:rec}.
\end{remark}

When comparing reconstruction operators corresponding to different models, we also need to make the following assumption.
\begin{assumption}
\label{a:reccomparedisc}
Fix two discrete models $(\Pi^\eps,\Gamma^\eps)$ and $(\bar\Pi^\eps,\bar\Gamma^\eps)$ with associated reconstruction operators $\CR^\eps$ and $\bar\CR^\eps$, and let $f\in\CD_\eps^{\gamma}(\Gamma^\eps)$ and $\bar f\in\CD_\eps^{\gamma}(\bar\Gamma^\eps)$. We assume that for every $z\in\R^d$ and every compact set $\K_\eps$ of diameter at most $2\eps$ containing $z$,
\begin{equation}
\label{eq:reccomparedisc}
\begin{aligned}
\|\cR^{\eps} f - \Pi^{\eps}_z f(z)-&\bar\CR^\eps\bar f(z) + \bar\Pi^{\eps}_z \bar f(z)\|_{\gamma;\K_\eps;z;\eps}\\
&\lesssim
\|\bar \Pi^\eps\|^{(\eps)}_{\gamma;\bar\K_\eps} \$ f;\bar f\$_{\gamma;\K_\eps;\eps} + \|\Pi^\eps-\bar\Pi^\eps\|_{\gamma;\bar\K_\eps}^{(\eps)}\$ f\$_{\gamma;\K_\eps;\eps}
\end{aligned}
\end{equation} 
locally uniformly over all $z$, over all such compact sets $\K_\eps$, and $\eps>0$. 
Fix a continuous model $(\Pi,\Gamma)$ with associated reconstruction operator $\CR$, $f\in\CD^{\gamma}(\Gamma)$, and $f^\eps\in\CD_\eps^{\gamma}(\Gamma^\eps)$. We assume that for every $\eta\in \Phi$, 
\begin{equation}
\begin{aligned}
\big| [\iota_\eps(\CR^\eps f^\eps -&\Pi_z^\eps f^\eps (z))- (\CR f - \Pi_z f(z))](\eta_z^{\eps})\big |\\
&\lesssim \eps^{\gamma}\big[\|\Pi^\eps\|_{\gamma;\overline{[\eta_z^{\eps}]}}^{(\eps)}\$f;f^\eps\$_{\gamma;[\eta_z^{\eps}]}
+ \|\Pi;\Pi^\eps\|_{\gamma;\overline{[\eta_z^{\eps}]}}\$f\$_{\gamma;[\eta_z^{\eps}]}\big],
\end{aligned}
\end{equation}
 locally uniformly in $z\in \R^d$, and uniformly over all $\eps$.
\end{assumption}
\begin{remark}
Given two discrete models as above, the form of the reconstruction operator alluded to in Remark~\ref{rem:rec} shows that $\CR^\eps$ is usually a bilinear function of the pair $(f,\Pi^\eps)$. Therefore, the validity of the Assumption~\ref{a:reccomparedisc} can be shown in a similar way as the one for Assumption~\ref{a:rec}.  
\end{remark}

\begin{theorem}
\label{thm:reconstruction}
Let $\gamma>0$, and fix a compact set $\K$. Fix a discrete model $(\Pi^\eps,\Gamma^\eps)$ such that Assumption~\ref{a:rec} is satisfied. We then have the estimate
\begin{equation}
\label{eq:reconstruction}
|\iota_\eps(\cR^{\eps} f-\Pi^{\eps}_zf(z))(\eta_z^{\delta})|
\lesssim \delta^\gamma \|\Pi^{\eps}\|^{(\eps)}_{\gamma;\bar\K}\$ f\$^{(\eps)}_{\gamma;[\eta_z^{\delta}]},
\end{equation}
uniformly over all test function $\eta\in \Phi$, all $\delta\in(\eps,1]$, all $f\in \CD_\eps^\gamma$, all $z\in\K$, and all $\eps\in (0,1]$.
Given a second discrete model $(\bar\Pi^{\eps},\bar\Gamma^{\eps})$ such that additionally Assumption~\ref{a:reccomparedisc} is satisfied, then for all $f\in\CD_\eps^{\gamma}(\Gamma^\eps)$ and $\bar f\in\CD_\eps^{\gamma}(\bar\Gamma^\eps)$,
\begin{equation}
\label{eq:differencelargescale}
\begin{aligned}
|\iota_\eps&(\cR^{\eps}f-\bar\cR^{\eps}\bar f-\Pi_z^{\eps}f(z)+\bar\Pi_z^{\eps}\bar f(z))(\eta_z^{\delta})|\\
&\lesssim \delta^{\gamma}(\|\bar\Pi^{\eps}\|^{(\eps)}_{\gamma;\bar\K}\$f;\bar f\$^{(\eps)}_{\gamma;[\eta_z^{\delta}]}
+\|\Pi^{\eps}-\bar\Pi^{\eps}\|^{(\eps)}_{\gamma;\bar\K}\$f\$^{(\eps)}_{\gamma;[\eta_z^{\delta}]}),
\end{aligned}
\end{equation}
uniformly over all parameters as above. Finally, if $\cR$ is the reconstruction operator corresponding to a continuous model $(\Pi,\Gamma)$ such that Assumption~\ref{a:reccomparedisc} holds, $f\in\CD^\gamma(\Gamma)$, and $f^\eps\in\CD_\eps^{\gamma}(\Gamma^\eps)$, then 
\begin{equation}
\label{eq:reconstrcontdiscrete}
\begin{aligned}
|(\iota_\eps&(\cR^{\eps}f^\eps-\Pi_z^\eps f^\eps(z))- \cR f+\Pi_zf (z))(\eta_z^{\delta})|\\
&\lesssim \delta^{\gamma}(\|\Pi^{(\eps)}\|_{\gamma;\bar\K}^{(\eps)}\$f^\eps; f\$_{\gamma;[\eta_z^{\delta}]}
+\|\Pi^{\eps};\Pi\|_{\gamma;\bar\K}\$f\$^{(\eps)}_{\gamma;[\eta_z^{\delta}]}),
\end{aligned}
\end{equation}
uniformly over all parameters as above.
\end{theorem}
Before we give the proof we need to introduce more notation. For $n\in\N$ we define the scaled lattice
\begin{equation}
\label{eq:lambda}
\Lambda_n^\s=\bigg\{\sum_{j=1}^{d}2^{-n\s_j}k_je_j:\, k_j\in\Z\bigg\},
\end{equation}
where we denote by $e_1,\ldots, e_d$ the canonical basis of $\R^d$.
\begin{proof}
We fix a smooth function $\Psi:\R\to[0,1]$ such that $\mathrm{supp}\,\Psi\subset [-1,1]$ and such that additionally
for every $z\in\R$,
\begin{equation}
\label{eq:summingtoone}
\sum_{k\in\Z}\Psi(z+k)=1.
\end{equation}
Fix $k\in\N$ and let $z_k\in\Lambda_{k}^\s, z_{k+1}\in\Lambda_{k+1}^{\s}$.
We define rescaled versions of $\Psi$ via
\begin{equation}
\label{eq:scalephi}
\Psi_{[z_{k}]}(y)=\prod_{i=1}^d \Psi(2^{k\s_i}(y_i-z_{k,i})),\quad\mbox{and}\quad
\Psi_{[z_{k},z_{k+1}]}(y)=\Psi_{[z_{k}]}(y)\Psi_{[z_{k+1}]}(y),
\end{equation}
where $y_i$ and $z_{k,i}$ denote the $i$-th coordinate of $y$ and $z_k$, respectively. 
Note the simple but useful identity
\begin{equation}
\label{eq:scalephiidentities}
\sum_{z_k\in\Lambda_k^{\s}}\Psi_{[z_{k}]}=1,
\end{equation}
which immediately follows from~\eqref{eq:summingtoone}.
We may now start with the core part of the proof. To that end, let $n_0$ be the smallest integer such that
$2^{-n_0}\leq \delta$ and define $\tilde\Psi_{z,[z_k]}^{\delta,k} = \eta_z^{\delta}\,\Psi_{[z_k]}$,
as well as $\tilde\Psi_{z,[z_k,z_{k+1}]}^{\delta,k} = \eta_z^{\delta}\,\Psi_{[z_k,z_{k+1}]}$ .
Note that for each $k\geq n_0$ the support of  $\tilde\Psi_{z,[z_k]}^{\delta,k}$ is contained in a ball of radius $2^{-k}$ with center $z_k$, and a similar statement holds for $\tilde\Psi_{z,[z_k,z_{k+1}]}^{\delta,k}$.
To continue, for each $z_k\in\Lambda_k^\s$ such that $d_\s(z_k,[\eta_z^{\delta}])< 2^{-k}$ we let $z_{|k}$ be an arbitrary chosen element in $[\eta_z^{\delta}]$ such that $d_\s(z_k,z_{|k}) <2^{-k}$.
With all this at hand we see that for every $N > n_0$ we can write
\begin{equation}
\label{eq:rec123}
\iota_\eps(\cR^{\eps} f-\Pi_z^\eps f(z))(\eta_z^{\delta})= \I + \II + \III,
\end{equation}
where
\begin{equs}
\label{eq:def123}
\I &= \sum_{z_N\in\Lambda_N^{\s}}\iota_\eps(\cR^{\eps} f-\Pi_{z_{|N}}^\eps f(z_{|N}))(\tilde\Psi_{z,[z_N]}^{\delta,N}),\\
\II &= \sum_{k=n_0}^{N-1}\sum_{z_{k}\in\Lambda_{k}^{\s}, z_{k+1}\in\Lambda_{k+1}^{\s}}\iota_\eps(\Pi_{z_{|k+1}}^\eps f(z_{|k+1})-\Pi_{z_{|k}}^\eps f(z_{|k}))(\tilde\Psi_{z,[z_k,z_{k+1}]}^{\delta,k}),\\
\III &= \sum_{z_{n_0}\in\Lambda_{n_0}^{\s}}\iota_\eps(\Pi_{z_{|n_0}}^\eps f(z_{|n_0})-\Pi_z^\eps f(z))(\tilde\Psi_{z,[z_{n_0}]}^{\delta,n_0}),
\end{equs}
where we made use of the identities~\eqref{eq:scalephiidentities} (recall at this point that $z_{|k}$ is really a function of $z_k$).
We now choose $N$ to be the smallest value such that $2^{-N} \le \eps$, so that 
the support of $\tilde\Psi_{z,[z_N]}^{\delta,N}$ is a compact set of diameter at most $2\eps$.
We also remark that for any $k \in [n_0,N]$ there exists an element $\psi \in \Phi$ and a constant $c_k$ that is uniformly bounded in $k$ such that
\begin{equ}
(\delta 2^k)^{|\s|} \tilde{\Psi}^{\delta,k}_{z,[z_k,z_{k+1}]} = c_k\CS_{z_{|k+1}}^{2^{-k-1}} \psi\;,
\end{equ}
and similarly for $\tilde{\Psi}^{\delta,k}_{z,[z_k]}$. Here, the constants $c_k$ appear since the various derivates of $(\delta 2^k)^{|\s|} \tilde{\Psi}^{\delta,k}_{z,[z_k,z_{k+1}]}$ up to order $r$ may not satisfy the estimates that rescaled elements of $\Phi$ necessarily do satisfy. We can therefore estimate
\begin{equation}
\label{eq:Iest}
\begin{aligned}
\big|\iota_\eps(\cR^{\eps} f-\Pi_{z_{|N}}^\eps f(z_{|N}))(\tilde\Psi_{z,[z_N]}^{\delta})\big|
&\lesssim (\delta 2^{N})^{-|\s|}2^{-N\gamma}\| \cR^{\eps} f-\Pi_{z_{|N}}^\eps f(z_{|N})\|_{\gamma;[\tilde\Psi_{z,[z_N]}^{\delta}];z_{|N};\eps}\\
&\lesssim (\delta 2^{N})^{-|\s|}2^{-N\gamma}\|\Pi^\eps\|_{\gamma,\bar\K}^{(\eps)}\$ f\$_{\gamma;[\tilde\Psi_{z,[z_N]}^{\delta}];\eps}.
\end{aligned}
\end{equation}
Thus, the desired estimate \eqref{eq:reconstruction} on $\I$ follows from the fact that, since $N > n_0$, one has 
$\tilde\Psi_{z,[z_N]}^{\delta,N}=0$ if 
$\delta\lesssim \| z_N-z\|_{\s}$, combined with the bound $|\{z_N\in\Lambda_N^{\s}:\, \|z_N-z\|_{\s}\leq \delta\}|\lesssim (2^N\delta)^{|\s|}$. Here, the proportionality constants only depend on the dimension.
To deal with $\II$, note that $\tilde{\Psi}^{\delta,k}_{z,[z_k,z_{k+1}]}=0$, unless $d_\s(z_k,z_{k+1})< 2^{-k+1}$.
Recall that $\CQ_\ell$ denotes the projection onto the subspace of degree $\ell$ in our regularity structure, we
can therefore estimate for each $k\in[n_0,\ldots, N-1]$ and each pair $z_k, z_{k+1}$ contributing to the sum in $\II$, 
\begin{equation}
\label{eq:IIest}
\begin{aligned}
\big|\iota_\eps(\Pi_{z_{|k+1}}^\eps &f(z_{|k+1})-\Pi_{z_{|k}}^\eps f(z_{|k}))(\tilde\Psi_{z,[z_k,z_{k+1}]}^{\delta,k})\big|\\
&= \big|\iota_\eps\big(\Pi_{z_{|k+1}}^\eps [f(z_{|k+1})-\Gamma^{\eps}_{z_{|k+1}z_{|k}}f(z_{|k})]\big)(\tilde\Psi_{z,[z_k,z_{k+1}]}^{\delta,k})\big|\\
&\leq \sum_{\ell<\gamma}
\big|\iota_\eps\big(\Pi_{z_{|k+1}}^\eps \CQ_{\ell}[f(z_{|k+1})-\Gamma^{\eps}_{z_{|k+1}z_{|k}}f(z_{|k})]\big)(\tilde\Psi_{z,[z_k,z_{k+1}]}^{\delta,k})\big|\\
&\leq (\delta 2^{k})^{-|\s|} \|\Pi^\eps\|_{\gamma,\bar\K}^{(\eps)}\sum_{\ell <\gamma}
2^{-(k+1)\ell}\|f(z_{|k+1})-\Gamma^{\eps}_{z_{|k+1}z_{|k}}f(z_{|k})\|_{\ell}\\
&\lesssim (\delta 2^{k})^{-|\s|} \|\Pi^{\eps}\|_{\gamma,\bar\K}^{(\eps)}\$f\$_{\gamma,[\eta_z^{\delta}]}^{(\eps)}2^{-k\gamma}\;,
\end{aligned}
\end{equation}
where we exploited the fact that the distance between $z_{|k}$ and $z_{|k+1}$ is at most of
order $2^{-k}$.
To conclude the estimate of $\II$ we note that for each $z_k\in\Lambda_k^{\s}$, the number of values for $z_{k+1}$ that contribute to the corresponding sum in~\eqref{eq:def123} is at most $5^d$ and the number of $z_k$ such that $\tilde{\Psi}_{z,[z_k,z_{k+1}]}^{\delta,k}\neq 0$, i.e., $\|z_k-z\|_{\s}\lesssim (\delta+2^{-k})$ is of the order $(\delta2^{k})^{|\s|}$. 
Thus, we see that summing the right hand side of~\eqref{eq:IIest} over $k\in[n_0,N-1]$ yields the desired estimate on $\II$.
The estimate on $\III$ is similar and we therefore omit the details.
To prove~\eqref{eq:differencelargescale} we first write
$\iota_\eps(\cR^{\eps}f-\bar\cR^{\eps}\bar f- \Pi_z^{\eps}f(z)+\bar\Pi_z^{\eps}\bar f(z))(\eta_z^{\delta})$ as a sum of three terms $\I'$, $\II'$ and $\III'$ where,
\begin{equation}
\label{eq:def1}
\begin{aligned}
\I' = \sum_{z_N\in\Lambda_N^{\s}}\iota_\eps(\cR^{\eps} f-\bar\cR^{\eps}\bar f- \Pi^{\eps}_{z_{|N}}f(z_{|N})+\bar\Pi_{z_{|N}}^{\eps}\bar f(z_{|N}))(\tilde\Psi_{z,[z_N]}^{\delta,N}),
\end{aligned}
\end{equation}
and $\II'$ and $\III'$ are the ``telescopic sum'' terms analogous to $\II$ and $\III$ in~\eqref{eq:def123}.
To estimate $\I'$ we may use Assumption~\ref{a:reccomparedisc} and we can proceed in exactly the same way as we did to estimate $\I$.
To bound $\II'$ and $\III'$ we make use of the identity 
\begin{equation}
\label{eq:pidifference}
\begin{aligned}
\Pi_z^{\eps}f(z)&-\bar\Pi^{\eps}_z\bar f(z)-\Pi^{\eps}_yf(y)+\bar\Pi^{\eps}_y\bar f(y)\\
&= \bar\Pi^{\eps}_z(f(z)- \Gamma^{\eps}_{zy} f(y)-\bar f(z) +\bar \Gamma^{\eps}_{zy}\bar f(y))
+(\Pi^{\eps}_z-\bar\Pi^{\eps}_z)(f(z)-\Gamma^{\eps}_{zy} f(y)).
\end{aligned}
\end{equation}
From that point on, one proceeds in the same way as for the bounds on $\II$ and $\III$.
The bound~\eqref{eq:reconstrcontdiscrete} is then obtained in virtually the same way as~\eqref{eq:differencelargescale}.
\end{proof}

\subsection{The reconstruction operator in weighted $\CD_\eps^\gamma$-spaces}
\label{S3.1}
In~\cite{Regularity}, versions of the $\CD^\gamma$-spaces were introduced that allow for singularities on a hyperplane $P$.
In the present context this translates to spaces $\CD_\eps^{\gamma,\eta}$ that have an additional dependence on $\eps$
and that generalise the corresponding spaces $\CD^{\gamma,\eta}$ in~\cite{Regularity}.
Fix $\bar d\in [1,d)$, let $P$ be the hyperplane given by
\begin{equation}
\label{eq:hyperplane}
P=\{z\in\R^d:\, z_i=0,\, i=1,\ldots, \bar d\}\;,
\end{equation}
and denote by $\mathfrak{m}=\s_1+\cdots+\s_{\bar d}$ the (effective) codimension of $P$.
We further let 
\begin{equation}
\label{eq:Pnorm}
\|z\|_{P}= 1\wedge d_\s(z,P), \quad\mbox{and}\quad \|y,z\|_{P}= \|y\|_{P}\wedge \|z\|_{P},
\end{equation}
as well as
\begin{equation}
\label{eq:kp}
\K_P=\{(y,z)\in (\K\setminus P)^2:\, y\neq z,\, \|y-z\|_\s\leq \|y,z\|_{P}\}.
\end{equation}
Our construction relies again on the existence of a family of ``small-scale'' norms 
exhibiting the correct kind of behaviour for $\eps \neq 0$.

\begin{assumption}
\label{ass:weightedDgamma}
We are given two families of extended seminorms $\$\cdot\$_{\gamma,\eta;\K;\eps}$ and $\$\cdot;\cdot\$_{\gamma,\eta;\K;\eps}$ on the space of functions $f:\R^d\to\CT_{<\gamma}$, respectively on pairs $(f,\bar f)$ of such functions, 
with $\gamma,\eta\in\R$, $\K\Subset \R^d$ and $\eps\in(0,1]$.

These are such that, for any two discrete models $(\Pi^\eps,\Gamma^\eps)$ and $(\bar\Pi^\eps,\bar\Gamma^\eps)$, and two functions
$f\in\CD_\eps^{\gamma}(\Gamma^\eps)$ and $\bar f\in\CD_\eps^{\gamma}(\bar\Gamma^\eps)$ one has
\begin{equation}
\label{eq:relationDgammanorms}
\begin{aligned}
\$f\$_{\gamma;\K;\eps}&\lesssim d_\s(z,P)^{\eta-\gamma}\$f\$_{\gamma,\eta;\K;\eps}\;,\\
\$f;\bar f\$_{\gamma;\K;\eps}&\lesssim d_s(z,P)^{\eta-\gamma}\$ f;\bar f\$_{\gamma,\eta;\K;\eps}\;,
\end{aligned}
\end{equation}
for any compact set $\K \ni z$ such that $d_\s(\K,P)\geq \diam(\K)\in (\eps,1]$,
as well as any $\gamma>0$ and $\eta\in\R$.
\end{assumption}
\begin{remark}
\label{rem:weightednorms}
The way to think about $\$\cdot\$_{\gamma, \eta;\K;\eps}$ and $\$\cdot;\cdot\$_{\gamma,\eta;\K;\eps}$ is that they are weighted versions of $\$\cdot\$_{\gamma;\K;\eps}$ and $\$\cdot;\cdot\$_{\gamma;\K;\eps}$, respectively.
A natural choice in the transparent case is given by
\cite[Def.~6.2]{Regularity}, but with the first supremum restricted to 
points $z$ with $\|z\|_P < \eps$ and the second supremum restricted to pairs $y,z$ with $\|y-z\|_\s<\eps$.

Similarly, a possible choice in the purely discrete case is 
\begin{equation}
\$ f\$_{\gamma,\eta;\K;\eps}=\sup_{\substack{z\in\K\setminus P\\ \|z\|_P < \eps}}\sup_{\beta <\gamma}\frac{\|f(z)\|_\beta}{\|z\|_{P,\eps}^{(\eta-\beta)\wedge 0}}
+ \sup_{\substack{(y,z)\in\K_P,\\ \|y-z\|_\s<\eps}}
\sup_{\beta <\gamma}\frac{\|f(z)-\Gamma^{\eps}_{zy} f(y)\|_\beta}{\eps^{\gamma-\beta}\|y,z\|_{P,\eps}^{\eta-\gamma}},
\end{equation}
and
\begin{equation}
\begin{aligned}
\$ f;\bar{f}\$_{\gamma,\eta;\K;\eps}=\sup_{\substack{z\in\K\setminus P\\ \|z\|_P < \eps}}&\sup_{\beta <\gamma}\frac{\|f(z)-\bar f(z)\|_\beta}{\|z\|_{P,\eps}^{(\eta-\beta)\wedge 0}}\\
&+ \sup_{\substack{(y,z)\in\K_P,\\ \|y-z\|_\s<\eps}}
\sup_{\beta <\gamma}\frac{\|f(z)-\Gamma^{\eps}_{zy} f(y)-\bar f(z)+\bar\Gamma_{zy}^\eps \bar f(y)\|_\beta}{\eps^{\gamma-\beta}\|y,z\|_{P,\eps}^{\eta-\gamma}},
\end{aligned}
\end{equation}
where $\|z\|_{P,\eps}=\|z\|_P\vee \eps$ and $\|y,z\|_{P,\eps}= \|y,z\|_P \vee\eps$.
In a similar manner, one may choose the corresponding norms in the continuous and in the semidiscrete case.
\end{remark}
We then have the following definition.
\begin{definition}
\label{def:weightedDgamma}
Fix a regularity structure $\TT$ and a discrete model $(\Pi^{\eps},\Gamma^{\eps})$ and let $\eta\in\R$. We define the space 
$\CD^{\gamma,\eta}_\eps$ as the space of all functions $f:\R^d\setminus P\to\CT_{<\gamma}$ such that $\$ f\$_{\gamma,\eta;\K}^{(\eps)}<\infty$, where $\$ f\$_{\gamma,\eta;\K}^{(\eps)}$ is defined by
\begin{equation}
\label{eq:weightedDgamma}
\sup_{\substack{z\in\K\setminus P\\ \|z\|_P\geq \eps}}\sup_{\beta <\gamma}\frac{\| f(z)\|_{\beta}}{\| z\|_{P}^{(\eta-\beta)\wedge 0}} + \sup_{\substack{(y,z)\in\K_P,\\ \eps\leq \|y-z\|_\s\leq 1}}\sup_{\beta <\gamma}
\frac{\|f(z)-\Gamma^{\eps}_{zy}f(y)\|_{\beta}}{\|y-z\|_\s^{\gamma-\beta}\|y,z\|_{P}^{\eta-\gamma}}
+\$f\$_{\gamma,\eta;\K;\eps}.
\end{equation}
Elements of $\CD_\eps^{\gamma,\eta}$ are called singular modelled distributions.
\end{definition}
\begin{remark}
\label{rem:comparisonweigthed}
Given two discrete models $(\Pi^\eps,\Gamma^\eps)$ and $(\bar\Pi^\eps,\bar\Gamma^\eps)$, and $f\in\CD^{\gamma,\eta}_\eps(\Gamma^\eps)$ as well as $\bar{f}\in\CD^{\gamma,\eta}_\eps(\bar\Gamma^\eps)$, we define
$\$f;\bar f\$_{\gamma,\eta;\K}^{(\eps)}$ via
\begin{equation}
\begin{aligned}
\$f;\bar f\$_{\gamma,\eta;\K}^{(\eps)} =& \sup_{\substack{z\in\K\setminus P\\ \|z\|_P\geq \eps}}\sup_{\beta <\gamma}\frac{\| f(z)-\bar f(z)\|_{\beta}}{\| z\|_{P}^{(\eta-\beta)\wedge 0}} + \$f;\bar f\$_{\gamma,\eta;\K;\eps}\\
&+ \sup_{\substack{(y,z)\in\K_P,\\ \eps\leq \|y-z\|_\s\leq 1}}\sup_{\beta <\gamma}
\frac{\|f(z)-\bar f(z) -\Gamma^{\eps}_{zy}f(y)+\bar\Gamma^{\eps}_{zy}\bar f(y)\|_{\beta}}{\|y-z\|_\s^{\gamma-\beta}\|y,z\|_{P}^{\eta-\gamma}}.
\end{aligned}
\end{equation}
\end{remark}

\begin{remark}
\label{rem:weightedDgammacont}
Given a continuous model $(\Pi,\Gamma)$, if we 
choose $\$\cdot\$_{\gamma,\eta;\K;\eps}$ as in the transparent case of Remark~\ref{rem:weightednorms}, then
Definition~\ref{def:weightedDgamma} coincides with \cite[Def.~6.2]{Regularity}. From now on we assume that in the transparent case $\$\cdot\$_{\gamma,\eta;\K;\eps}$ is given in this way.
\end{remark}
\begin{definition}
\label{def:weightedDgammacontdisc}
Fix a continuous model $(\Pi,\Gamma)$, a discrete model $(\Pi^\eps,\Gamma^\eps)$, and two functions $f\in \CD^{\gamma,\eta}(\Gamma)$ and $f^\eps\in\CD_\eps^{\gamma,\eta}(\Gamma^\eps)$. We define their distance $\$f;f^\eps\$_{\gamma,\eta;\K}$ as the sum 
\begin{equation}
\$f;f^\eps\$_{\gamma,\eta;\K;\geq\eps}+ \$f\$_{\gamma,\eta;\K;\eps} + \$f^\eps\$_{\gamma,\eta;\K;\eps},
\end{equation}
where
\begin{equation}
\label{eq:weightedDgammacontdiscI}
\begin{aligned}
\$f;f^\eps\$_{\gamma,\eta;\K;\geq\eps}= &\sup_{\substack{z\in\K\setminus P\\ \|z\|_P\geq \eps}}\sup_{\beta <\gamma}\frac{\| f(z)-f^\eps (z)\|_{\beta}}{\| z\|_{P}^{(\eta-\beta)\wedge 0}}\\
&+\sup_{\substack{(y,z)\in\K_P,\\ \eps\leq \|y-z\|_\s\leq 1}}\sup_{\beta <\gamma}
\frac{\|f(z)-\Gamma_{zy}f(y)-f^\eps(z)+\Gamma^\eps_{zy} f^\eps(y)\|_{\beta}}{\|y-z\|_\s^{\gamma-\beta}\|y,z\|_{P}^{\eta-\gamma}}.
\end{aligned}
\end{equation}
Here, we remind the reader of the convention made in Remark~\ref{rem:weightedDgammacont}.
\end{definition}

We make the following assumption.
\begin{assumption}
\label{ass:Rsmallscalesweighted}
Let $V$ be a sector of a regularity structure $\TT$ of regularity $\alpha$, let $(\Pi^\eps,\Gamma^\eps)$ and $(\bar\Pi^\eps,\bar\Gamma^\eps)$ be two discrete models with associated reconstruction operators $\CR^\eps$ and $\bar\CR^\eps$, and let $f\in\CD_\eps^{\gamma,\eta}(\Gamma^\eps)$ and $\bar f\in\CD_\eps^{\gamma,\eta}(\bar\Gamma^\eps)$. Let further $f$ and $\bar f$ be such that they take values in $V$. We assume that,
\begin{equation}\label{eq:Rsmallscalesweighted}
\begin{aligned}
\|\CR^\eps f\|_{\alpha\wedge\eta;\K_\eps;z;\eps}&\lesssim \$ f\$_{\gamma,\eta;\K_\eps}^{(\eps)}\;,\\
\|\CR^\eps f-\bar\CR^\eps \bar f\|_{\alpha\wedge\eta;\K_\eps;z;\eps} &\lesssim \|\bar\Pi^\eps\|_{\gamma;\bar \K_\eps}^{(\eps)}\$ f;\bar f\$_{\gamma,\eta;\K_\eps}^{(\eps)} + \|\Pi^\eps-\bar \Pi^\eps\|_{\gamma;\bar\K_\eps}^{(\eps)}\$ f\$_{\gamma,\eta;\K_\eps}^{(\eps)}
\end{aligned}
\end{equation}
locally uniformly over $z$ and over all compact sets $\K_\eps$ with $\mathrm{diam}(\K_\eps)\leq 2\eps$ containing $z$.
Fix a continuous model $(\Pi,\Gamma)$ with associated reconstruction operator $\CR$, further fix $f\in\CD^{\gamma,\eta}(\Gamma)$, and $f^\eps\in\CD_\eps^{\gamma,\eta}(\Gamma^\eps)$ both taking values in $V$. We assume that for every $\eta\in \Phi$, 
\begin{equation}
\begin{aligned}
\big|(\iota_\eps\CR^\eps f^\eps - &\CR f )(\eta_z^{\eps})\big |\\
&\lesssim \eps^{\alpha\wedge\eta}\big[\|\Pi^\eps\|_{\gamma;\overline{[\eta_z^{\eps}]}}^{(\eps)}\$f;f^\eps\$_{\gamma;[\eta_z^{\eps}]}
+ \|\Pi;\Pi^\eps\|_{\gamma;\overline{[\eta_z^{\eps}]}}\$f\$_{\gamma;[\eta_z^{\eps}]}\big],
\end{aligned}
\end{equation}
locally uniformly in $z\in\R^d$, and uniformly over all $\eps\in (0,1]$.
\end{assumption}
\begin{remark}
We shortly argue why Assumption~\ref{ass:Rsmallscalesweighted} is reasonable. Assume that we are in the purely discrete case and that $\CR^\eps$ is given as in Remark~\ref{rem:rec}. Then, for every $f\in\CD_\eps^{\gamma,\eta}(\Gamma^\eps)$, any compact set $\K_\eps$, and any $y\in\Lambda_\eps^d\cap\K_\eps$ we have
\begin{equation}
\begin{aligned}
|\CR^\eps f(y)|&=
|(\Pi_y^\eps f(y))(y)|
\leq \sum_{\ell <\gamma}|(\Pi_y^\eps \CQ_\ell f(y))(y)|
\lesssim \|\Pi^\eps\|_{\gamma;\K_\eps}^{(\eps)}
\sum_{\ell <\gamma}\|f(y)\|_\ell \eps^{\ell}.
\end{aligned}
\end{equation}	
To see that this implies the first bound in~\eqref{eq:Rsmallscalesweighted}, multiply and divide each summand above by $\|y\|_{P,\eps}^{(\eta-\ell)\wedge 0}$ and note that $\|y\|_{P,\eps}^{(\eta-\ell)\wedge 0}\eps^{\ell-\alpha\wedge \eta}\leq 1$. The remaining
two parts of the assumption can be shown in a similar way in this case.
\end{remark}
\begin{theorem}
\label{thm:recwithweights}
Fix a discrete model $(\Pi^\eps,\Gamma^\eps)$, let $f\in\CD_\eps^{\gamma,\eta}(\Gamma^\eps)$ for some $\gamma>0$, some $\eta\leq \gamma$, such that $f$ takes values in a sector $V$ of regularity $\alpha\leq 0$, and assume that Assumption~\ref{ass:weightedDgamma} holds. Further let $\K$ be a compact set, and let $\eta\in\Phi$.
Then, provided that $\alpha\wedge\eta>-\mathfrak{m}$, the reconstruction operator $\CR^{\eps}$ satisfies the following.
\begin{itemize}
\item[1.] One has the bound
\begin{equation}
\label{eq:recwithweights1}
|\iota_\eps(\CR^{\eps}f-\Pi_z^\eps f(z))(\eta_z^{\delta})|\lesssim d_\s(z,P)^{\eta-\gamma}\delta^{\gamma},
\end{equation}
for all $\delta\in[\eps,1]$,
and for all $z\in\K$ such that $d_\s(z,P)\geq \bfc\eps+2\delta$.
\item[2.]
If Assumption~\ref{ass:Rsmallscalesweighted} is satisfied, then
\begin{equation}
\label{eq:recwithweights3}
|\iota_\eps(\CR^{\eps}f)(\eta_z^{\delta})|\lesssim \delta^{\alpha\wedge\eta}
\end{equation}
\end{itemize}
for all $\delta\in[\eps,1]$, and for all $z\in\K$.
In particular, there is no requirement on the location of the support of $\eta_z^{\delta}$.
In both estimates the proportionality constant is a multiple of $\|\Pi^\eps\|_{\gamma;\bar\K}^{(\eps)}\$f\$_{\gamma,\eta;[\eta_z^{\delta}]}^{(\eps)}$.
\begin{itemize}
\item[3.] Let $(\bar\Pi^\eps,\bar\Gamma^\eps)$ be a second discrete model with associated reconstruction operator $\bar\CR^\eps$ such that Assumption~\ref{ass:Rsmallscalesweighted} is satisfied, then
\begin{equation}
\label{eq:recwithweights4}
\begin{aligned}
|\iota_\eps&(\cR^{\eps}f-\bar\cR^{\eps}\bar f-\Pi_z^{\eps}f(z)+\bar\Pi_z^{\eps}\bar f(z))(\eta_z^{\delta})|
\lesssim d_\s(z,P)^{\eta-\gamma}\delta^\gamma,
\end{aligned}
\end{equation}
for all $\delta\in[\eps,1]$  and for all $z\in\K$ such that $d_\s(z,P)\geq \bfc\eps+2\delta.$
Moreover, the estimate~\eqref{eq:recwithweights3} holds for $|\iota_\eps(\CR^{\eps} f-\bar\CR^\eps\bar f)(\eta_z^{\delta})|$ as well.
In both cases, the proportionality constant is a multiple of 
\begin{equation}
\label{eq:propconstweighted}
\|\bar\Pi^\eps\|_{\gamma;\bar\K}^{(\eps)}\$f;\bar f\$^{(\eps)}_{\gamma,\eta;[\eta_z^{\delta}]}
+\|\Pi^{\eps}-\bar\Pi^{\eps}\|^{(\eps)}_{\gamma;\bar\K}\$ f\$_{\gamma,\eta;[\eta_z^{\delta}]}^{(\eps)}.
\end{equation}
\item[4.] Let $(\Pi,\Gamma)$ be a continuous model such that Assumption~\ref{ass:Rsmallscalesweighted} holds, then Item 3. holds for it as well. More precisely, if one replaces $(\bar\Pi^\eps,\bar\Gamma^\eps)$ by $(\Pi,\Gamma)$, the discrete reconstruction operator $\bar \CR^\eps$ by $\CR$ and $\bar f$ by $f\in\CD^{\gamma,\eta}(\Gamma)$, the same bounds as in Item 3 hold, provided the proportionality constants are adapted accordingly. 
\end{itemize}
\end{theorem}
\begin{remark}
We assume for the rest of this article that Assumptions~\ref{ass:weightedDgamma} and~\ref{ass:Rsmallscalesweighted} hold.
\end{remark}
\begin{proof}
The bound~\eqref{eq:recwithweights1} can be derived in the same way as \cite[Eq.~6.6]{Regularity}
(note that instead of using~\cite[Lem.~6.7]{Regularity} one may directly apply Theorem~\ref{thm:reconstruction}).
Thus, we assume from now on that $d_\s([\eta_z^{\delta}],P)< \bfc\eps+ 2\delta$.
We start by choosing a smooth function $\tilde\Psi:\R_+\to[0,1]$ such that $\tilde\Psi(\rho)=0$ for $\rho\notin [1/2,2]$ and such that
\begin{equation}
\label{eq:partofunity}
\sum_{n\in\Z}\tilde\Psi(2^n \rho)=1,\quad\mbox{for all }\rho>0.
\end{equation}
Furthermore, we let $\Psi:\R\to [0,1]$ be as in~\eqref{eq:summingtoone}. The final ingredients are a smooth mapping $N_P:\R^d\setminus P\to \R_+$ that satisfies the scaling relation $N_P(\CS_\s^{\delta^{-1}}z)=\delta N_P(z)$, depends only on $(z_1,\ldots, z_{\bar{d}})$, is such that for some constant $c>0$, the relation $c^{-1}N_P(z) \leq d_\s(z,P)\leq cN_P(z)$ holds for any $z\in\R^d$, and the sets $\Xi_P^n$ defined by
\begin{equation}
\label{eq:Xi}
\Xi_P^n=\{z\in\R^d:\, z_i=0\mbox{ for }i\leq \bar{d}\mbox{ and }z_i\in 2^{-n\s_i}\Z\mbox{ for }i>\bar{d}\}.
\end{equation}
We then define 
\begin{equation}
\label{eq:varphin}
\tilde\Psi_n(y)= \tilde\Psi(2^nN_P(y)).
\end{equation}
Let $n_0\in\Z$ be the greatest integer with $[\tilde\Psi_{k}]\cap [\eta_z^{\delta}]= \emptyset$ for all $k<n_0$ and $N$ be the greatest integer such that $d_\s([\tilde\Psi_N],P)\geq (1+\bfc)\varepsilon$. Note that the former (together with our additional assumption on the support of $\eta_z^{\delta}$ at the beginning of the proof) implies that $2^{-n_0}$ is of the order $\delta$, whereas the latter implies that $2^{-N}$ is of the order $\eps$. It then follows from~\eqref{eq:partofunity} that
\begin{equation}
\eta_z^{\delta} = \sum_{n_0\leq n\leq N}\eta_z^{\delta}\tilde\Psi_n + \sum_{n>N}\eta_z^{\delta}\tilde\Psi_n,
\end{equation}
which we write as $\I_z+\II_z$. 
Note that
\begin{equation}
\label{eq:Irewrite}
\I_z= \sum_{n_0\leq n\leq N}\sum_{y\in\Xi_P^n}\eta_z^{\delta}\tilde\Psi_n\Psi(2^{n\s_{\bar{d}+1}}(\cdot - y_{\bar{d}+1}))\cdot\ldots\cdot
\Psi(2^{n\s_{d}}(\cdot - y_{d})).
\end{equation}
To proceed define $\chi_{n,zy}(x)=\delta^{|\s|}2^{n|\s|}\eta_z^{\delta}(x)\tilde\Psi_n(x)\Psi(2^{n\s_{\bar{d}+1}}(x_{\bar d+1} - y_{\bar{d}+1}))\cdot\ldots\cdot
\Psi(2^{n\s_{d}}(x_d - y_{d}))$ and note that each $\chi_{n,zy}$ defines a test function. Applying a suitable partition of the identity we may write 
\begin{equation}
\chi_{n,zy}= \sum_{j=1}^{M}\chi_{n,zy}^{j}
\end{equation}
for some fixed constant $M$ and each $\chi_{n,zy}^{j}$ is supported in some ball with center $z_{y,j}\in(\R^d\setminus P)\cap [\chi_{n,zy}]$ and whose radius $r$ satisfies $\bfc\varepsilon+ 2r\lesssim d_\s(z_{y,j},P).$ 
Thus, by~\ref{eq:recwithweights1}, for each $j\in\{1,\ldots, M\}$,
\begin{equation}
\Big|\iota_\eps(\CR^{\eps}f-\Pi^\eps_{z_{y,j}}f(z_{y,j}))(\chi_{n,zy}^j)\Big|
\lesssim d_\s(z_{y,j},P)^{\eta-\gamma}2^{-\gamma n}\lesssim 2^{-\eta n}.
\end{equation}
Since $|\iota_\eps(\Pi^\eps_{z_{y,j}} f(z_{y,j}))(\chi_{n,zy}^j)|\lesssim 2^{-n(\alpha\wedge \eta)}$ and the fact that $2^{-n_0}\approx \delta$, we may conclude as in the proof of~\cite[Prop.~6.9]{Regularity}. To deal with $\II_z$ we 
multiply it by a partition of unity like so:
\begin{equation}
\label{eq:IIrewrite}
\II_z = \sum_{y\in \Xi_P^N}\II_z \, \Psi(2^{N\s_{\bar{d}+1}}(\cdot-y_{\bar{d}+1}))\cdot\ldots\cdot \Psi(2^{N\s_d}(\cdot-y_{d})).
\end{equation}
Define $\hat{\chi}_{N,zy}(x)= \delta^{|\s|}2^{N|\s|}\II_z(x) \Psi(2^{N\s_{\bar{d}+1}}(x_{\bar d+1}-y_{\bar{d}+1}))\cdot\ldots\cdot \Psi(2^{N\s_d}(x_d-y_{d}))$ and note that by Assumption~\ref{ass:Rsmallscalesweighted} and Equation~\ref{eq:relationtoseminorm},
\begin{equation}
|\iota_\eps(\CR^{\eps}f)(\hat{\chi}_{N,zy})|\lesssim \eps^{\alpha\wedge \eta} \|\CR^{\eps}f\|_{\alpha\wedge \eta;[\hat{\chi}_{N,zy}];z;\eps}
\lesssim \eps^{\alpha\wedge \eta}\$f\$_{\gamma,\eta;[\hat{\chi}_{N,zy}]}^{(\eps)}. 
\end{equation}
We may now finish as above. We omit the details.
Items \rm{3.} and \rm{4.} may be shown in a similar manner.
\end{proof}

%


\section{Convolution operators}
\label{S4}
In this section we explain how to convolve a modelled distribution with a discrete kernel.
Here, one should think of the kernel given by the Green's function of the linear part of the equation at hand.
It will be a standing assumption from now on that our regularity structure contains the polynomial regularity 
structure, and we write $\bar \CT$ for the span of the symbols $X^k$ representing the usual Taylor monomials. 
Throughout all of this section, the following assumption is in force.

\begin{assumption}
\label{a:model}
The regularity structure $\TT$ contains the polynomial regularity structure
$\bar\TT = (\bar \CT, \bar \CG, \N)$ corresponding to the scaling $\s$. 
We also assume that we are given a family of discrete ``polynomial models'' $(\Pi^\eps,\Gamma^\eps)$ on
$\TT$ converging to the canonical continuous polynomial model. 
\end{assumption}

\begin{remark}
It follows from our assumption that, for $\eps$ small enough, the map $\Pi_z$ is injective on $\bar \CT$
for every $z$. We henceforth assume that this is the case for all $\eps$, which is of course not a real assumption
since we are mostly interested in the case of $\eps \to 0$.
\end{remark}

We now describe the assumptions we want the kernel at hand to satisfy. Unfortunately these have a quite abstract appearance, so to illustrate what these really mean in practice we provide concrete examples in Remark~\ref{rem:example} below.  
Let $N$ be the smallest integer such that $2^{-N}\leq\eps$. Throughout this section we fix $\beta>0$. For $z\in\R$, and $\zeta\in\R$ we let $\CX_{\eps,\zeta,z}$ be the set of all $F\in\CX_\eps$ such that
\begin{equ}
|\iota_\eps(F)(\varphi_z^n)|\lesssim 2^{-n\zeta},
\quad\mbox{and}\quad
\|F\|_{\zeta; \K_\eps;z;\eps}\lesssim 1,
\end{equ}
for all scaled test functions $\varphi_{z}^{n}$ with $n\leq N$, and all compact sets $\K_\eps$ of diameter at most $2\eps$ containing $z$, and we require the proportionality constants to be independent of $n\leq N$. Recall at this point that we use the notation $\varphi_z^n=\varphi_z^{2^{-n}}$. We let $\CX_{\eps,\zeta}=\cup_{z\in\R}\CX_{\eps,\zeta,z}$.
We then assume that
there is a family of linear operators $K_n^{\eps}$ on $\CX_\eps$, as well as a family of linear operators $T_{n,\zeta}^{\eps} \colon \CX_{\eps,\zeta} \to (\bar\CT_{<\zeta+\beta})^{\R^d}$. 
For $F\in\CX_{\eps,\zeta,z}$ we use the notation 
\begin{equation}
\label{eq:projectionpoly}
(T_{n,\zeta}^{\eps}F)(z)= \sum_{|k|_{\s}<\zeta+\beta} X^k\, \CQ_{k}((T_{n,\zeta}^{\eps}F)(z))\;,
\end{equation}
and we abbreviate $K^{\eps}=\sum_{n=0}^{N}K_n^{\eps}$, and $T_{\zeta}^{\eps}= \sum_{n=0}^{N}T_{n,\zeta}^{\eps}$.
We then require the following.
\begin{assumption}
\label{a:TandK}
There is $\beta>0$ such that for all $n\leq N$, $\zeta,\zeta'\in \R$ with $\zeta'\leq \zeta$, all $y,z\in\R^d$ such that $\eps\leq \|y-z\|_\s\leq 2^{-n}$, and all $F\in\CX_{\eps,\zeta,z}\cap\CX_{\eps,\zeta,y},$
\begin{claim}
\item[1.] the consistency relation $\CQ_k((T_{n,\zeta}^{\eps} F)(z))= \CQ_k((T_{n,\zeta'}^{\eps}F)(z))$ holds for all multiindices $|k|_{\s} <  \zeta'+\beta$, 
\item[2.] one has the estimates
\begin{equation}
\label{eq:Tcomponentest}
\|(T_{n,\zeta}^{\eps}F)(z)\|_k\lesssim
\begin{cases} 
2^{n(|k|_\s-\beta)}\sup_{\varphi \in \Phi}|(\iota_\eps F)(\varphi_z^n)|, &\mbox{if }n\leq N-1,\\
2^{N(|k|_\s-\beta-\zeta)}\|F\|_{\zeta;\K_\eps;z;\eps}, &\mbox{if }n=N,
\end{cases}
\end{equation}
where in the latter estimate $\K_\eps$ denotes a ball of radius $\eps$ centred at $z$,
\item[3.] one has the estimate
\begin{equation}
\label{eq:Tdifferentest}
\begin{aligned}
\big\|&(T_{n,\zeta}^{\eps} F)(z)-\Gamma_{zy}^{\eps}(T_{n,\zeta}^{\eps}F)(y)\big\|_k\\
&\lesssim 2^{n(\lceil \zeta+\beta \rceil -\beta)}\|y-z\|_{\s}^{\lceil \zeta+\beta \rceil-|k|_\s}
\sup_{\varphi \in \Phi}
|(\iota_{\eps}F)(\varphi_{z}^{n-1})|,
\end{aligned}
\end{equation}
which only needs to be satisfied for $n\leq N-1$,
\item[4.] one has the estimate
\begin{equation}
\label{eq:relationKandT}
\begin{aligned}
\sum_{n\leq N}\| K_n^{\eps}F -&\Pi_z^\eps (T_{n,\zeta}^{\eps}F)(z)\|_{\zeta+\beta;\K_\eps;z;\eps}\\
&\lesssim 
\sup_{n\leq N} 2^{n\zeta}\sup_{\substack{\varphi  \in \Phi\\ z'\in\K_\eps}}|\iota_\eps(F)(\varphi_{z'}^{n-1})|
+\sup_{h\in \cB_{\s}(z,2\eps)}\|F\|_{\zeta;\K_\eps;z+h;\eps},
\end{aligned}
\end{equation}
locally uniformly over all compact sets $\K_\eps$ of diameter bounded by $2\eps$,
\item[5.] uniformly over $\phi\in\Phi$ and $\delta\in (\eps,1]$ one has the estimates
\begin{equs}
|\iota_\eps(K_n^\eps F)(\phi_z^\delta)|&\lesssim |\langle \CQ_0((T_{n,0}^\eps F)(\cdot)),\phi_z^\delta\rangle|,\\
|\iota_\eps(K_n^\eps F)(\phi_z^\delta)|&\lesssim 2^{-n\beta}\sup_{\eta\in\Phi}|(\iota_\eps F)(\eta_z^{2\delta})|,
\end{equs}
where the second inequality above is only required for $2^{-n}\leq \delta$.
The proportionality constants above are uniform over $y,z$ and $n$.
\end{claim}
\end{assumption}
\begin{remark}\label{rem:fourthitem}
The fourth item above is only needed in the proof of Theorem~\ref{thm:Schauder} to show that the reconstruction operator and the convolution operator $\CK_\gamma^\eps$ (whose definition will be given in the sequel) do essentially commute (see Theorem~\ref{thm:Schauder} for a more precise statement).
It however follows from Remark~\ref{rem:A} and~\ref{rem:commute} that in many cases of interest this follows immediately from the respective definitions, in which case the fourth item above is unnecessary.
	\end{remark}
As in~\cite{Regularity} we impose that the kernels $K_n^{\eps}$ kill polynomials up to a sufficiently high degree. In our setting this may be formulated as follows.
\begin{assumption}
\label{a:killpoly}
There is a $\sigma>0$ such that for all $n\leq N$, all $z \in \R^d$, and all $|k|_\s\leq \sigma$, one has
$K_n^{\eps}\Pi_z^\eps X^k=0$ and $(T_{n,\zeta}^{\eps}\Pi_z^\eps X^k)(z) = 0$.
\end{assumption}
We remind the reader of the following definition from~\cite{Regularity}.
\begin{definition}
\label{def:I}
Given a sector $V$, a linear map $\CI:V\to \CT$ is an abstract integration map of order $\beta>0$, if the following properties are satisfied:
\begin{claim}
\item[1.] One has $\CI:V_\zeta\to \CT_{\zeta+\beta}$, and this mapping is continuous.
\item[2.] One has $\CI\tau=0$ for all $\tau\in V\cap \bar\CT$.
\item[3.] One has $\Gamma\CI-\CI\Gamma\in \bar\CT$ for every $\Gamma\in G$.
\end{claim}
\end{definition}  

\begin{definition}
\label{def:PirealisesK}
Fix a sector $V$ and an abstract integration map $\CI$ on it. We say that $\Pi^\eps$ realises $K^\eps$ for $\CI$, if for every $\zeta\in A$, every $\tau\in V_{\zeta}$, and every $z\in\R^d$,
\begin{equation}
\label{eq:PirealisesK}
\Pi_z^\eps \CI(\tau) = K^{\eps} \Pi_z^\eps \tau -  \Pi_z^\eps (T_{|\tau|}^{\eps} \Pi_z^\eps\tau)(z).
\end{equation} 
We furthermore require that $\Pi^\eps$ agrees with the discrete polynomial model from Assumption~\ref{a:model}
on $\bar \CT$.
\end{definition}

With all of these definitions at hand, we are now in the position to provide the
definition of the ``convolution map'' $\CK^{\eps}$ on modelled distributions announced at the beginning of this
section. As in~\cite{Regularity} it turns out that for different values of $\gamma$ one should use slightly
different definitions. Given $f\in\CD^{\gamma}_\eps$, we set
\begin{equation}
\label{eq:convolution}
\CK_{\gamma}^{\eps} f(z)= \CI f(z) +\sum_{\zeta\in A}(T_{\zeta}^{\eps}\Pi_z^\eps\CQ_{\zeta}f(z))(z)
+(T_{\gamma}^{\eps}(\cR^\eps f-\Pi_z^\eps f(z)))(z).
\end{equation}
Before we state one of the main results of this article we need one more assumption. 
\begin{assumption}
\label{a:Schauersmallscale}
Let $\TT=(A,\CT, G)$ be a regularity structure, let $V$ be a sector of $\CT$, let
$(\Pi^\eps,\Gamma^\eps)$ be a discrete model, and fix $\gamma,\beta> 0$. We assume that,
for $f\in\CD^{\gamma}_\eps(V,\Gamma^\eps)$
and any compact set $\K$,
\begin{equation}
\label{eq:Schaudersmallscale}
\$ \CK_{\gamma}^{\eps} f\$_{\gamma+\beta;\K;\eps}\lesssim \|\Pi^\eps\|_{\gamma;\dbar\K}^{(\eps)}\$ f\$_{\gamma;\dbar\K;\eps}.
\end{equation}
Let $(\bar\Pi^\eps,\bar\Gamma^\eps)$ be a second discrete model for $\TT$, denote by $\bar\CK_\gamma^\eps$ the associated convolution operator, and let $\bar f\in\CD_\eps^{\gamma}(V,\bar\Gamma^\eps)$. We assume that,
\begin{equation}
\label{eq:Schauderdifsmallscale}
\$ \CK_{\gamma}^{\eps} f;\bar\CK_\gamma^\eps\bar f\$_{\gamma+\beta;\K;\eps}\lesssim \|\Pi^\eps\|_{\gamma;\dbar\K}^{(\eps)}\$ f;\bar f\$_{\gamma;\dbar\K}^{(\eps)} + \|\Pi^\eps-\bar\Pi^\eps\|_{\gamma;\dbar\K}^{(\eps)}\$f\$_{\gamma;\dbar\K}^{(\eps)}.
\end{equation} 
In both estimates the proportionality constant is supposed to be uniform in $\eps >0$ and we remind the reader that $\dbar\K$ denotes the 2-fattening of $\K$.
\end{assumption}

\begin{theorem}
\label{thm:Schauder}
Let $\TT=(A,\CT,G)$ be a regularity structure and $(\Pi^\eps,\Gamma^{\eps})$ be a model for 
$\TT$ satisfying Assumption~\ref{a:model} and fix a compact set $\K$. Let $\beta>0$ and assume that there are operators $K_n^{\eps}$ and 
$T_{n,\zeta}^{\eps}$ satisfying Assumption~\ref{a:TandK}, let $\CI$ be an abstract integration map acting on some sector $V$
and let $Z^\eps=(\Pi^{\eps},\Gamma^{\eps})$ be a discrete model realising $K^{\eps}=\sum_{n=0}^{N}K_n^{\eps}$ for $\CI$. Let furthermore $\gamma>0$, assume that Assumption~\ref{a:killpoly} is satisfied for $\sigma=\gamma$ and define the operator
 $\CK_{\gamma}^{\eps}$ by~\eqref{eq:convolution}. Then provided that $\gamma+\beta\notin\N$, and that $\CK_{\gamma}^{\eps}$ satisfies Assumption~\ref{a:Schauersmallscale} we have that $\CK_{\gamma}^{\eps}$ maps $\CD^{\gamma}_\eps(V)$ into
$\CD^{\gamma+\beta}_\eps(V)$ and there is an operator $A^{\eps}:\CD^{\gamma}_\eps(V)\to \CX_\eps$ such that the identity
\begin{equation}
\label{eq:convolutionidentity}
\CR^{\eps}\CK_{\gamma}^{\eps}f= K^{\eps}\CR^{\eps}f+ A^{\eps}f
\end{equation}
holds for every $f\in\CD^{\gamma}_\eps(V)$. The operator $A^{\eps}$ satisfies the estimate
\begin{equation}
\label{eq:A}
\|A^{\eps}f\|_{\gamma+\beta;\K_\eps;z;\eps}\lesssim \$ f\$_{\gamma;\dbar\K_\eps}^{(\eps)}
\end{equation}
locally uniformly over compact sets $\K_\eps$ of diameter bounded by $2\eps$ and over $z\in\R^d$.
The proportionality constant depends on the norm of the model, but is independent of $\eps$ otherwise.
If $\bar Z^\eps=(\bar\Pi^{\eps},\bar\Gamma^{\eps})$ is a second model satisfying Assumption~\ref{a:model} and realising $K^{\eps}$ for $\CI$ such that Assumption~\ref{a:Schauersmallscale} is satisfied, then one has for every $f\in\CD^{\gamma}_\eps(V,\Gamma^{\eps})$ and $\bar f\in\CD^{\gamma}(V,\bar \Gamma^{\eps})$ the bound
\begin{equation}
\label{eq:convolutiondifference}
\$ \CK_{\gamma}^{\eps} f; \bar\CK_{\gamma}^{\eps}\bar f\$^{(\eps)}_{\gamma+\beta;\K}
\lesssim \|\bar\Pi^\eps\|_{\gamma;\dbar\K}^{(\eps)}\$f;\bar f\$^{(\eps)}_{\gamma;\dbar\K}+ \|\Pi^{\eps}-\bar\Pi^{\eps}\|^{(\eps)}_{\gamma;\dbar\K}\$f\$_{\gamma;\dbar\K}^{(\eps)},
\end{equation}
uniformly in $\eps$. 
\end{theorem}
\begin{remark}\label{rem:smallscale}
Assumption~\ref{a:Schauersmallscale} will never be explicitely used in the proof of Theorem~\ref{thm:Schauder}. Therefore, the above theorem essentially states that given $f\in \CD_\eps^\gamma$, then $\CK_\gamma^\eps f$ satisfies the large scale estimate in the Definition~\ref{def:Dgamma}.	
	\end{remark}
\begin{remark}
	\label{rem:example}
	Examples of kernels satisfying all assumptions from this section are:
	\begin{claim}
		\item[\rm{1.}] the usual heat kernel, i.e., the fundamental solution to $\partial_t K(t,x)= \Delta K(t,x)$;
		\item[\rm{2.}] the usual discrete heat kernel, i.e., the fundamental solution to $\partial_t K(t,x)= \Delta^d K(t,x)$, where $\Delta^d$ is the discrete Laplacian acting on functions defined on $\Z^d$;
		\item[\rm{3.}] the fundamental solution to $\partial_t K^\eps (t,x) =(\Delta -\eps^2 \Delta ^2)K^{\eps}(t,x)$.
	\end{claim}
We will only give the arguments for the kernel $K^\eps$ in the third item. For the heat kernel (discrete heat kernel) the arguments then follow in the same way (but using the estimates in~\cite[Section 5]{MatetskiDiscrete} for the discrete heat kernel).
We assume that we are in the continuous case, and that the scaling is given by $\s=(2,1,\ldots,1)$. 
We claim that for $\|z\|_\s \ge \eps$ this kernel satisfies the bounds
\begin{equation}\label{e:boundKeps}
|D_x^\ell D_t^m K^\eps(z)|\lesssim \|z\|_\s^{-|\s|+2-|\ell| - 2|m|}\;.
\end{equation}
To show \eqref{e:boundKeps}, setting $z = (t,x)$ as usual, we distinguish between the 
case $t \le \eps^2$ and the case $t \ge \eps^2$ and we exploit the explicit form of $K^\eps$:
\begin{equ}
D_x^\ell D_t^m K^\eps = \widehat{f^\eps_{\ell,m}(\cdot,t)}\;,\quad
f^\eps_{\ell,m}(k,t) = (-1)^m i^\ell k^\ell (|k|^2 + \eps^2 |k|^4)^m e^{-(|k|^2 + \eps^2 |k|^4)t}\;.
\end{equ}
For $t \le \eps^2$, we exploit the identity
\begin{equ}
{1\over t^{d/2}} f^\eps_{\ell,m}(k/\sqrt t,t) = t^{-{|\ell|+2|m|+d \over 2}} g_{\ell,m}(\eps^2/t,k)\;,
\end{equ}
where we set 
\begin{equ}
g_{\ell,m}(c,k) = (-1)^m i^\ell k^\ell (|k|^2 + c |k|^4)^m e^{-(|k|^2 + c |k|^4)}\;.
\end{equ}
Since $g_{\ell,m}(c,\cdot)$ is a Schwartz function for every $c$ with all its seminorms bounded uniformly over 
$c \in [0,1]$, we deduce that there exists a uniformly bounded collection of Schwartz functions $G(c,\cdot)$ 
such that
\begin{equ}
D_x^\ell D_t^m K^\eps (t,x) = t^{-{|\ell|+2|m|+d \over 2}} G\Big({\eps^2\over t}, {x \over \sqrt t}\Big)\;.
\end{equ}
This immediately implies that
\begin{equ}
|D_x^\ell D_t^m K^\eps (t,x)| \lesssim {t^{-{|\ell|+2|m|+d \over 2}} \over 1 + |t^{-1/2} x|^{|\ell| + 2|m| + d}}
\lesssim \|z\|_\s^{-d-|\ell| - 2|m|}\;,
\end{equ}
as claimed.
For $t \le \eps^2$ we similarly note that 
\begin{equ}
{1\over (\eps^2t)^{d/4}} f^\eps_{\ell,m}(k/(\eps^2t)^{d/4},t) = \eps^{-{d+|\ell| \over 2}} t^{-{|\ell|+4|m|+d \over 4}} h_{\ell,m}(t/\eps^2,k)\;,
\end{equ}
for some collection $h_{\ell,m}(c,\cdot)$ of Schwartz functions, uniformly bounded over 
$c \in [0,1]$, so that for $t \le \eps^2$ one can write
\begin{equ}[e:sharpBound]
D_x^\ell D_t^m K^\eps (t,x) = \eps^{-{d+|\ell| \over 2}} t^{-{|\ell|+4|m|+d \over 4}} H\Big({t\over \eps^2}, {x \over (\eps^2 t)^{1/4}}\Big)\;,
\end{equ}
with $H$ having the same properties as $G$. This yields as before
\begin{equ}[e:sharpHomogeneousBound]
|D_x^\ell D_t^m K^\eps (t,x)| \lesssim {\eps^{2|m|} \over |\eps^2 t|^{|\ell|+4|m| + d\over 4} + |x|^{|\ell|+4|m|+d}}\;,
\end{equ}
which yields the required bound when combining it with $|t| \le \eps^2$ and $|x| \ge \eps$.
We moreover note that as a consequence of~\eqref{e:sharpBound} and~\eqref{e:sharpHomogeneousBound} the bound~\eqref{e:boundKeps} holds even for $\|z\|_\s\leq \eps$ provided that $m=\ell=0$.

With this estimate at hand we can then use the same techniques as in~\cite[Sections 5 and 7]{Regularity} to decompose the kernel $K^\eps$ as
\begin{equation}
K^\eps =R^\eps + K_N^\eps + \sum_{n = 1}^{N-1} K_n^\eps,
\end{equation}
where $R^\eps$ is a smooth (uniformly in $\eps$) 
remainder and each $K^\eps_n$ is supported in the set $\{z:\, \|z\|_\s\leq 2^{-n}\}$, and satisfies the estimate
\begin{equation}
\sup_z |D_x^\ell D_t^m K_n^\eps (t,x)|\lesssim 2^{n(|\s|-2 +|\ell| +2|m|)},
\end{equation}
but only for $\ell = m = 0$ when $n= N$.
We then define for any $\zeta\in\R$, and any $n\leq N$,
\begin{equation}
(T_{n,\zeta}^\eps F)(z)=
\begin{cases}
 \sum_{|k|_\s < \zeta+2} \frac{X^k}{k!}\int D^k K_n^\eps(z-y) F(y)\, dy, &\mbox{if }n\leq N-1,\\
  \delta_{\zeta+2 >0}\one\int K_N^\eps(z-y) F(y)\, dy, &\mbox{if }n=N,
\end{cases}
\end{equation}
and now we can finally show the validity of all the assumptions in this section. To that end we assume that $\zeta$ is such that $\zeta+2 >0$, since otherwise there is nothing to be shown.
The first item follows directly from the construction of $T_{n,\zeta}^\eps$. To see that the second item is satisfied note that for all $n\leq N-1$ and all multiindices $k$ the function
$z\mapsto 2^{n(2-|k|_\s)}D^k K_n^\eps(z)$ defines a scaled test function, so that we can write
\begin{equation}
\begin{aligned}
k!\|(T_{n,\zeta}^\eps F)(z)\|_k
&= \Big|\int D^k K_n^\eps(z-y) F(y)\, dy\Big|\\
&=2^{n(|k|_\s-2)}\Big|\int 2^{n(2-|k|_s)}D^k K_n^\eps(z-y) F(y)\, dy\Big|,
\end{aligned}
\end{equation}
which implies the first part in the estimate~\eqref{eq:Tcomponentest} (with $\beta=2$). 
For the small scale estimate, in the same way we see that
\begin{equ}
\|(T_{N,\zeta}^\eps F)(z)\|_0
=\Big|\int K_N^\eps(z-y) F(y)\, dy\Big|
\lesssim 2^{-N(2+\zeta)}\|F\|_{\zeta;\K_\eps;z;\eps}, 
\end{equ}
where we used the definition \eqref{e:defNormcont} of $\|\cdot\|_{\zeta;\K_\eps;z;\eps}$ to gain the additional factor
$2^{-\zeta N}$. Since $\|(T_{N,\zeta}^\eps F)(z)\|_k=0$ by definition for all multi-indices $k$ with at least one positive coordinate, 
we conclude that the second item in Assumption~\ref{a:TandK} is satisfied.
We turn to the third item in Assumption~\ref{a:TandK}. To that end we note that for $n\leq N-1$,
\begin{equation}
\begin{aligned}
k!\|&(T_{n,\zeta}^\eps F)(z)- \Gamma_{zy}^\eps (T_{n,\zeta}^\eps F)(y)\|_k\\
&=\Big|\int [D^k K_n^\eps(z-x)-\sum_{|\alpha+k|_\s< \zeta +2}\frac{1}{\alpha!}(z-y)^\alpha D^{k+\alpha}K_n^\eps(y-x)]F(x)\, dx\Big|,
\end{aligned}
\end{equation}
so that the proof of~\cite[Lemma 5.18]{Regularity} yields the desired estimate. Indeed, one may simply copy its proof (which is an application of Taylor's formula), but in \cite[Equ.~5.29]{Regularity} one uses the fact that $2^{n(2-|k+\ell|_\s)}D^{k+\ell} K_n^\eps(y+h-z)$ defines a test function with a support of radius $2^{n-1}$ centred around $y$.
Since we are in the continuous case, according to Remarks~\ref{rem:A} and~\ref{rem:commute} the fourth item is in practice not necessary, so that we omit its proof.
		The first part in the fifth item of Assumption~\ref{a:TandK} is a direct consequence of the construction of $T_{n,\zeta}^\eps$. Regarding the second part we note that for $n\leq N-1$, such that $2^{-n}\leq \delta$,
		\begin{equation}
		\iota_\eps (K_n^\eps F)(\varphi_z^\delta)= 
		\int \varphi_z^\delta (y) K_n^\eps (y-x) F(x)\, dx,
		\end{equation}
		and that $2^{2n}K_n^\eps(y-x)\lesssim 2^{n|s|}$, so that 
		\begin{equation}
		\int 2^{2n}K_n^\eps(y-x)\varphi_z^\delta (y)\, dy
		\end{equation}
		defines a test function with a support of radius $2\delta$ centred around $z$. The estimate in the case $n\leq N-1$ is then immediate. 
		To see that we also have the desired estimate for $n=N$ we simply use the fact that $\int K_N^\eps \lesssim 2^{-2N}$.
		It moreover follows from the proof of~\cite[Lemma 5.5]{Regularity} that one can modify the $K_n^\eps$'s such that additionally Assumption~\ref{a:killpoly} is satisfied. Finally, Assumption~\ref{a:Schauersmallscale} is actually a consequence of Theorem~\ref{thm:Schauder}. Indeed, according to Remark~\ref{rem:smallscale} Assumption~\ref{a:Schauersmallscale} is never explicitly used in its proof. Since in the continuous case the $\CD_\eps^\gamma$-norm on small scales involves division by powers of $\eps$ (and not the true distance between any chosen points $y$ and $z$), and since moreover in the the large scale part of the definition of the $\CD_\eps^\gamma$-spaces one is allowed to choose $y$ and $z$ such that their distance is of order $\eps$, a triangle inequality argument shows the validity of Assumption~\ref{a:Schauersmallscale}.
	\end{remark}
\begin{remark}
In the above example for the kernel $K^\eps$ we assumed that we are in the continuous case. This would be for example a good choice in the case of the $\Phi_4^3$-equation, since one can make sense of $K^\eps* \xi$, where $\xi$ denotes space-time white noise, ``by hand'' as a continuous function. One might however imagine other examples where the continuous case is not the right choice. If this happens then it seems to be the case that one may also choose some sort of inhomogeneous transparent case, where one keeps the definitions of the usual transparent case, but at scales smaller than $\eps$ test functions are scaled differently, i.e., they are scaled in the following way:
\begin{equation*}
\phi_{(t,x)}^\delta (s,y)=\delta^{-d-4}\eps ^2 \phi(\delta^{-4}\eps^2 (s-t),\delta^{-1}(y-x)).
\end{equation*}
This scaling is such that, at scale $\eps$, it coincides with the usual parabolic scaling, but at scales below $\eps$ it respects the fact that as a consequence of~\eqref{e:sharpBound} and~\eqref{e:sharpHomogeneousBound} one has for all multi-indices $\ell$ and $m$ the estimate
\begin{equation*}
|D_x^\ell D_t^m K^\eps(t,x)|\lesssim \eps^{2|m|}\|z\|_\eps^{-(|\s|-2)-|\ell|-4|m|},
\end{equation*}
for $\|z\|_\s\leq \eps$, where $\|z\|_\eps=((\eps^2 t)^{1/4}+|x|)$.
	\end{remark}
\begin{remark}
\label{rem:A}
As a consequence of Theorem~\ref{thm:Schauder} one only has an identity of the form 
$\CR^{\eps}\CK_{\gamma}^{\eps}=K^{\eps}\CR^{\eps} + \mathrm{small\, error}$. Since the kind of approximation 
schemes we are interested in usually have a divergent part, one 
might get worried that one is not able to control the error term. However, it turns out that in many examples of interest 
one is able to enforce the identity $\CR^{\eps}\CK_{\gamma}^{\eps}=K^{\eps}\CR^{\eps}$ by modifying $\CK_{\gamma}^{\eps}$. 
Indeed, consider the common situation in which elements of $\CX_\eps$ can be identified with functions
on $\R^d$ (or some subset thereof) and where the reconstruction operator is given 
by $(\CR^\eps f)(z)=(\Pi_z^{\eps}f(z))(z)$. One then has
\begin{equation}
\label{eq:exampleK}
(\CR^\eps\CK_{\gamma}^{\eps} f)(z)= (K^\eps\Pi_z^\eps f(z))(z)+ \Pi_z^\eps(T_{\gamma}^\eps(\CR^\eps f-\Pi_z^\eps f(z))(z))(z).
\end{equation}
The example to have in mind to think about~\eqref{eq:exampleK} is that the first summand on the right hand side is 
given by
\begin{equation}
\label{eq:firstterm}
\int  K^\eps(z,y)(\Pi_z^\eps f(z))(y)\, dy,
\end{equation}
(for a suitable, possibly discrete, reference measure $dz$),
whereas the second term on the right hand side is given by
\begin{equation}
\label{eq:Texample}
\sum_{|k|_\s<\gamma+\beta} \frac{(\Pi_z^\eps X^k)(z)}{k!}\int D_1^k K^\eps(z,y)(\CR^\eps f-\Pi_z^\eps f(z))(y)\, dy.
\end{equation}
Assume that $\Pi_z^\eps \one=1$ (here, we set $\one = X^0$). The definition of a model then yields $(\Pi_z^\eps X^k)(z) \approx \eps^{|k|_\s}$ (it does however \emph{not} yield $(\Pi_z^\eps X^k)(z)=0$). In this setting one then has
\begin{equation}
\label{eq:CAeps}
\begin{aligned}
(\CR^\eps\CK_{\gamma}^{\eps} f)&(z)- (K^\eps\CR^\eps f)(z)\\
&=\sum_{n}\sum_{0<|k|_\s<\gamma+\beta}\frac{(\Pi_z^\eps X^k)(z)}{k!}\int D_1^k K_n^\eps(z,y)(\CR^\eps f-\Pi_z^\eps f(z))(y)\, dy.
\end{aligned}
\end{equation} 
Define now an operator $\CA^\eps$ via
\begin{equation}
\label{eq:CAdef}
(\CA^\eps f)(z)=\bigg[\sum_{n}\sum_{0<|k|_\s<\gamma+\beta}\frac{(\Pi_z^\eps X^k)(z)}{k!}\int D_1^k K_n^\eps(z,y)(\CR^\eps f-\Pi_z^\eps f(z))(y)\, dy\bigg]\one.
\end{equation}
Then, the operator $\bar{\CK}_\gamma^\eps=\CK_{\gamma}^{\eps} -\CA^\eps$ satisfies $\CR^\eps\bar{\CK}_\gamma^\eps= K^\eps\CR^\eps$ as desired.
\begin{remark}\label{rem:commute}
The problem with the above construction is that there seems to be no reason in general for $\CA^\eps$ to map $\CD_\eps^\gamma$ into $\CD_\eps^{\gamma+\beta}$ so that it may be necessary to introduce higher order corrections to $\CA^\eps$.
However, consider the discrete, semidiscrete or continuous case introduced in Section~\ref{S2}. It is usually possible to impose $(\Pi_z^\eps X^k)(z)=0$, for all $z$ in the support of the reference measure (this was for instance imposed in~\cite{MatetskiDiscrete}), in which case \eref{eq:CAeps} shows that the choice $(\CR^\eps f)(z)= (\Pi_z^\eps f(z))(z)$, automatically yields $\CA^\eps \equiv 0$.
Finally, let us mention at this point that it was shown in~\cite{Regularity} that the identity $\CA^\eps\equiv 0$ also
holds in the transparent case. This illustrates that for most cases of interest no further modification of 
$\CK_\gamma^\eps$ is needed.
\end{remark}
\end{remark}
\begin{proof}
Fix $y,z\in\K$ such that $\eps\leq \|y-z\|_\s \leq 1$. We first estimate the non-polynomial part.
Making use of the first and third property in Definition~\ref{def:I} we see that for any $\ell\notin \N$,
\begin{equs}
\label{eq:nonpolypart}
\|&\CK_{\gamma}^{\eps} f(z)-\Gamma^{\eps}_{zy}\CK_{\gamma}^{\eps} f(y)\|_\ell\\
&=\|\CI(f(z)-\Gamma^{\eps}_{zy}f(y))\|_\ell
\lesssim \|f(z)-\Gamma^{\eps}_{zy}f(y)\|_{\ell-\beta}
\leq \|y-z\|_\s^{\gamma+\beta-\ell}\$ f\$^{(\eps)}_{\gamma;\dbar\K}.
\end{equs}
In a similar way we see that 
\begin{equation}
\label{eq:nonpolypartdif}
\|\CK_{\gamma}^{\eps} f(z)-\Gamma^{\eps}_{zy}\CK_{\gamma}^{\eps} f(y)-\bar\CK_{\gamma}^{\eps}\bar f(z)+\bar\Gamma^{\eps}_{zy}\bar f(y)\|_\ell
\lesssim \|y-z\|_\s^{\gamma+\beta-\ell}\$f;\bar f\$_{\gamma;\dbar\K}^{(\eps)}.
\end{equation}
Thus, \eqref{eq:nonpolypart}--\eqref{eq:nonpolypartdif} show that the  non-polynomial components satisfy the desired estimates.
To deal with the polynomial components of $\CK_{\gamma}^{\eps}$ we make use of the following lemma that is established after this proof.
\begin{lemma}
\label{lem:commutationrelation}
Under the assumptions of Theorem~\ref{thm:Schauder}, for any $\zeta\in A$ and any $a\in V$ with homogeneity $\zeta$, one has the identity
\begin{equation}
\label{eq:commutationrelation}
\Gamma^{\eps}_{zy}\Big(\CI a+\sum_{n}(T_{n,\zeta}^{\eps}\Pi_y^\eps a)(y)\Big)
= \CI \Gamma^{\eps}_{zy} a+\sum_{n,\bar\zeta\in A}(T_{n,\bar\zeta }^{\eps}\Pi_z^\eps\CQ_{\bar\zeta}\Gamma^{\eps}_{zy}a)(z)
\end{equation}
for any choice of $y,z\in\R^d$.
\end{lemma}
As a consequence of Lemma~\ref{lem:commutationrelation} we see that, 
\begin{equation}
\label{eq:applycommutation}
\begin{aligned}
\Gamma^{\eps}_{zy}\CI f(y)&+\Gamma^{\eps}_{zy}\sum_{n,\zeta}(T_{n,\zeta}^{\eps}\Pi_y^\eps\CQ_{\zeta} f(y))(y)\\
&= \CI\Gamma^{\eps}_{zy}f(y) +\sum_{n,\bar\zeta}(T_{n,\bar\zeta}^{\eps}\Pi_z^\eps\CQ_{\bar\zeta}\Gamma^{\eps}_{zy} f(y))(z).
\end{aligned}
\end{equation}
Note that $\CI$ does not produce any polynomial component. Thus, with~\eqref{eq:applycommutation} at hand, we see that for any multiindex $k$, $(\Gamma^{\eps}_{zy}\CK_{\gamma}^{\eps} f(y))_k-(\CK_{\gamma}^{\eps} f(z))_k= \sum_n[\I_n+\II_n - \III_n]$, where
\begin{equs}
\I_n&= \textstyle{\sum_{\bar\zeta}}\CQ_k(T_{n,\bar\zeta}^{\eps}(\Pi_z^\eps\CQ_{\bar\zeta}[\Gamma^{\eps}_{zy} f(y)-f(z)])(z)),\\
\II_n&= \CQ_k(\Gamma^{\eps}_{zy}(T_{n,\gamma}^\eps(\CR^{\eps}f-\Pi_y^\eps f(y))(y)),\label{eq:polypart}\\
\III_n&=\CQ_k(T_{n,\gamma}^\eps(\CR^{\eps}f-\Pi_z^{\eps}f(z))(z)).
\end{equs}
We first assume that $2^{-n}\leq \|y-z\|_\s$. In this case we bound $\I_n$, $\II_n$ and $\III_n$ seperately.
We deduce from Assumption~\ref{a:TandK} that for all $\bar\zeta$ and all $n$,
\begin{equation}
\label{eq:firstpolysmallscale}
\big\|T_{n,\bar\zeta}^{\eps}(\Pi_z^\eps\CQ_{\bar\zeta}[\Gamma^{\eps}_{zy}f(y)-f(z)])(z)\big\|_{|k|_\s}
\lesssim 2^{n(|k|_\s-\beta-\bar\zeta)}\|y-z\|_\s^{\gamma-\bar\zeta}\|\Pi^{\eps}\|_{\gamma;\dbar\K}^{(\eps)}\$f\$_{\gamma;\dbar\K}^{(\eps)}.
\end{equation}
We deduce from the representation in~\eqref{eq:polypart} that only those values $\bar\zeta$ contribute to the first sum in~\eqref{eq:polypart} for which $\bar\zeta > |k|_\s-\beta$. Thus, summing the right hand side of~\eqref{eq:firstpolysmallscale} over $n$ such that $2^{-n}\leq \|y-z\|_\s$ we obtain an upper bound that is a multiple of $\|y-z\|_\s^{\gamma+\beta-|k|_\s}$.
In a similar way we can deal with $\III_n$, but making also use of the Reconstruction Theorem~\ref{thm:reconstruction}.
To estimate $\II_n$ we first note that as a consequence of Assumption~\ref{a:model},
\begin{equation}
\label{eq:smallscaleIIrepresentation}
\begin{aligned}
\big\|\Gamma_{zy}^\eps T_{n,\gamma}^\eps&(\CR^{\eps}f-\Pi_y^{\eps}f(y))(y)\big\|_{|k|_\s}\\
&\lesssim \sum_{\ell: |k+\ell|_\s<\gamma+\beta}\binom{k+\ell}{\ell} \|z-y\|_\s^\ell
\big|\CQ_{\ell+k}(T_{n,\gamma}^\eps(\CR^{\eps} f-\Pi_y^\eps f(y))(y))\big|\\
&\lesssim \sum_{\ell:|k+\ell|_\s<\gamma+\beta} 2^{n(|\ell+k|_\s-\beta-\gamma)}\|z-y\|_\s^{|\ell|_\s}\$f\$_{\gamma,\dbar\K}^{(\eps)}\|\Pi^{\eps}\|_{\gamma;\dbar\K}^{(\eps)}.
\end{aligned}
\end{equation}
Summing each summand in~\eqref{eq:smallscaleIIrepresentation} first over $n$ such that $2^{-n}\leq \|y-z\|_\s$, we obtain a bound that is a multiple of $\|y-z\|^{\gamma+\beta-|k|_\s}\$f\$_{\gamma,\dbar\K}^{(\eps)}\|\Pi^{\eps}\|_{\gamma;\dbar\K}^{(\eps)}$ as desired.
The corresponding bounds on the difference $\CK_{\gamma}^{\eps}-\bar\CK_{\gamma}^{\eps}$ are obtained in a similar fashion.
We now seek for bounds on large scales, i.e., for $\|y-z\|_\s< 2^{-n}$. Recall the consistency relation in Assumption~\ref{a:TandK}. Thus, for any $\bar\zeta \in (|k|_\s-\beta , \gamma)$
\begin{equation}
\label{eq:consistency}
\CQ_k(T_{n,\bar\zeta}^\eps(\Pi_z^{\eps}\CQ_{\bar\zeta}
[\Gamma_{zy}^{\eps}f(y)-f(z)])(z))
= \CQ_k(T_{n,\gamma}^\eps (\Pi_z^{\eps}\CQ_{\bar\zeta}
[\Gamma_{zy}^{\eps}f(y)-f(z)])(z)).
\end{equation}
Consequently, adding and subtracting $\CQ_k(T_{n,\gamma}^\eps \Pi_z^\eps[\Gamma_{zy}^\eps f(y)-f(z)](z))$ to $\I_n+\II_n-\III_n$, we see that we can write $\I_n+\II_n-\III_n$ as $-\I'_n+\II'_n-\III'_n$, where
\begin{equation}
\label{eq:123rewrite}
\begin{aligned}
\I'_n&=\sum_{\bar\zeta\leq |k|_\s-\beta}\CQ_k(T_{n,\gamma}^\eps(\Pi_z^{\eps}\CQ_{\bar\zeta}
[\Gamma_{zy}^{\eps}f(y)-f(z)])(z)),\\
\II'_n&= \CQ_k(T_{n,\gamma}^\eps(\Pi_y^{\eps}f(y)-\CR^{\eps}f)(z)),\\
\III'_n&= \CQ_k(\Gamma_{zy}^{\eps}(T_{n,\gamma}^{\eps}(\Pi_y^{\eps}f(y)-\CR^{\eps}f))(y))
\end{aligned}
\end{equation}
To bound $\I'_n$ first note that for any $\bar\zeta\leq |k|_\s-\beta$,
\begin{equation}
\label{eq:boundonIprime}
\I_n'\lesssim 2^{n(|k|_\s-\beta-\bar\zeta)}\|\Pi^{\eps}\|_{\gamma;\dbar\K}^{(\eps)}\|y-z\|_\s^{\gamma-\bar\zeta}
\$f\$_{\gamma;\dbar\K}^{(\eps)}.
\end{equation}
Summing this over $n$, such that $2^n < \|y-z\|_\s^{-1}$, leads to a bound that is a multiple of $\|y-z\|_\s^{\gamma+\beta-|k|_\s}$. It remains to bound the difference $\II'_n-\III'_n$. We note that as a consequence of the third item in Assumption~\ref{a:TandK},
\begin{equation}
\label{eq:23prime}
|\II'_n-\III'_n|\lesssim
2^{n(\lceil \gamma+\beta\rceil-\beta)}\|y-z\|_s^{\lceil \gamma+\beta\rceil-|k|_\s}
\sup_{\substack{\varphi \in \Phi}}
\big|\iota_\eps(\Pi_y^{\eps}f(y)-\CR^{\eps}f)(\varphi_{z}^{n-1})\big|.
\end{equation}
Writing 
\begin{equation}
\label{eq:piminusRrewrite}
\Pi_y^{\eps}f(y)-\CR^{\eps}f= (\Pi_{z}^{\eps}f(z)-\CR^{\eps}f)
+\Pi_{z}^{\eps}(\Gamma_{zy}^{\eps}f(y)-f(z)),
\end{equation}
we may estimate using the Reconstruction Theorem~\ref{thm:reconstruction},
\begin{equation}
\label{eq:firsttermin23dif}
\sup_{\varphi \in \Phi}
\big|\iota_\eps(\Pi_{z}^{\eps}f(z)-\CR^{\eps}f)(\varphi_{z}^{n})\big|
\lesssim 2^{-n\gamma}\|\Pi^{\eps}\|_{\gamma;\dbar\K}^{(\eps)}\$f\$_{\gamma;\dbar\K}^{(\eps)}.
\end{equation}
Regarding the second term on the right hand side of~\eqref{eq:piminusRrewrite}, we can estimate
\begin{equation}
\label{eq:secondterm23dif}
\begin{aligned}
\sup_{\varphi \in \Phi}
&\big|\iota_\eps(\Pi_{z}^{\eps}(\Gamma_{zy}^{\eps}f(y)-f(z)))(\varphi_{z}^{n-1})\big|
&\lesssim\sum_{\zeta<\gamma}\|\Pi^{\eps}\|_{\gamma;\dbar\K}^{(\eps)}\|y-z\|^{\gamma-\zeta}\$f\$_{\gamma;\dbar\K}^{(\eps)}2^{-n\zeta}.
\end{aligned}
\end{equation}
Plugging the right hand sides of \eqref{eq:firsttermin23dif} and \eqref{eq:secondterm23dif} into~\eqref{eq:23prime}, and summing over $n$ such that $2^{n}\leq \|y-z\|_\s^{-1}$ yields the bound of the required order.
The corresponding bounds on the difference $\CK_{\gamma}^{\eps}-\bar\CK_{\gamma}^{\eps}$ can be obtained in a similar way. We omit  the details.

We now turn to the proof of~\eqref{eq:convolutionidentity}.
As a consequence of~\eref{eq:rec} and Assumption~\ref{a:Schauersmallscale} we have that for every compact set $\K_\eps$ of diameter bounded by $2\eps$, and every $z\in\R^d$,
\begin{equs}
\label{eq:reconstructionconsequence}
\|\CR^\eps\CK_{\gamma}^{\eps} f-\Pi_z^{\eps}\CK_{\gamma}^{\eps} f(z)\|_{\gamma+\beta;\K_\eps;z;\eps}
&\lesssim\|\Pi^{\eps}\|_{\gamma+\beta;\bar\K_\eps}^{(\eps)}\$\CK_{\gamma}^{\eps} f\$_{\gamma+\beta;\K_\eps;\eps} \\
&\lesssim (\|\Pi^{\eps}\|_{\gamma+\beta;\dbar\K_\eps}^{(\eps)})^2\$ f\$_{\gamma;\dbar\K_\eps;\eps} .
\end{equs}
Setting $A^{\eps}f=\CR^{\eps}\CK_{\gamma}^{\eps} f-K^{\eps}\CR^{\eps}f$, it remains to show that
\begin{equation}
\label{eq:toestablish}
\|\Pi_z^{\eps}\CK_{\gamma}^{\eps} f(z)-K^{\eps}\CR^{\eps}f\|_{\gamma+\beta;\K_\eps;z;\eps}\lesssim\|\Pi^{\eps}\|_{\gamma+\beta;\bar\K_\eps}^{(\eps)}\$f\$_{\gamma;\dbar\K_\eps}^{(\eps)}.
\end{equation}
Since the model realises $K^{\eps}$ for $\CI$, one has
\begin{equation}
\label{eq:PiKf}
\Pi_z^{\eps}\CK_{\gamma}^{\eps} f(z)
=\sum_{n}\Big(K_n^{\eps}\Pi_z^{\eps}f(z)+\Pi_z^{\eps}(T_{n,\gamma}^\eps[\CR^{\eps}f-\Pi_z^{\eps}f(z)])(z)\Big).
\end{equation}
Consequently, the left hand side in~\eqref{eq:toestablish} can be written as 
\begin{equation}
\label{eq:toestablish2}
\sum_{n}\Big(K_n^{\eps}(\Pi_z^{\eps}f(z)-\CR^{\eps}f)- \Pi_z^\eps (T_{n,\gamma}^{\eps}[\Pi_z^{\eps}f(z)-\CR^{\eps}f](z))\Big)
\end{equation}
and we may deduce~\eqref{eq:toestablish} using the fourth item in Assumption~\ref{a:TandK}.
\end{proof}
We now provide the proof of Lemma~\ref{lem:commutationrelation}.
\begin{proof}
First note that $\Gamma_{zy}^{\eps}\CI a-\CI\Gamma_{zy}^{\eps}a\in\bar\CT$. Since we assumed that $\Pi_z^{\eps}$
is injective on $\bar\CT$, it is enough to show that
\begin{equation}
\label{eq:usingPiinjective}
\Pi_z^\eps\Big(\Gamma_{zy}^{\eps}\CI a+\sum_{n}\Gamma_{zy}^{\eps}(T_{n,\zeta}^{\eps}\Pi_y^{\eps} a)(y)\Big)
=\Pi_z^\eps\Big(\CI\Gamma_{zy}^{\eps} a+\sum_{n,\bar\zeta}(T_{n,\bar\zeta}^{\eps}\Pi_z^{\eps} \CQ_{\bar\zeta}\Gamma_{zy}^{\eps}a)(z)\Big).
\end{equation}
This identity however follows from the fact that $\Pi^{\eps}$ realises $K^{\eps}$ for $\CI$.
\end{proof}

\begin{remark}\label{rem:lsm}
At first sight it is not clear that a map $\CI$ with the properties as stated in this section exists. However, the result below shows that one is always able to extend the regularity structure such that it accommodates a map $\CI$ as in Definition~\ref{def:I}. It also shows that one is able to extend the model to a pair $(\hat\Pi^\eps,\hat\Gamma^\eps)$ 
such that $\hat\Pi^\eps$ realises $K^\eps$ for $\CI$. The pair $(\hat\Pi^\eps,\hat\Gamma^\eps)$ turns out to satisfy all defining properties of a model except the second inequality in~\eqref{eq:Gamma} (which however still holds on the original regularity structure). We call such a pair a \textit{large scale model} for the extended regularity structure.
\end{remark}
\begin{remark}
The proof of Theorem~\ref{thm:extension} below shows that in the setup of the discrete, semidiscrete or continuous case, assuming that on the right hand side of~\eqref{eq:Tdifferentest} the factor $\|y-z\|_\s^{\lceil \zeta+\beta\rceil-|k|_\s}$ can be replaced by $\eps^{\lceil \zeta+\beta\rceil-|k|_\s}$ whenever $\|y-z\|_\s\leq \eps$, the large scale model $(\hat\Pi^\eps,\hat\Gamma^\eps)$ turns out to be a (proper) model. In the transparent setting it is sufficient that~\eqref{eq:Tdifferentest} also holds for all $\|y-z\|_\s\leq \eps$.
\end{remark}
\begin{theorem}
\label{thm:extension}
Let $\TT=(A,\CT,G)$ be a regularity structure and $(\Pi^\eps,\Gamma^\eps)$ be a model for $\TT$ satisfying Assumption~\ref{a:model}. Let $\beta>0$ and assume that there are operators $K_n^\eps$ and $T_{n,\zeta}^\eps$ satisfying Assumption~\ref{a:TandK}. Let $V$ be a sector with the property that for every $\alpha\notin\N$ with $V_\alpha\neq 0$, one has $\alpha+\beta\notin\N$. Let furthermore $W$ be a subsector of $V$ and let $\CI:W\to\CT$ be an abstract integration map of order $\beta$ such that the model $(\Pi^\eps,\Gamma^\eps)$ realises $K^\eps$ for $\CI$. Then, there exists a regularity structure $\hat\TT$ containing $\TT$ and a large scale  model $(\hat\Pi^\eps,\hat\Gamma^\eps)$ 
(see Remark~\ref{rem:lsm}) for $\hat\TT$ extending $(\Pi^\eps,\Gamma^\eps)$ and an abstract integration map $\hat\CI$ of order $\beta$ defined on $V$ such that $\hat\Pi^\eps$ realises $K^\eps$ for $\hat\CI$ and such that $\hat\CI\tau=\CI\tau$ for all $\tau\in W$. Moreover, this extension is continuous.
\end{theorem}
\begin{remark}
We refer the reader to~\cite[Thm~5.14]{Regularity} for a formulation of a more quantitative 
version of the continuity statement, which also holds in our case. 
\end{remark}
\begin{proof}
As in the proof of~\cite[Thm~5.14]{Regularity} we may restrict ourselves to the situation where the sector $V$ is given by
\begin{equation}
V=V_{\alpha_1}\oplus V_{\alpha_2}\oplus\ldots\oplus V_{\alpha_n},
\end{equation}
the $\alpha_i$ are an increasing sequence of elements in $A$, and $W_{\alpha_k}=V_{\alpha_k}$ for all $k<n$.
The algebraic part of the proof of~\cite[Thm~5.14]{Regularity} shows that it is possible to define $\hat\CI$ on $V$ such that it satisfies the properties stated in Definition~\ref{def:I}.
It remains to extend the model $(\Pi^\eps,\Gamma^\eps)$. To that end we define for $\tau\in V_{\alpha_n}$,
\begin{equation}
\label{eq:defPiandGamma}
\begin{aligned}
\hat\Pi_z^\eps\hat\CI(\tau) &\eqdef K^\eps\hat\Pi_z^\eps \tau-\hat\Pi_z^\eps (T_{|\tau|}^\eps \hat\Pi_z^\eps \tau)(z),\\
\hat\Gamma^\eps_{zy}\hat\CI(\tau) &\eqdef \hat\CI\hat\Gamma_{zy}^\eps\tau
+\sum_{\zeta\in A}(T_{\zeta}^\eps \hat\Pi_z^\eps \CQ_{\zeta}\hat\Gamma_{zy}^\eps \tau)(z)
-\hat\Gamma_{zy}^\eps(T_{|\tau|}^\eps \hat\Pi_y^\eps \tau)(y).
\end{aligned}
\end{equation}
Note that both quantities are well defined, since on the corresponding right hand sides $\hat\Pi^\eps$ and $\hat\Gamma^\eps$ are either applied to elements of homogeneity smaller or equal to $|\tau|$ or to polynomials. We leave it as an exercise to verify that $\hat\Pi^\eps$ and $\hat\Gamma^\eps$ defined in this way indeed satisfy the required algebraic constraints.
We now show that $\hat\Pi_z^\eps\hat\CI(\tau)$ satisfies the analytical estimates stated in Definition~\ref{def:model}.
First of all, note that by the fourth item in Assumption~\ref{a:TandK},
\begin{equation}
\label{eq:smallscaleextension}
\|\hat\Pi_z^\eps \hat\CI(\tau)\|_{|\tau|+\beta;\K_\eps;z;\eps}
\leq \sum_{n\leq N}\|K_n^\eps\hat\Pi_z^\eps \tau-\hat\Pi_z^\eps(T_{n,|\tau|}^\eps \hat\Pi_z^\eps\tau)(z)\|_{|\tau|+\beta;\K_\eps;z;\eps}
\lesssim 1,
\end{equation}
where we recall that $N$ is the smallest integer such that $2^{-N}\leq \eps$. We now turn to the required estimates on scales larger than $\eps$. 
We first treat the case $|\tau|+\beta <0$. Let $\delta\in(\eps,1]$, fix $\eta\in\Phi$, and $z\in\R^d$.
We first note that $\hat\Pi_z^\eps\hat\CI(\tau)=K^\eps\hat\Pi_z^\eps\tau$. Thus, for $n$ such that $2^{-n}>\delta$ we can estimate using the fifth and second item in Assumption~\ref{a:TandK},
\begin{equs}
\label{eq:smallnextension}
|\iota_\eps(K_n^\eps\hat\Pi_z^\eps\tau)(\eta_z^\delta)|
&\lesssim \Big|\int \CQ_0((T_{n,0}^\eps\hat\Pi_z^\eps\tau)(x))\eta_z^\delta(x)\, dx\Big|\\
&\lesssim 2^{-n\beta}\int\sup_{\phi\in\Phi}|(\iota_\eps\hat\Pi_z^\eps\tau)(\phi_x^{\delta})||\eta_z^\delta(x)|\, dx.
\end{equs}
One may now show that uniformly over all $x$ in the support of $\eta_z^\delta$
\begin{equs}
\label{eq:piestimateextension}
|(\iota_\eps\hat\Pi_z^\eps\tau)(\phi_x^n)|\lesssim \sum_{\alpha\leq |\tau|}2^{-n\alpha}\delta^{|\tau|-\alpha}.
\end{equs} Thus, plugging this estimate in into the right hand side of~\eqref{eq:smallnextension} and summing over $2^{-n}>\delta$ yields the correct bound. The bound in the case $2^{-n}\leq\delta$ may be obtained in a similar manner, but making use of the second estimate of the fifth item in Assumption~\ref{a:TandK}.
We now deal with the case $|\tau|+\beta >0$, and we note that the case $|\tau|+\beta=0$ is excluded by assumption. Define as in the proof of Theorem~\ref{thm:reconstruction} functions $\tilde{\Psi}_{z,[z_k]}^{\delta,k}$ and $\tilde{\Psi}_{z,[z_k,z_{k+1}]}^{\delta,k}$ with $z_k\in\Lambda_k^{\s}$ and $z_{k+1}\in\Lambda_{k+1}^{\s}$. Moreover, let $n_0$ be the smallest integer such that $2^{-n_0}\leq \delta$ and let $z_{|k}$ be as in the aforementioned proof. We can write
$\iota_\eps(\hat\Pi_z^\eps\hat\CI(\tau))(\eta_z^{\delta})=\I+\II+\III$, where 
\begin{equs}
\I&=\sum_{z_N\in\Lambda_N^{\s}}\iota_\eps(\hat\Pi_{z_{|N}}^\eps\hat\CI(\tau))(\tilde{\Psi}_{z,[z_N]}^{\delta,k}),\\
\II&= \sum_{k=n_0}^{N-1}\sum_{z_k\in\Lambda_k^{\s}, z_{k+1}\in\Lambda_{k+1}^{\s}}
\iota_\eps(\hat\Pi_{z_{|k}}^\eps \hat\CI(\tau)-\hat\Pi_{z_{|k+1}}^\eps \hat\CI(\tau))(\tilde{\Psi}_{z,[z_k,z_{k+1}]}^{\delta,k}),\\
\III&= \sum_{z_{n_0}\in\Lambda_{n_0}^{\s}}\iota_\eps(\hat\Pi_z^\eps\hat\CI(\tau) -\hat\Pi_{z_{|n_0}}^\eps\hat\CI(\tau))(\tilde{\Psi}_{z,[z_{n_0}]}^{\delta,n_0}).
\end{equs}
As in~\eqref{eq:Iest} we can estimate
\begin{equation}
|\iota_\eps(\hat\Pi_{z_{|N}}^\eps\hat\CI(\tau))(\tilde{\Psi}_{z,[z_N]}^{\delta,k})|
\lesssim (\delta 2^N)^{-|\s|}2^{-N(|\tau|+\beta)}
\|\hat\Pi_{z_{|N}}^{\eps}\hat\CI(\tau)\|_{|\tau|+\beta;[\tilde{\Psi}_{z,[z_N]}^{\delta,N}];z_{|N};\eps}.
\end{equation}
We now obtain the desired estimate on $\I$ making use of~\eqref{eq:smallscaleextension} and arguing in the same way as in the proof of Theorem~\ref{thm:reconstruction}.
To estimate $\II$, note that
\begin{equation}
\hat\Pi_{z_{|k}}^{\eps}\hat\CI(\tau)-\hat\Pi_{z_{|k+1}}^{\eps}\hat\CI(\tau)
= \hat\Pi_{z_{|k}}^{\eps}(\hat\CI(\tau)-\hat\Gamma_{z_{|k}z_{|k+1}}^\eps\hat\CI(\tau)).
\end{equation}
By the defining equation of $\hat\Gamma^\eps$ in~\eqref{eq:defPiandGamma}, the term inside the brackets may be written as
\begin{equation}
\label{eq:polynonpoly}
\hat\CI(\tau-\hat\Gamma_{z_{|k}z_{|k+1}}^{\eps}\tau)
-\sum_{\zeta\in A}(T_{\zeta}^{\eps}\hat\Pi_{z_{|k}}^\eps \CQ_{\zeta}\hat\Gamma_{z_{|k}z_{|k+1}}^\eps\tau)(z_{|k})
+\hat\Gamma_{z_{|k}z_{|k+1}}^\eps (T_{|\tau|}^\eps \hat\Pi_{z_{|k+1}}^\eps\tau)(z_{|k+1}).
\end{equation}
We estimate the polynomial and non-polynomial parts separately. Note that $\tau-\hat\Gamma_{z_{|k}z_{|k+1}}^\eps\tau$ is an element of $\CT_{<|\tau|}$. Thus, by our recursive construction
\begin{equation}
\begin{aligned}
|\iota_\eps(\hat\Pi_{z_{|k}}^\eps &\hat\CI(\tau-\hat\Gamma_{z_{|k}z_{|k+1}}^\eps\tau))(\tilde{\Psi}_{z,[z_k,z_{k+1}]}^{\delta,k})|\\
&\lesssim (\delta 2^k)^{-|\s|}\sum_{\zeta<|\tau|+\beta}2^{-k\zeta}\|\hat\CI(\tau-\hat\Gamma_{z_{|k}z_{|k+1}}^\eps\tau)\|_{\zeta}\\
&\lesssim (\delta 2^k)^{-|\s|}\sum_{\zeta<|\tau|+\beta}2^{-k\zeta}\|\tau-\hat\Gamma_{z_{|k}z_{|k+1}}^\eps\tau\|_{\zeta-\beta}\\
&\lesssim 2^{-k(|\tau|+\beta)}(\delta 2^k)^{-|\s|}.
\end{aligned}
\end{equation}
Here we used the continuity of $\hat\CI$ in the penultimate inequality. To obtain the last inequality we used the fact that by our recursive construction $\hat\Gamma^\eps$ satisfies the required analytical estimates on large scales when applied to elements of homogeneity smaller than or equal to $|\tau|$. Summing the above over the domain of summation of $\II$ yields the desired bound.
To bound the non-polynomial part we make use of the following lemma whose proof is deferred to the end of this section.
\begin{lemma}
\label{lem:nonpolyextension}
Fix $\alpha<|\tau|+\beta$, then uniformly in $k\in[n_0,\ldots, N-1]$,
\begin{equation}
\label{eq:polycomponentsextension}
\begin{aligned}
\Big\|\sum_{\zeta\in A}&(T_{\zeta}^{\eps}\hat\Pi_{z_{|k}}^\eps \CQ_{\zeta}\hat\Gamma_{z_{|k}z_{|k+1}}^\eps\tau)(z_{|k})
-\hat\Gamma_{z_{|k}z_{|k+1}}^\eps (T_{|\tau|}^\eps \hat\Pi_{z_{|k+1}}^\eps\tau)(z_{|k+1})\Big\|_{\alpha}\\
&\lesssim \|z_{|k}-z_{|k+1}\|_\s^{|\tau|+\beta-\alpha}.
\end{aligned}
\end{equation}
\end{lemma}
With Lemma~\ref{lem:nonpolyextension} at hand the desired estimate on $\II$ readily follows. The estimate on $\III$ works along the same lines, we omit the details. This shows that $\hat\Pi^\eps$ satisfies the required analytical estimates.
We turn to the analytical estimates of $\hat\Gamma^\eps$. To that end recall~\eqref{eq:defPiandGamma}. We first deal with the non-polynomial part. In that case we have for any $\alpha<|\tau|+\beta$,
\begin{equation}
\|\hat\CI(\hat\Gamma_{zy}^\eps\tau-\tau)\|_{\alpha}
\lesssim \|\hat\Gamma_{zy}^\eps\tau-\tau\|_{\alpha-\beta}.
\end{equation}
Thus, the desired estimate follows from our recursive construction.
The desired estimate of the polynomial part is a consequence of Lemma~\ref{lem:nonpolyextension}.
\end{proof}
We finish this section with the proof of Lemma~\ref{lem:nonpolyextension}.
\begin{proof}
We write $T_{\zeta}^\eps=\sum_{n}T_{n,\zeta}^\eps$ and analogously for $T_{|\tau|}^\eps$. Then, for each $n$, the component of order $\alpha$ of the term on the left hand side of~\eqref{eq:polycomponentsextension} can be written as
\begin{equation}
\label{eq:rewriteextension}
\sum_{\zeta>\alpha-\beta}\CQ_\alpha((T_{n,|\tau|}^{\eps}\hat\Pi_{z_{|k}}^\eps \CQ_{\zeta}\hat\Gamma_{z_{|k}z_{|k+1}}^\eps\tau)(z_{|k}))-\CQ_\alpha(\hat\Gamma_{z_{|k}z_{|k+1}}^\eps (T_{n,|\tau|}^\eps \hat\Pi_{z_{|k+1}}^\eps\tau)(z_{|k+1})).
\end{equation} 
We first treat the case $2^{-n}\leq\|z_{|k}-z_{|k+1}\|_\s$. We bound the terms above seperately.
The second term in~\eqref{eq:rewriteextension} is bounded from above by
\begin{equs}
\sum_{\zeta< |\tau|+\beta}&\|\hat\Gamma_{z_{|k}z_{|k+1}}^\eps \CQ_\zeta(T_{n,|\tau|}^\eps\hat\Pi_{z_{|k+1}}^\eps\tau)(z_{|k+1}))\|_{\alpha}\\
&\lesssim \sum_{\zeta< |\tau|+\beta}\|z_{|k}-z_{|k+1}\|_\s^{\zeta-\alpha}2^{n(\zeta-\beta-|\tau|)},
\end{equs}
where we made use of the second item in Assumption~\ref{a:TandK}. Summing the above over $n$ such that $2^{-n}\leq \|z_{|k}-z_{|k+1}\|_\s$ yields the desired bound. Making again use of the second item in Assumption~\ref{a:TandK} and our recursive construction we see that the first term in~\eqref{eq:rewriteextension} is bounded from above by
\begin{equation}
\sum_{\zeta>\alpha-\beta}2^{n(\alpha-\beta-\zeta)}\|\hat\Gamma_{z_{|k}z_{|k+1}}^\eps\tau\|_{\zeta}
\lesssim \sum_{\zeta>\alpha-\beta}2^{n(\alpha-\beta-\zeta)}\|z_{|k}-z_{|k+1}\|_\s^{|\tau|-\zeta}.
\end{equation}
Summing over $n$ such that $2^{-n}\leq \|z_{|k}-z_{|k+1}\|_\s$ yields the the correct bound.
We now treat the case $\|z_{|k}-z_{|k+1}\|_\s< 2^{-n}$. To that end we rewrite~\eqref{eq:rewriteextension} as
$\I_{z_{|k}z_{|k+1}}^\alpha-\II_{z_{|k}z_{|k+1}}^\alpha$, where
\begin{equs}
\I_{z_{|k}z_{|k+1}}^\alpha&=\CQ_\alpha((T_{n,|\tau|}^\eps\hat\Pi_{z_{|k+1}}^\eps\tau)(z_{|k}))
-\CQ_\alpha(\hat\Gamma_{z_{|k}z_{|k+1}}^\eps(T_{n,|\tau|}^\eps\hat\Pi_{z_{|k+1}}^\eps\tau)(z_{|k+1})),\\
\II_{z_{|k}z_{|k+1}}^\alpha&= \sum_{\zeta\leq \alpha-\beta}\CQ_\alpha((T_{n,|\tau|}^\eps\hat\Pi_{z_{|k}}^\eps\CQ_\zeta\hat\Gamma_{z_{|k}z_{|k+1}}^\eps\tau)(z_{|k})).
\end{equs}
The third item in Assumption~\ref{a:TandK} shows that $\I_{z_{|k}z_{|k+1}}^\alpha$ is bounded from above by
\begin{equation}
\label{eq:Iextension}
2^{n(\lceil |\tau|+\beta\rceil-\beta)}\|z_{|k}-z_{|k+1}\|_\s^{\lceil |\tau|+\beta\rceil-\alpha}2^{-n|\tau|}.
\end{equation}
We turn to the estimate of $\II_{z_{|k}z_{|k+1}}^\alpha$. A similar reasoning as above yields
\begin{equation}
\label{eq:IIextension}
\II_{z_{|k}z_{|k+1}}^\alpha
\lesssim \sum_{\zeta\leq \alpha-\beta}2^{n(\alpha-\beta-\zeta)}\|z_{|k}-z_{|k+1}\|_\s^{|\tau|-\zeta}.
\end{equation}
Thus, summing~\eqref{eq:Iextension} and~\eqref{eq:IIextension} over $2^n\leq \|z_{|k}-z_{|k+1}\|_\s^{-1}$ once again yields the correct bound.
\end{proof}
\subsection{Comparison between continuous and discrete convolution operators}
\label{S4.1}
In this section we show how one can compare the convolution operators built from a continuous model $(\Pi,\Gamma)$ and from a discrete model $(\Pi^\eps,\Gamma^\eps)$. Before we dive into the details we shortly explain the notion of continuous convolution operators. This notion was introduced in \cite[Section 5]{Regularity} and we explain here how it fits into the framework of the current article.
Assume that there is a kernel $K:\R^d\times\R^d\to\R$ that can be decomposed as
\begin{equation}
K(y,z)= \sum_{n\geq 0} K_n(y,z),
\end{equation}
where the $K_n$'s are smooth functions supported in the set $\{(y,z):\,\|y-z\|_\s\leq 2^{-n}\}$.
Given a distribution $\xi\in\CD'(\R^d)$ such that for some $\zeta\in\R$,
\begin{equation}
\label{eq:goodxi}
|\xi(\varphi_z^n)|\lesssim 2^{-n\zeta},
\end{equation} 
for all scaled test functions $\varphi_z^n$, all $z\in\R^d$ and all $n\leq N$, we define for any multiindex $k$,
\begin{equation}
\label{eq:contT}
\CQ_k((T_{n,\zeta}\xi)(z))=
\begin{cases}
{1\over k!}\int D^k K_n(z,y)\xi(dy),\quad &\mbox{if }n<N,\\
{1\over k!}\sum_{n\geq N}\int D^k K_n(z,y)\xi(dy),\quad &\mbox{if }n=N.
\end{cases}
\end{equation}
We can now define $(T_{n,\zeta}\xi)(z)$ in the same way as $(T_{n,\zeta}^{\eps}F)(z)$ in Equation~\ref{eq:projectionpoly}. 

We note that the first item in Assumption~\ref{a:TandK} is satisfied for $T_{n,\zeta}$ by construction and we
assume that there is $\beta>0$ such that also the remaining four items are satisfied for it. We moreover assume that 
the value of $\beta$ coincides with the value of $\beta$ for the discrete convolution operator. We now define a 
continuous convolution operator $\CK_\gamma$ in the same way as $\CK_\gamma^\eps$ in~\eqref{eq:convolution} with 
the only difference that each $T_{\zeta}^{(\eps)}$ is replaced by $T_{\zeta}$, and likewise 
$T_{\gamma}^{(\eps)}$ is replaced by $T_{\gamma}$.
We make the following assumption.
\begin{remark}
We note that Assumption~\ref{a:TandK} was shown to hold in the transparent case. Indeed, as mentioned above, the first item is a direct consequence of the construction of the $T_{n,\zeta}$'s, the second item is a consequence of~\cite[Equ.~5.4]{Regularity}, the third item is a consequence of~\cite[Equs~5.28 \&\ 5.31]{Regularity} and the fourth item can be deduced from the third by making use of choice of the family of seminorms on $\CX_\eps=\CD'(\R^d)$ in the transparent case.
\end{remark}
\begin{assumption}
\label{a:TandKcontdisc}
There is a constant $C(\eps) >0$ such that for all $n\leq N$, all $\zeta\in A$, all $y,z\in\R^d$ such that $\eps\leq \|y-z\|_\s\leq 2^{-n}$, 
and $F^\eps\in\CX_{\eps,\zeta,z}\cap\CX_{\eps,\zeta,y}$ and $F\in\CD'(\R^d)$ satisfying~\eqref{eq:goodxi},
\begin{enumerate}
\item[1.] one has the estimate,
\begin{equation}
\begin{aligned}
\big\|&(T_{n,\zeta}^\eps F^\eps)(z) - (T_{n,\zeta}F)(z)\big\|_{|k|_\s}\\
&\lesssim 2^{n(|k|_\s-\beta)}\Big(\sup_{\varphi \in \Phi}\big|(\iota_\eps F^\eps)(\varphi_z^n)- F(\varphi_z^n)\big|
+ C(\eps)\sup_{\varphi \in \Phi}\big|F(\varphi_{z}^{n})\big|\Big),
\end{aligned}
\end{equation}
\item[2.] one has the estimate,
\begin{equs}
\big\|(T_{n,\zeta}^\eps F^\eps)&(z) - (T_{n,\zeta}F)(z) - \Gamma_{zy}^\eps(T_{n,\zeta}^\eps F^\eps)(y) + \Gamma_{zy}(T_{n,\zeta}F)(y)\big\|_{|k|_\s}\\
&\lesssim 2^{n(\lceil \zeta+\beta\rceil-\beta)}\|y-z\|_\s^{\lceil\zeta+\beta\rceil-|k|_\s}\\
&\times\Big(\sup_{\varphi \in \Phi}\big|(\iota_\eps F^\eps)(\varphi_{z}^{n-1})- F(\varphi_{z}^{n-1})\big| + 
C(\eps)\sup_{\varphi \in \Phi}\big|F(\varphi_{z}^{n-1})\big|\Big).
\end{equs}
\end{enumerate}
The proportionality constants above are assumed to be independent of all parameters involved.
\end{assumption}

\begin{remark}
Although we did not state it like this, this assumption only seems to be useful when
$\lim_{\eps \to 0} C(\eps) = 0$.
\end{remark}

With Assumption~\ref{a:TandKcontdisc} at hand, we have the following comparison theorem.
\begin{theorem}
\label{thm:Schaudercontdisc}
Let $(A,\TT, G)$ be a regularity structure and let $(\Pi^\eps,\Gamma^\eps)$ and $(\Pi,\Gamma)$ be a discrete and a continuous model for it. Let the assumptions of Theorem~\ref{thm:Schauder} be satisfied and let Assumption~\ref{a:TandKcontdisc} be satisfied. Then, for any $f^\eps\in\CD_\eps^{\gamma}$ and $f\in\CD^{\gamma}$, one has the bound
\begin{equation}
\label{eq:Schaudercontdisc}
\begin{aligned}
\$\CK_\gamma &f;\CK_{\gamma}^{\eps}f^\eps\$_{\gamma+\beta;\K}
\lesssim \|\Pi;\Pi^\eps\|_{\gamma;\dbar\K}\$f\$_{\gamma;\dbar\K} + \|\Pi^{(\eps)}\|_{\gamma;\dbar\K}^{(\eps)}\$f;f^\eps\$_{\gamma;\dbar\K} \\&+ C(\eps)\|\Pi\|_{\gamma;\dbar\K}\$f\$_{\gamma;\dbar\K} 
+\$\CK_\gamma^\eps f^\eps\$_{\gamma+\beta;\K,\eps}+ \sup_{\substack{(y,z)\in\K\\ \|y-z\|_\s < \eps}}\sup_{\beta <\gamma}\frac{\|\CK_\gamma f(z)-\Gamma_{zy}\CK_\gamma f(y)\|_\beta}{\|y-z\|_\s^{\gamma-\beta}},
\end{aligned}
\end{equation}
where the proportionality constant is independent of all parameters.
\end{theorem}
\begin{proof}
Since the proof works along similar lines as the proof of Theorem~\ref{thm:Schauder}, we refrain from providing
all the details. We however illustrate the appearance of the terms in the second line of \eref{eq:Schaudercontdisc}. First of all, the last two terms are a direct consequence of Definition~\ref{def:Dgammacontdisc} of $\$\cdot,\cdot\$_{\gamma;\K}$. To see how the first term appears, note that the term in~\eqref{eq:firstpolysmallscale} equals in the current context
\begin{equation}
\CQ_k(T_{n,\bar\zeta}(\Pi_z^\eps\CQ_{\bar\zeta}[\Gamma_{zy}^\eps f^\eps(y)-f^\eps(z)])(z))
- \CQ_k(T_{n,\bar\zeta}(\Pi_z\CQ_{\bar\zeta}[\Gamma_{zy}f(y)-f(z)])(z)).
\end{equation}
According to Assumption~\ref{a:TandKcontdisc} it may be bounded by
\begin{equation}
\label{eq:Qkeps0}
\begin{aligned}
2^{n(|k|_\s-\beta)}\Big(\sup_{\varphi_z^n}\big|\iota_\eps(\Pi_z^\eps& \CQ_{\bar\zeta}[\Gamma_{zy}^{\eps}f^{\eps}(y)-f^{\eps}(z)])(\varphi_z^n)- \Pi_z \CQ_{\bar\zeta}[\Gamma_{zy}f(y)-f(z)](\varphi_z^n)\big|\\
&+ C(\eps)\sup_{\varphi_z^n}\big|\Pi_z \CQ_{\bar\zeta}[\Gamma_{zy}f(y)-f(z)](\varphi_z^n)\big|\Big),
\end{aligned}
\end{equation}
and one see that the latter term in~\eqref{eq:Qkeps0} is responsible for the appearance of the term
$C(\eps)\|\Pi\|_{\gamma;\bar\K}\$f\$_{\gamma;\bar\K}.$  
\end{proof}

\subsection{Convolution operators in weigthed $\CD^\gamma_\eps$-spaces}
\label{S4.2}
In this section we extend the result from the previous section to spaces of singular modelled distributions. 
Before we state the main result of this subsection we need one more assumption.
\begin{assumption}
\label{ass:Schauderweightedsmall}
Let $\TT=(A,\CT, G)$ be a regularity structure and let $V$ be a sector of regularity $\alpha\in A$. Fix $\gamma,\beta >0$, $\eta\in \R$ and a compact set $\K$. We assume that
\begin{equation}
\label{eq:Schauderweightedsmall}
\$\CK_{\gamma}^{\eps} f\$_{\bar\Gamma,\bar\eta;\K;\eps}\lesssim \|\Pi^\eps\|_{\gamma;\dbar\K}^{(\eps)}\$ f\$_{\gamma,\eta;\dbar\K;\eps},
\end{equation}
where $\bar\Gamma=\gamma+\beta$ and $\bar\eta=(\alpha\wedge\eta)+\beta$ and $f\in\CD_\eps^{\gamma,\eta}(V,\Gamma^\eps)$.
Given a second discrete model $(\bar\Pi^\eps,\bar\Gamma^\eps)$ with associated convolution operator $\bar\CK_\gamma^\eps$, and $\bar f\in\CD_\eps^{\gamma,\eta}(V,\bar\Gamma^\eps)$ taking values in $V$, we assume that
\begin{equation}
\label{eq:Schauderweightedsmalldifference}
\$\CK_{\gamma}^{\eps} f;\bar \CK_\gamma^\eps\bar f\$_{\bar\Gamma,\bar\eta;\K;\eps}\lesssim \|\Pi^\eps\|_{\gamma;\dbar\K}^{(\eps)}\$ f;\bar f\$_{\gamma,\eta;\dbar\K}^{(\eps)} + \|\Pi^\eps-\bar\Pi^\eps\|_{\gamma;\dbar\K}^{(\eps)}\$f\$_{\gamma,\eta;\dbar\K}^{(\eps)}.
\end{equation}
Here, the proportionality constant is independent of $\eps$.
\end{assumption}

\begin{theorem}
\label{thm:Schauderweighted}
Under the same assumption as in Theorem~\ref{thm:Schauder} let $\alpha$ be the regularity of the sector $V$, let $\eta<\gamma$, assume that $\alpha\wedge\eta >-\mathfrak{m}$ and that the operators $\CK_{\gamma}^{\eps}$ and $\bar\CK_\gamma^\eps$ satisfy Assumption~\ref{ass:Schauderweightedsmall}. If we also assume that the reconstruction operator satisfies Assumption~\ref{ass:Rsmallscalesweighted}, then, provided that $\bar\Gamma=\gamma+\beta\notin\N$ and 
$\bar\eta=(\alpha\wedge\eta)+\beta\notin\N$ one has $\CK_{\gamma}^{\eps} f\in \CD^{\bar\Gamma,\bar\eta}_\eps(V)$.
Furthermore, one has the bound
\begin{equation}
\label{eq:Kdifferenceboundweighted}
\$\CK_{\gamma}^{\eps} f;\bar\CK_{\gamma}^{\eps}\bar f\$_{\bar\Gamma,\bar\eta;\K}^{(\eps)}\lesssim
\$f;\bar f\$_{\gamma,\eta;\dbar\K}^{(\eps)}+ \|\Pi^{\eps}-\bar\Pi^{\eps}\|_{\gamma;\dbar\K}^{(\eps)}.
\end{equation}
Let $(\Pi,\Gamma)$ be a continuous model with associated convolution operator $\CK_\gamma$. If additionally Assumption~\ref{a:TandKcontdisc} is satisfied, then for any $f\in\CD^{\gamma,\eta}(\Gamma)$ and any $f^\eps\in\CD_{\eps}^{\gamma,\eta}(\Gamma^\eps)$ we have the estimate
\begin{equation}
\label{eq:Kdiffweightedcontdisc}
\begin{aligned}
&\$\CK_\gamma f;\CK_{\gamma}^{\eps}f^\eps\$_{\bar\Gamma,\bar\eta;\K}
\lesssim \|\Pi;\Pi^\eps\|_{\gamma;\dbar\K}\$f\$_{\gamma,\eta;\dbar\K} + \|\Pi^{(\eps)}\|_{\gamma;\dbar\K}^{(\eps)}\$f;f^\eps\$_{\gamma,\eta;\dbar\K}\\
&+ C(\eps)\|\Pi\|_{\gamma;\dbar\K}\$f\$_{\gamma,\eta;\dbar\K} 
+\$\CK_\gamma^\eps f^\eps\$_{\gamma,\eta;\K;\eps}+ \sup_{\substack{(y,z)\in\K_P\\ \|y-z\|_\s < \eps}}\sup_{\beta <\gamma}\frac{\|\CK_\gamma f(z)-\Gamma_{zy}\CK_\gamma f(y)\|_\beta}{\|y-z\|_\s^{\gamma-\beta}\|y,z\|_P^{\eta-\gamma}}.
\end{aligned}
\end{equation}
\end{theorem} 
\begin{proof}
To show that $\CK_{\gamma}^{\eps} f\in \CD^{\bar\Gamma,\bar\eta}_\eps(V)$ we proceed as in the proof of Theorem~\ref{thm:Schauder}.
We first note that for non-integer values $k$ we can bound the corresponding terms exactly as in the proof of Theorem~\ref{thm:Schauder}.
To bound the differences $\|\CK_{\gamma}^{\eps} f(z)-\Gamma_{zy}^\eps\CK_{\gamma}^{\eps} f(y)\|_k$ for integer values $k$, we distinguish the cases $2^{-n} \leq \|y-z\|_\s$, $2^{-n}\in (\|y-z\|_\s,\frac12\|y,z\|_P]$ and $2^{-n} >\|y-z\|_\s\vee \frac12\|y,z\|_P$. In all these cases it is assumed that $\eps\leq \|y-z\|_\s$ and $\|y-z\|_\s\leq \|y,z\|_P$. The first two cases can be dealt with in a very similar way as in the proof of Theorem~\ref{thm:Schauder}.
We will now provide the details for the third case for the expression $\|\CK_{\gamma}^{\eps} f(z)-\Gamma_{zy}^\eps\CK_{\gamma}^{\eps} f(y)\|_k$ and we refer to~\cite[Prop.~6.16]{Regularity} for a proof in a similar setup.
Following the proof of Theorem~\ref{thm:Schauder} we see that we need to estimate $[-\I'_n+\II'_n-\III'_n]$ in~\eqref{eq:123rewrite}. The first term can again be dealt with as in the proof of Theorem~\ref{thm:Schauder}.
Regarding the difference of the other two terms we may invoke the third item in Assumption~\ref{a:TandK} to estimate
\begin{equation}
\label{eq:IIminusIIIweighted}
|\II'_n-\III'_n|\leq 2^{n(\lceil\gamma+\beta\rceil-\beta)}\|y-z\|_\s^{\lceil\gamma+\beta\rceil-|k|_\s}
\sup_{\varphi \in \Phi}\big|\iota_\eps(\Pi_y^\eps f(y)-\CR^\eps f)(\varphi_{z}^{n-1})\big|.
\end{equation}
To estimate the right hand side of~\eqref{eq:IIminusIIIweighted} we apply the triangle inequality to the last term. We further estimate, making use of Theorem~\ref{thm:recwithweights},
\begin{equation}
\label{eq:estofRweighted}
\big|\iota_\eps(\CR^\eps f)(\varphi_{z}^{n-1})\big|\lesssim 2^{-n(\alpha\wedge\eta)}.
\end{equation}
Now note that the range of values of $n$ we are considering implies in particular that $2^n\lesssim \|y,z\|_P^{-1}$.
Plugging~\eqref{eq:estofRweighted} into~\eqref{eq:IIminusIIIweighted} and summing over those $n$ yields a bound that is a multiple of $\|y,z\|_P^{(\alpha\wedge\eta)+\beta-\lceil\gamma+\beta\rceil}\|y-z\|_\s^{\lceil\gamma+\beta\rceil-|k|_\s}$.
Note that in any case $\|y-z\|_\s\leq \|y,z\|_P$, so that
\begin{equation}
\|y,z\|_P^{(\alpha\wedge\eta)+\beta-\lceil\gamma+\beta\rceil}
\leq \|y,z\|_P^{(\alpha\wedge\eta)-\gamma}\|y-z\|_\s^{\gamma+\beta-\lceil\gamma+\beta\rceil},
\end{equation}
yielding the desired estimate.
It remains to estimate the term involving $\Pi_y^\eps f(y)$ in~\eqref{eq:IIminusIIIweighted}.
Shifting the model, we have
\begin{equs}
\big|\iota_\eps(\Pi_{z}^\eps \Gamma_{zy}^\eps f(y))(\varphi_{z}^{n-1})\big|
&\leq \sum_{\bar\zeta< \gamma, \zeta\leq \bar\zeta}
\big|\iota_\eps(\Pi_{z}^\eps\CQ_\zeta (\Gamma_{zy}^\eps \CQ_{\bar\zeta}f(y)))(\varphi_{z}^{n-1})\big| \\
\label{eq:sumoverzetabarzeta}
&\lesssim \sum_{\bar\zeta<\gamma, \zeta\leq \bar\zeta}
2^{-n\zeta}\|y-z\|_\s^{\bar\zeta-\zeta}\|y\|_P^{(\eta-\bar\zeta)\wedge 0}.
\end{equs}
To estimate this, we first note that we have the bound $\|y\|_P^{(\eta-\bar\zeta)\wedge 0}\leq \|y,z\|_P^{(\eta-\bar\zeta)\wedge 0}$ and $\|y-z\|_\s^{\bar\zeta-\zeta}\leq 2^{-n(\bar\zeta-\zeta)}$. We distinguish two cases. First, if $\eta-\bar\zeta \geq 0$, 
the corresponding terms in~\eqref{eq:sumoverzetabarzeta} are bounded by
$2^{-n\bar\zeta}\leq 2^{-n(\alpha\wedge\eta)}$. This estimate is of the same form as in~\eqref{eq:estofRweighted}, 
so we may conclude as above.
If on the other hand $\eta-\bar\zeta<0$, the corresponding terms in~\eqref{eq:sumoverzetabarzeta} are bounded by
$2^{-n\bar\zeta}\|y,z\|_P^{\eta-\bar\zeta}$.
Taking the prefactor $2^{n(\lceil\gamma+\beta\rceil-\beta)}\|y-z\|_\s^{\lceil\gamma+\beta\rceil-|k|_\s}$ coming from~\eqref{eq:IIminusIIIweighted} into account, and summing over the range of values of $n$ under consideration yields the bound
\begin{equation}
\|y,z\|_P^{\beta-\lceil\gamma+\beta\rceil+\eta}\|y-z\|_\s^{\lceil\gamma+\beta\rceil-|k|_\s}.
\end{equation}
Using that $\|y-z\|_\s\leq \|y,z\|_P$ we see that this is indeed bounded by
\begin{equation}
\|y-z\|_\s^{\gamma+\beta-|k|_\s}\|y,z\|_P^{\eta-\gamma},
\end{equation}
as desired. 
The bound
\begin{equation}
\|\CK_\gamma^{\eps} f(z)\|_{\beta}\lesssim \|z\|_P^{(\bar\eta-\beta)\wedge 0}, 
\end{equation}
may be obtained in a similar way as in~\cite[Prop.~6.16]{Regularity}.
The expression $\|\CK_{\gamma}^{\eps} f(z)-\Gamma_{zy}^\eps f(y)-\bar\CK_{\gamma}^{\eps}\bar f(z)+\bar\Gamma_{zy}^\eps\bar\CK_{\gamma}^{\eps}\bar f(y)\|_k$ can also be dealt with in a similar way and the proof of~\eqref{eq:Kdiffweightedcontdisc} works along the same lines. 
\end{proof}
\section{Local operations}
\label{S5}
\subsection{Multiplication}

One of the surprising results in~\cite{Regularity} is that the multiplication between two (singular) modelled distributions behaves very much like the multiplication of two continuous functions.
We show that under a suitable assumption the same holds true in the setting of the current article.
Before we dive into the details we shortly remind the reader of~\cite[Def.~4.1]{Regularity}, which defines a product to be a continuous bilinear map $(a,b)\mapsto a\star b$ such that
\begin{claim}
\item For every $a\in \CT_\alpha,$ and $b\in \CT_\beta$, one has $a\star b\in \CT_{\alpha+\beta}$.
\item There exists a unit vector $\one \in \CT_0$ such that $\one\star a=a\star \one$ for every $a\in \CT$.
\end{claim}
Given a regularity structure $\TT$ and a pair of sectors $(V,W)$, we say that $(V,W)$ is $\gamma$-regular if $\Gamma(a\star b)=(\Gamma a)\star (\Gamma b)$ for every $\Gamma\in G$, for every $a\in V_\alpha$ and $b\in W_\beta$ such that $\alpha+\beta <\gamma$.
\begin{assumption}
\label{ass:multiplication}
Let $(V,W)$ be a pair of sectors of regularity $\alpha_1$ and $\alpha_2$, respectively. Let $Z^\eps=(\Pi^\eps,\Gamma^\eps)$ and $\bar Z^\eps=(\bar\Pi^\eps,\bar\Gamma^\eps)$ be two discrete models. Let $f_1\in \CD^{\gamma_1}_\eps(V,\Gamma^\eps)$, $g_1\in \CD^{\gamma_1}_\eps(V,\bar \Gamma^\eps)$, $f_2\in \CD^{\gamma_2}_\eps(W,\Gamma^\eps)$, and $g_2\in \CD^{\gamma_2}_\eps(W,\bar \Gamma^\eps)$ and let $\gamma= (\gamma_1+\alpha_2)\wedge (\gamma_2+\alpha_1)$. Then, provided that $(V,W)$ is $\gamma$-regular, we assume that for every compact set $\K$
\begin{equs}
\label{eq:multiplication}
\$f_1\star f_2\$_{\gamma;\K;\eps}&\lesssim \$f_1\$_{\gamma_1;\K;\eps} \$f_2\$_{\gamma_2;\K;\eps},\\
\$f_1\star f_2;g_1\star g_2\$_{\gamma;\K;\eps}&\lesssim \$f_1;g_1\$_{\gamma_1;\K,\eps} + \$f_2;g_2\$_{\gamma_2;\K;\eps}+
\|\Gamma^\eps;\bar \Gamma^\eps\|_{\gamma_1+\gamma_2;\K}^{(\eps)}.
\end{equs}
If $f_1\in \CD^{\gamma_1,\eta_1}_\eps(V,\Gamma^\eps)$, $g_1\in \CD^{\gamma_1,\eta_1}_\eps(V,\bar \Gamma^\eps)$, $f_2\in \CD^{\gamma_2,\eta_2}_\eps(W,\Gamma^\eps)$, and $g_2\in \CD^{\gamma_2,\eta_2}_\eps(W,\bar \Gamma^\eps)$ we further assume that
\begin{equs}
\label{eq:multiplicationweighted}
\$f_1\star f_2\$_{\gamma,\eta;\K;\eps}&\lesssim \$f_1\$_{\gamma_1,\eta_1;\K;\eps} \$f_2\$_{\gamma_2,\eta_2;\K;\eps},\\
\$f_1\star f_2;g_1\star g_2\$_{\gamma,\eta;\K;\eps}&\lesssim \$f_1;g_1\$_{\gamma_1,\eta_1;\K,\eps} + \$f_2;g_2\$_{\gamma_2,\eta_2;\K;\eps}+
\|\Gamma^\eps;\bar \Gamma^\eps\|_{\gamma_1+\gamma_2;\K}^{(\eps)},
\end{equs}
where $\eta=(\eta_1+\alpha_2)\wedge (\eta_2+\alpha_1)\wedge (\eta_1+\eta_2)$.
\end{assumption}
Under Assumption~\ref{ass:multiplication}, with the same choice of coefficients, a straightforward adaptation of the 
proofs in~\cite[Sections 4 and 6.2]{Regularity} yield that:
\begin{itemize}
\item $f_1\star f_2\in \CD^{\gamma}_\eps$ ($f_1\star f_2\in \CD^{\gamma,\eta}_\eps$) provided that $f_1\in \CD^{\gamma_1}_\eps(V)$ ($f_1\in \CD^{\gamma_1,\eta_1}_\eps(V)$) and $f_2\in \CD^{\gamma_2}_\eps(W)$ ($f_2\in \CD^{\gamma_2,\eta_2}_\eps(W)$).
\item The product between two (singular) modelled distributions is continuous. See~\cite[Props~4.10, 6.12]{Regularity} for precise quantitative statements.
\end{itemize} 
One can furthermore verify that all examples mentioned in Section~\ref{S2} (discrete, semidiscrete, continuous and transparent) satisfy Assumption~\ref{ass:multiplication}.

\subsection{Composition with a smooth function}

We shortly review the setup of \cite[Sec.~4.2]{Regularity}.
Given a function-like sector $V$ (i.e., a sector with regularity zero), one can write $a\in V$ as $a=\bar{A}\one+ \tilde{a}$ with $\tilde{a}\in \CT_0^{+}$. For a smooth function
$F:\R^n\to\R$, we define
\begin{equation}
\hat{F}(a)=\sum_k \frac{D^k F(\bar{A})}{k!}\tilde{a}^{\star k}\;,
\end{equation}
where the sum is locally finite. Here, $a=(a_1,\ldots, a_n)$ with $a_i\in V$ and $k$ is a multiindex. 
\begin{assumption}
\label{ass:composition}
Let $V$ be a function-like sector. Fix $\gamma>0$, and let $F\in \CC^{\kappa}(\R^n,\R)$ for some $\kappa\geq \gamma/\zeta\vee 1$, where $\zeta>0$ is the smallest value such that $V_{\zeta}\neq \emptyset$. Given a collection of $n$ functions $f_i\in \CD^{\gamma}_\eps(V)$ for a $\gamma$-regular sector $V$, define $\hat{F}_\gamma(f)(x)= Q_\gamma^{-}\hat{F}(f(x))$, where $f=(f_1,\ldots, f_n)$. We assume that for every compact set $\K$,
\begin{equation}
\label{eq:composition}
\$\hat{F}_\gamma (f)\$_{\gamma;\K;\eps}\lesssim 1.
\end{equation}
If furthermore $\kappa\geq (\gamma/\zeta\vee 1)+1$, then if $g=(g_1,\ldots,g_n)$ for some functions $g_1,\ldots, g_n\in\CD_\eps^{\gamma}(V)$, we also assume that
\begin{equation}
\label{eq:difference}
\$\hat{F}(f);\hat{F}(g)\$_{\gamma;\K;\eps}\lesssim \$f;g\$_{\gamma;\K}.
\end{equation}
In~\eqref{eq:composition} the proportionality constant is allowed to depend only on the norm of $f$, whereas in~\eqref{eq:difference} it is allowed to depend also on the norm of $g$.
In case that $f_i,g_i\in \CD^{\gamma,\eta}_\eps(V)$ for some $\eta\in[0,\gamma]$, we assume the bounds~\eqref{eq:composition} and~\eqref{eq:difference} to hold for the norms $\$\cdot\$_{\gamma,\eta;\K;\eps}$ and $\$\cdot;\cdot\$_{\gamma,\eta;\K;\eps}$, respectively.
\end{assumption}
The arguments in~\cite[Secs~4.2, 6.3]{Regularity} then show that $\hat{F}_\gamma(f)$ defined as above defines an element in $\CD^{\gamma}_\eps(V)$ and $\CD^{\gamma,\eta}_\eps(V)$, respectively. Moreover, in the case that  $\kappa\geq (\gamma/\zeta\vee 1)+1$ they also show that $f\mapsto\hat{F}(f)$ is locally Lipschitz continuous in $\CD_\eps^{\gamma}(V)$ and $\CD_\eps^{\gamma,\eta}(V)$, respectively. The same arguments also show that Assumption~\ref{ass:composition} is satisfied for the four examples introduced in Section~\ref{S2}.\\

\subsection{Differentiation}

Following \cite[Def.~5.25]{Regularity}, given a sector $V$, we say that a family of continuous operators $\sD_i:V\to \CT$
for $i$ in some finite index set $I$, is an abstract collection of derivations if
\begin{itemize}
\item[1.] There is a map $g:I\to\{1,2,\ldots,d\}$ such that $\sD_i a\in \CT_{\alpha-\s_i}$ for every $a\in V_\alpha$, and every $i \in I$,
\item[2.] one has $\Gamma\sD_i a =\sD_i\Gamma a$ for every $a\in V$ and every $i \in I$.
\end{itemize}
Furthermore, a model $(\Pi,\Gamma)$ on $\R^d$ is said to be compatible with $\sD$ if the identity $D_i\Pi_z a= \Pi_z\sD_i a$ holds for every $a\in V$, $z\in\R^d$ and every $i$ where $D_i$ denotes the usual distributional derivative 
in some direction $v_i \in \R^d$ which belongs to the subspace spanned by those directions $j$ such that
$\s_j = \s_{g(i)}$.

\begin{remark}
This is a minor generalisation of the setting of \cite{Regularity} which is natural in some discrete settings.
For example, in the case of a one-dimensional finite difference discretisation of the derivative, one naturally
has one operation corresponding to left-differences and one corresponding to right-differences.
Similarly, if we consider a discretisation of the plane by a triangular grid, we have six natural
differentiation operators.
\end{remark}

This notion of compatibility is not well suited to the current context when $\eps > 0$. Indeed, given $f \in \CX_\eps$, 
there is no reason to assume in general that $D_i \iota_\eps f$ is again in the range of $\iota_\eps$. 
Instead, we make the following assumption where, 
given a compact set $\K$ and $h\in \R$, we denote by $\tau_h\K$ the translation of $\K$ in direction $h$, namely
$\tau_h\K= \K+h$:
\begin{assumption}
\label{ass:differentationonXeps}
There are operators $\mathfrak{D}_{i}^\eps$, $i\in I$, on $\CX_\eps$ such that for all $f\in\CX_\eps$, all $i\in I$, all $\alpha\in\R$, all compact sets $\K_\eps$, and all $z\in\R^d$,
\begin{equation}
\label{eq:Dismallscale}
\|\mathfrak{D}_{i}^\eps f\|_{\alpha-\s_{g(i)};\K_\eps;z;\eps}\lesssim \sup_{h\in\CB_\s(0,\eps)} \|f\|_{\alpha;\tau_h\K_{\eps};z;\eps},
\end{equation}
with a proportionality constant that is independent of $\eps$. 
\end{assumption}
The main definition of this section then reads as follows.
\begin{definition}
\label{def:compatibility}
With the same notation as in Assumption~\ref{ass:differentationonXeps}, we say that a family of continuous operators $\sD_{i}:V\to \CT$, $i\in I$, is an abstract gradient for $\R^d$ with scaling $\s$ if $\sD_{i}a\in\CT_{\alpha-\s_{g(i)}}$ for every $a\in V_\alpha$ and every $i$, and if Property \rm{2.} above is satisfied.
We say that a model $(\Pi^\eps,\Gamma^\eps)$ is compatible with $\sD$ if the identity
\begin{equation}
\label{eq:compatibility}
\mathfrak{D}_{i}^\eps \Pi_z^\eps = \Pi_z^\eps \sD_{i},
\end{equation}
holds for all $i$ and all $z\in\R^d$.
\end{definition}
We have the following result.
\begin{proposition}
\label{prop:differentation}
Let $f\in \CD^{\gamma}_{\alpha,\eps}$ for some $\gamma >\s_{g(i)}$ and some model compatible with $\sD$. Then, $\sD_i f\in  \CD^{\gamma-\s_{g(i)}}_{\alpha-\s_{g(i)},\eps}$, provided that for all compact sets $\K$,
\begin{equation}
\$\sD_i f\$_{\gamma-\s_{g(i)};\K;\eps}\lesssim \$f\$_{\gamma;\K;\eps},
\end{equation}
Under the same assumption, the identity $\CR^\eps \sD_i f= \mathfrak{D}_{i}^\eps \CR^\eps f + H_{i}^\eps f$ holds. Here, the operator $H_{i}^\eps$ satisfies
\begin{equation}
\label{eq:H}
\| H_{i}^\eps f\|_{\gamma-\s_{g(i)};\K_\eps;z;\eps}\lesssim \|\Pi^\eps\|_{\gamma;\bar\K_\eps^\eps}^{(\eps)}\sup_{h\in\CB_\s(0,\eps)} \$ f\$_{\gamma;\tau_h\K_\eps;\eps},
\end{equation}
for all compact sets $\K_\eps$ (of diameter at most $2\eps$) and all $z\in\R^d$. Here, we denoted by $\bar\K_\eps^{\eps}$ the $(1+\eps)$-fattening of $\K_\eps$.
\end{proposition}
\begin{proof}
The first property is a consequence of the respective definitions.
To see~\eqref{eq:H}, note that for any compact set $\K_\eps$ and any $z\in\K_\eps$, by Assumptions~\ref{a:rec}, \ref{ass:differentationonXeps} and the fact that the model is compatible with $\sD$, 
\begin{equation}
\begin{aligned}
\|\Pi_z^{\eps} \sD_i f(z)- \mathfrak{D}_{i}^{\eps} \CR^{\eps} f\|_{\gamma-\s_{g(i)};\K_\eps;z;\eps}
&=\|\mathfrak{D}_{i}^\eps (\Pi_z^\eps f(z)-\CR^\eps f)\|_{\gamma-\s_{g(i)};\K_\eps;z;\eps}\\
&\lesssim \sup_{h\in\CB_\s(0,\eps)}\|\Pi_z^\eps f(z)-\CR^\eps f\|_{\gamma;\tau_h\K_{\eps};z;\eps}\\
&\lesssim \|\Pi^\eps\|_{\gamma;\bar\K_\eps^{\eps}}^{(\eps)}\sup_{h\in\CB_\s(0,\eps)}\$ f\$_{\gamma;\tau_h\K_{\eps};\eps}.
\end{aligned}
\end{equation}
The claim now follows from this chain of inequalities and Assumption~\ref{a:rec}.
\end{proof}
\begin{remark}
\label{rem:examplederivative}
To illustrate the above definitions, consider the polynomial regularity structure $\bar\TT$ and assume that $\CX_\eps=\R^{\Lambda_\eps}$, where $\Lambda_\eps\subset\R$ is a graph of degree two, i.e., each vertex has two neighbours. Assume that the distance between two vertices is between $\eps$ and $2\eps$. We moreover assume that the action of the model $(\Pi^\eps,\Gamma^\eps)$ on the monomials $X^k$ are given via
$(\Pi_z^\eps X^k)(y)= (y-z)^k$ and $\Gamma_{yz}^\eps X^k= (X+ (y-z))^k$.
A natural candidate for $\mathfrak{D}^\eps$ is given by
\begin{equation}
\label{eq:finitedifference}
(\mathfrak{D}^\eps f)(z)= \frac{f(z+\eps_z)-f(z)}{\eps_z},
\end{equation}
where $\eps_z$ denotes the distance of $z$ from its neighbour to the right.
To enforce the compatibility condition~\eqref{eq:compatibility} one \emph{could} define
 $\sD$ via
\begin{equation}
\sD X^k =  \frac{1}{\eps_z}[(X+\eps_z)^k- X^k]
= \sum_{\ell=0}^{k-1}{ k \choose l} X^{\ell}\eps_z^{k-\ell-1}.
\end{equation}
This choice is motivated by the discrete product rule $(\mathfrak{D}^\eps fg)(z)= (\mathfrak{D}^\eps f)(z)g(z)+(\mathfrak{D}^\eps g)(z)f(z)+
\eps_z(\mathfrak{D}^\eps f)(z)(\mathfrak{D}^\eps g)(z)$. The problem with the above definition however is, that it makes the structural object
$\sD$ dependent on $\eps_z$. A way to circumvent this is to introduce a new symbol $\mathcal{E}$, having homogeneity $\s$  as in \cite{KPZJeremy}, and such that additionally for every $\tau\in\bar\TT$ the symbol $\CE\tau$ has homogeneity $\s+|\tau|$ (in the multidimensional case one would introduce symbols $\CE_{i}$ with homogeneity $\s_{g(i)}$). One can then define
\begin{equation}
\sD X^k 
=\sum_{\ell=0}^{k-1}{ k \choose l} X^{\ell}\mathcal{E}^{k-\ell-1}
\end{equation}
and letting the action of $\Pi^\eps$ on $\mathcal{E}^{k}$ be given by $\Pi_z^{\eps}X^k\CE^{\ell}= \eps_z^{\ell}\Pi_z^{\eps}X^k$
yields that the model $(\Pi^\eps,\Gamma^\eps)$ is compatible with $\sD$.
\end{remark}
\begin{remark}
Assume that $(\Pi^\eps,\Gamma^\eps)$ is a model that is compatible with $\sD$ and that $\Pi_z^\eps$ is injective on $\CT$ for every $z\in\R^d$. In this case the identity $\sD_i\Gamma=\Gamma\sD_i$ is automatic. Indeed, one has
\begin{equs}
\Pi_z^\eps\Gamma_{zy}^\eps \sD_i a &= \Pi_y^\eps \sD_i a
= \mathfrak{D}_{i}^\eps \Pi_y^\eps a\\ &= \mathfrak{D}_{i}^\eps \Pi_z^\eps\Gamma_{zy}^\eps a
= \Pi_z^\eps \sD_i \Gamma_{zy}^\eps a\;,
\end{equs} 
so that $\Gamma_{zy}^\eps \sD_i= \sD_i \Gamma_{zy}^\eps$ as desired.
\end{remark}
\begin{remark}
Let $(\Pi^\eps,\Gamma^\eps)$ be a discrete model, and let $g$ be the identity, further let $\mathfrak{D}_i^\eps$ be a finite difference approximation of the usual gradient in direction $i$ (as for instance in Remark~\ref{rem:examplederivative}) and assume that $\sD$ is compatible with $(\Pi^\eps,\Gamma^\eps)$.
Given a scaled test function $\varphi_z^{\lambda}$ with $\lambda\geq\eps$, integration by parts typically yields
for any $a\in\CT$,
\begin{equation}
(\Pi_z^{\eps}\sD_i a)(\varphi_z^{\lambda})
= (\mathfrak{D}_i^{\eps}\Pi_z^{\eps}a)(\varphi_z^{\lambda})
= -(\Pi_z^{\eps} a)(\mathfrak{D}_i^{\eps} \varphi_z^{\lambda}).
\end{equation}
Note that $\mathfrak{D}_i^{\eps} \varphi_z^{\lambda}$ often behaves like a rescaled version of $\varphi_z^{\lambda}$, so that analytical bounds on $\mathfrak{D}_ia$ are implied by analytical bounds on $a$. This is analogous to the continuous case.
\end{remark}
\begin{remark}
In the transparent case $g$ is the identity and \eqref{eq:H} already implies that $H\equiv 0$. Hence, in this case one has the identity $\CR^\eps \sD_i= \mathfrak{D}_i^\eps\CR^\eps$ (and $\mathfrak{D}_i^\eps$ coincides with the distributional derivative $D_i$).
In all other three examples mentioned in Section~\ref{S2}, this identity is not necessarily true. This is of course not very surprising since all objects are only described up to some error term.
\end{remark}

\section{A fixed point Theorem}
\label{S6}
The goal of this section is to establish a fixed point theorem for the discrete analogue of the setting 
in~\cite{Regularity}. More precisely, in order to avoid the problem of having to control the behaviour 
of functions at infinity, we assume that there is a discrete subgroup $\mathscr{S}$ of the group of 
isometries of $\R^{d-1}$ acting on $\R^d$ via $T_g(t,x)= (t,T_g x)$ for all $g\in\mathscr{S}$. Here, 
$t\in\R$ and $x\in \R^{d-1}$ and we denote points in $\R^d$ either by $(t,x)$ or by $z$.
A further assumption we make is that the fundamental domain of $\mathscr{S}$ is compact.
Moreover, we assume that $\mathscr{S}$ acts on our regularity structures and that all models we are considering are adapted to it in the sense of~\cite[Sections 3.6 and 5.3]{Regularity} (think for instance of $\mathscr{S}$ restricting the space variable to the torus).

We note at this point that following the approach of Hairer and Labb\'e~\cite{HairerLabbe} one could probably 
also deal with non-compact situations. However, to keep this exposition technically less involved we refrain 
from elaborating more on it.
One feature of SPDEs is that they come with boundary data, which in our context will typically given by an initial condition. Depending on the data the solution to the SPDE at hand may have a singularity for small times. To deal with
these situations we let the hyperplane $P$, introduced in Section~\ref{S3.1}, be given by the ``time $0$''-hyperplane, i.e.
\begin{equation}
\label{eq:time0}
P=\{(t,x):\, t=0\}.
\end{equation}
Further, we let $R^{+}:\R\times\R^{d-1}\to\R$ be the indicator function of $\{(t,x)\,:\, t \ge 0\}$.
We assume that the map $f\mapsto R^{+}f$ is bounded from $\CD^{\gamma,\eta}_\eps$ to $\CD^{\gamma,\eta}_\eps$, uniformly
over $\eps$, for any choice of $\gamma$ and $\eta$. This assumption is satisfied for all four examples mentioned in Section~\ref{S2}.
Given $T\in\R$, we set $O=[-1,3]\times\R^{d-1}$ and $O_T=(-\infty,T]\times\R^{d-1}$, and 
we use $\$\cdot\$_{\gamma,\eta;T}^{(\eps)}$ as a shorthand for $\$\cdot\$_{\gamma,\eta;O_T}^{(\eps)}$.
For many concrete SPDEs it may not be possible to decompose the Green's function $G^\eps$ of the linear part of the equation such that it satisfies the assumptions of Section~\ref{S4}. However, it is often possible to write $G^{\eps}=K^{\eps}+R^{\eps}$, where for some $\beta>0$, $K^\eps$ satisfies all assumptions of Section~\ref{S4} and $R^\eps$  is a smooth, compactly supported remainder. The same remark holds for the Green's function $G$ corresponding to the linear part of an equation defined on $\R^d$, i.e., we write $G$ as $G=K+R$, where $K$ satisfies all assumptions of Section~\ref{S4.1}.
We make the following assumption.
\begin{assumption}
\label{ass:R}
Fix two discrete models $Z^\eps=(\Pi^\eps,\Gamma^\eps)$ and $\bar{Z}^\eps=(\bar{\Pi}^\eps,\bar{\Gamma}^\eps)$, and $\eta>-\s_1$. Let $V$ be a sector of regularity $\alpha>-\s_1$. We assume that the operator $R^\eps:\CX_\eps\to\CX_\eps$ alluded to above can be lifted to an operator $R_{\gamma}^\eps:\CX_\eps\to \CD^{\gamma}_\eps(\Gamma^\eps)$ for any $\gamma>0$. We moreover assume that for every $T\in (0,1]$, every $\gamma>0$ and every $f^\eps\in\CD^{\gamma,\eta}_\eps(\Gamma^\eps)$,
\begin{equation}
\CR^\eps R_{\gamma}^\eps= R^\eps,
\end{equation} 
and
\begin{equation}
\begin{aligned}
&\$R_{\gamma+\beta}^\eps \CR^{\eps}R^{+} f^\eps\$_{\gamma+\beta,\bar{\eta};T}^{(\eps)}\lesssim T\$f^\eps\$_{\gamma,\eta;T}^{(\eps)},\\
\end{aligned}
\end{equation}
and 
\begin{equation}
\begin{aligned}
\$R_{\gamma+\beta}^\eps\CR^{\eps}R^{+}f^\eps;R_{\gamma+\beta}^\eps\bar{\CR^\eps}R^{+}\bar{f}^\eps\$_{\gamma+\beta,\bar{\eta};T}^{(\eps)}
\lesssim T(\$f^\eps;\bar{f}^\eps\$_{\gamma,\eta;T}^{(\eps)}+ \$Z^\eps;\bar{Z}^\eps\$_{\gamma;O}^{(\eps)}).
\end{aligned}
\end{equation}
The proportionality constant in the first bound is allowed to depend only on $\$Z^\eps\$_{\gamma;O}^{(\eps)}$ whereas it is allowed to depend also on $\$f^\eps\$_{\gamma,\eta;T}^{(\eps)}$ and $\$\bar{f}^\eps\$_{\gamma,\eta;T}^{(\eps)}$ in the second bound.
Here, $\bar{\CR^{\eps}}$ is the reconstruction operator corresponding to the model $\bar{Z}^\eps$, $\bar{f}^\eps\in\CD^{\gamma,\eta}_\eps(\bar\Gamma^\eps)$ and $\bar{\eta}=\alpha\wedge\eta+\beta-\kappa$ for some $\kappa>0$.
Let $Z=(\Pi,\Gamma)$ be a continuous model. In the same setting as above we assume that there is an operator $R:\CD'(\R^d)\to \CD'(\R^d)$ that can be lifted to an operator $R_{\gamma}:\CD'(\R^d)\to\CD^{\gamma}(\Gamma)$ for any $\gamma>0$ such that there is a constant $C(\eps)>0$ with
\begin{equation}
\begin{aligned}
\$R_{\gamma+\beta}\CR R^{+}f;&R_{\gamma+\beta}^\eps \CR^\eps R^{+} f^{\eps}\$_{\gamma+\beta,\bar{\eta};T;\geq\eps}\\
&\lesssim T\Big(\$Z;Z^\eps\$_{\gamma;O}+ \$f;f^\eps\$_{\gamma,\eta;T} + C(\eps)\$Z\$_{\gamma;O}\Big),
\end{aligned}
\end{equation}
where $f\in\CD^{\gamma,\eta}(Z)$.
Here, the proportionality constant is allowed to depend on the bounds of the models and on $\$f\$_{\gamma,\eta;T}$ and $\$f^\eps\$_{\gamma,\eta;T}^{(\eps)}$.
\end{assumption}
A further assumption that we need is that $K_n^\eps$ and $T_{n,\gamma+\beta}^{\eps}$ defined in Section~\ref{S4} are non-anticipative in the sense that the test functions $\varphi_z^n$ appearing in Equations~\eqref{eq:Tcomponentest}--\eqref{eq:relationKandT} have support contained in
$\{\bar{z}= (\bar{t},\bar{x})\in\R^d:\, \|\bar{z}-z\|_\s\leq 2^{-n},\, \bar{t}\leq t\}$ and ($\CK_{\gamma}^{\eps} f)(t,x)$
depends only on those values $(\bar{t},x)$ such that $\bar{t}\leq t$.
Going carefully through the proofs of Theorems~\ref{thm:Schauder} and~\ref{thm:Schauderweighted} we see that as a consequence of that extra assumption
\begin{equs}[eq:estKwithkapp]
\$\CK_{\gamma}^{\eps} f^\eps\$_{\gamma+\beta,\bar{\eta}; T}^{(\eps)}&\lesssim \$f^\eps\$_{\gamma,\eta; T}^{(\eps)},\\
\$\CK_{\gamma}^{\eps} f^\eps;\bar{\CK}_\gamma^\eps\bar{f}^\eps\$_{\gamma+\beta,\bar{\eta};T}^{(\eps)}
&\lesssim \$f^\eps;\bar{f}^\eps\$_{\gamma,\eta; T}^{(\eps)}+ \|\Pi^{\eps}-\bar{\Pi}^\eps\|_{\gamma; O}^{(\eps)}.
\end{equs}
Here $\bar{\eta}$ is chosen as in Assumption~\ref{ass:R}. 
We assume the same for $K_n$ and $T_{n,\gamma+\beta}$, then~\cite[Thm~7.1]{Regularity} shows the analogue of~\eqref{eq:estKwithkapp} in the case of continuous convolution operators.
\begin{lemma}
\label{lem:Kgammakappa}
With the same notation as in Assumption~\ref{ass:R}, assume that
\begin{equs}
\label{eq:aKgammakappa}
\$\CK_{\gamma}^{\eps} R^{+}f^\eps\$_{\gamma+\beta,\bar{\eta};T;\eps}&\lesssim
T^{\kappa/\s_1}\$ f^\eps\$_{\gamma,\eta;T}^{(\eps)},\\
\$\CK_{\gamma}^{\eps} R^{+}f^\eps;\bar{\CK}_\gamma^\eps R^{+}\bar{f}^\eps\$_{\gamma+\beta,\bar{\eta};T;\eps}&\lesssim
T^{\kappa/\s_1}\big(\$ f^\eps;\bar{f}^\eps\$_{\gamma,\eta;T}^{(\eps)}+ \$Z^\eps;\bar{Z}^\eps\$_{\gamma;O}^{(\eps)}),
\end{equs}
where the proportionality constant in the first bound is allowed to depend only on $\$Z^\eps\$_{\gamma;O}^{(\eps)}$, while  the proportionality constant in the second bound is allowed to depend also on $\$f^\eps\$_{\gamma,\eta;T}^{(\eps)}$ and $\$\bar{f}^\eps\$_{\gamma,\eta;T}^{(\eps)}$.
Then, 
\begin{equs}
\$\CK_{\gamma}^{\eps} R^{+}f^\eps\$_{\gamma+\beta,\bar{\eta};T}^{(\eps)}&\lesssim
T^{\kappa/\s_1}\$ f^\eps\$_{\gamma,\eta;T}^{(\eps)},\\
\$\CK_{\gamma}^{\eps} R^{+}f^\eps;\bar{\CK}_\gamma^\eps R^{+}\bar{f}^\eps\$_{\gamma+\beta,\bar{\eta};T}^{(\eps)}&\lesssim
T^{\kappa/\s_1}\big(\$ f^\eps;\bar{f}^\eps\$_{\gamma,\eta;T}^{(\eps)}+ \$Z^\eps;\bar{Z}^\eps\$_{\gamma;O}^{(\eps)}\big),
\end{equs}
and
\begin{equation}
\label{eq:Kgammakappacontdisc}
\begin{aligned}
\$\CK_\gamma R^{+}f;&\CK_{\gamma}^{\eps}R^{+}f^\eps\$_{\gamma+\beta,\bar{\eta};T;\geq \eps}\\
&\lesssim T^{\kappa/\s_1}\big(\$ f;f^\eps\$_{\gamma,\eta;T;\geq\eps}+ \$Z;Z^\eps\$_{\gamma;O}
+C(\eps)\$Z\$_{\gamma;O}\big).
\end{aligned}
\end{equation}
Here, the proportionality constants are analogue to the ones is~\eqref{eq:aKgammakappa} and $C(\eps)$ is as in Theorem~\ref{thm:Schaudercontdisc}.
\end{lemma}
\begin{proof}
The proof uses some ideas of~\cite[Proof of Lemma 6.5]{Regularity}. We therefore only sketch some of the arguments and provide details only for those ingredients that are new. 
We first note that as a consequence of Theorem~\ref{thm:Schauderweighted} and the definition of the $\CD_{\eps}^{\gamma,\eta}$-norm we only need to control components $k$ such that $|k|_\s<\alpha\wedge \eta+\beta$.
We define a sequence $z_n =S_P^{2^{-n}}z$, where 
$S_P^{2^{-n}}z= (2^{-n}t,x_1,\ldots, x_{d-1})$. Then,
$\|z_{n+1}-z_n\|_\s = \|z_{n+1}\|_P$, provided that $d_\s(z_{n+1},P)\leq 1$ (which is satisfied for $n$ large enough).
One then proceeds via reverse induction. More precisely, we assume that there is a multiindex $k$ such that that
\begin{equation}
\|\CK_{\gamma}^{\eps} R^{+} f(z)\|_m\lesssim \|z\|_P^{(\alpha\wedge \eta)+\beta-m},\quad \mbox{for all }m>|k|_\s.
\end{equation}
Since $\CK_{\gamma}^{\eps} R^{+}f\in\CD^{\gamma+\beta,\alpha\wedge\eta+\beta}_\eps$, this is certainly the case for $k\in A$ such that $|k|_\s$ is smaller than $\alpha\wedge\eta+\beta$.
Provided that $\|z_{n+1}-z_n\|_\s\geq \eps$, one may then show as in~\cite{Regularity} that 
\begin{equation}
\|\CK_{\gamma}^{\eps} \CR^\eps R^{+} f(z_{n+1})-\CK_{\gamma}^{\eps} \CR^\eps R^{+} f(z_n)\|_{|k|_\s} \lesssim 2^{-n(\alpha\wedge\eta+\beta-|k|_\s)}.
\end{equation}
Assume now that for any fixed $\bar{c}>0$ and $\bar{z}\in\R^d$ such that $\|\bar{z}\|_P\leq \bar{c}\eps$ we have the estimate
\begin{equation}
\label{eq:smallzest}
\|\CK_{\gamma}^{\eps} R^{+}f(\bar{z})\|_{|k|_\s}\lesssim \eps^{\eta\wedge\alpha+\beta-|k|_\s}\quad\mbox{for }|k|_\s<\eta\wedge\alpha +\beta.
\end{equation} 
In this case a teleskop sum argument as in~\cite[Lemma 6.5]{Regularity} yields the claim. It remains to establish~\eqref{eq:smallzest}. 
To that end we first intend to show that
\begin{equation}
\label{eq:RRfzero}
(T_{\gamma}^{\eps}\CR^\eps R^{+}f)(t,x)=0,\quad \mbox{if }t\leq -\bfc\eps.
\end{equation}
Here, $\bfc\geq 0$, is the constant introduced in Definition~\ref{def:admissible}.
We first note that by~\eqref{eq:Tcomponentest},
\begin{equation}
\big|\CQ_k((T_{n,\gamma}^{\eps}\CR^\eps R^{+}f)(t,x))\big|
\leq 2^{n(|k|_\s-\beta)}
\sup_{\varphi_{(t,x)}^{n}}\big|(\iota_\eps \CR^\eps R^{+}f)(\varphi_{(t,x)}^n)\big|,
\end{equation}
for all multiindex $k$ and any $n$.
By our additional assumption of non-anticipativity, we see that the support of each $\varphi_{(t,x)}^{n}$ is contained in the set
\begin{equation}
\{(\bar{t},\bar{x}):\, \bar{t}\leq t, \|(\bar{t},\bar{x})-(t,x)\|_\s\leq 2^{-n}\}.
\end{equation}
In what follows we write $\varphi_z^n$ instead of $\varphi_{(t,x)}^n$. To proceed we fix a smooth function $\Psi$ as in the proof of Theorem~\ref{thm:reconstruction}. 
Let $\bar{z}\in\Lambda_N^{\s}$ (recall that $\Lambda_N^{\s}$ denotes the $d$-dimensional grid with mesh $2^{-N}$, see Equation~\eqref{eq:lambda}), we define a function $\Psi_{\bar{z}}$ on $\R^d$ via
\begin{equation}
\Psi_{\bar{z}}(z)=\prod_{i=1}^{d}\Psi(2^{N\s_i}(z_i-\bar{z}_i)).
\end{equation}
As a consequence of~\eqref{eq:summingtoone} we can write
\begin{equation}
\varphi_z^{n}=\sum_{\bar{z}\in\Lambda_N^{\s}}\varphi_z^{n}\Psi_{\bar{z}},
\end{equation}
and we see that $\chi_{z,\bar{z}}^n:= 2^{(N-n)|\s|}\varphi_z^n\Psi_{\bar{z}}$ is a test function with support of size $\eps$.  Note that by~\eqref{eq:recwithweights3},
\begin{equation}
|\iota_\eps(\CR^\eps R^{+} f)(\chi_{z,\bar{z}}^n)|
\lesssim 2^{-n\alpha\wedge\eta}\|\Pi^\eps\|_{\gamma;O}^{(\eps)}\$R^{+}f\$_{\gamma,\eta;[\chi_{z,\bar{z}}^n]}^{(\eps)}.
\end{equation}
Thus,~\eqref{eq:RRfzero} is a consequence of the fact that $R^{+}f\equiv 0$ on the $\bfc \eps$-fattening of $[\chi_{z,\bar{z}}^n]$, provided that $t\leq -\bfc \eps$.
To continue, note that
\begin{itemize}
\item $\CI R^{+}f\in \CT_{\alpha\wedge\eta+\beta}^{+}$, and
\item $\CQ_k((T_{n,\zeta}^{\eps}\Pi_z^\eps \CQ_{\zeta} R^{+}f(z))(z))= \CQ_k((T_{n,\gamma}^{\eps}\Pi_z^\eps \CQ_{\zeta} R^{+}f(z))(z))$ for $|k|_\s<\zeta+\beta$ and all $n$ (note that on the left hand side the index of $T^\eps$ is $\zeta$, whereas on  the right hand side it is $\gamma$). 
\end{itemize}
As a consequence of these two items and the definition of $\CK_\gamma^\eps$, the only term we need to estimate to derive~\eqref{eq:smallzest} is
\begin{equation}
\CQ_k((T_{n,\gamma}^{\eps}\CR^{\eps}R^{+}f)(\bar{z})).
\end{equation}
Fix $z=(t,\bar x)\in\R^d$ such that $t=-\bfc\eps$. Note that $z$ is chosen such that the spatial coordinates of $z$ and $\bar z$ coincide, in particular $\|z-\bar z\|_\s \approx \eps$. Let $n\in\N$ such that $\|z-\bar z\|_\s\leq 2^{-n}$. Then,~\eqref{eq:RRfzero} yields
\begin{equation}
\CQ_k((T_{n,\gamma}^{\eps}\CR^{\eps}R^{+}f)(\bar{z}))
= \CQ_k((T_{n,\gamma}^{\eps}\CR^{\eps}R^{+}f)(\bar{z}))- \CQ_k(\Gamma_{\bar z z}^{\eps}(T_{n,\gamma}^{\eps}\CR^{\eps}R^{+}f)(z)).
\end{equation}
Using~\eqref{eq:Tdifferentest} we see that the above is bounded by
\begin{equation}
2^{n(\lceil \gamma+\beta\rceil-\beta)}\|z-\bar{z}\|_\s^{\lceil\gamma+\beta\rceil-|k|_\s}
\sup_{\varphi_{\tilde{z}}^{n-1}}\big|(\iota_\eps\CR^\eps R^{+}f)(\varphi_{\tilde{z}}^{n-1})\big|.
\end{equation}
Applying Theorem~\ref{thm:recwithweights} and summing over $n$ such that  $\|z-\bar z\|_\s\leq 2^{-n}$ yields an upper bound of the order $\eps^{\alpha\wedge\eta+\beta-|k|_\s}$. If $\|z-\bar z\|_\s >2^{-n}$, then our choice of $z$ and $\bar z$ yield
$2^{-n}\approx\eps$. Thus, we may conclude with~\eqref{eq:Tcomponentest} and Theorem~\ref{thm:recwithweights}.
The bound on $\|\CK_{\gamma}^{\eps} R^{+} f;\bar{\CK}_\gamma^\eps R^{+}\bar{f}\|_{\bar{\Gamma},\bar{\eta};T}^{(\eps)}$ and~\eqref{eq:Kgammakappacontdisc} follow
in a similar manner. 
\end{proof}
Recall that we only have the identity $\CR^\eps \CK_{\gamma}^{\eps}= K^\eps\CR^\eps + A^\eps f$, where $A^\eps$ is an operator that can be controlled on small scales.  We saw in Remark~\ref{rem:A} that typically one is able to lift $A^\eps$ to an operator $\CA^\eps:\CD^{\gamma}_\eps\to \CD^{\gamma+\beta}_\eps$. The precise form of $\CA^\eps$ in~\eqref{eq:CAdef} and the arguments in the proof of Lemma~\ref{lem:Kgammakappa} show that the following assumption is natural.
\begin{assumption}
\label{ass:Aeps}
We assume that there is an operator $\CA^\eps:\CD^{\gamma}_\eps\to \CD^{\gamma+\beta}_\eps$ taking values only in the polynomial part of the regularity structure such that $\CR^\eps \CA^\eps =A^\eps$. Moreover, in the setting of Assumption~\ref{ass:R} we have the estimates
\begin{equs}
\$\CA^\eps R^{+}f^\eps\$_{\gamma+\beta,\bar{\eta};T}^{(\eps)}&\lesssim
T^{\kappa/\s_1}\$ f^\eps\$_{\gamma,\eta;T}^{(\eps)},\\
\$\CA^\eps R^{+}f^\eps;\bar{\CA}^\eps R^{+}\bar{f}^\eps\$_{\gamma+\beta,\bar{\eta};T}^{(\eps)}&\lesssim
T^{\kappa/\s_1}\big(\$ f^\eps;\bar{f}^\eps\$_{\gamma,\eta;T}^{(\eps)}+ \$Z^\eps;\bar{Z}^\eps\$_{\gamma;O}^{(\eps)}\big).
\end{equs}
The smallest proportonality constant in the first bound is denoted $\$\CA^\eps\$_{\gamma+\beta,\bar{\eta};T}^{(\eps)}$
and is allowed to depend only on $\|Z^\eps\|_{\gamma;O}^{(\eps)}$.
The proportionality constant in the second bound is only allowed to depend on the norm of the models, on the norms of $\CA^\eps$ and $\bar{\CA}^\eps$ and on $\$f^\eps\$_{\gamma,\eta;T}^{(\eps)}, \$\bar{f}^\eps\$_{\gamma,\eta;T}^{(\eps)}$.
\end{assumption} 
Before we can finally state the fixed point theorem we are after, we need to introduce more setup. 
Let $\gamma\geq\bar{\gamma} >0$, $F:\R^d\times \CT_{\gamma}\to \CT_{\bar{\gamma}}$ and $f:\R^d\to \CT_{\gamma}$.
We define
\begin{equation}
(F(f))(z)= F(z,f(z)).
\end{equation}
Fix $R>0$ and assume that $F$ maps $\CD^{\gamma,\eta}_{P,\eps}$ into $D^{\bar{\gamma},\bar{\eta}}_{P,\eps}$ for some $\eta,\bar\eta\in\R$. We say that $F$ is strongly locally Lipschitz continuous if
\begin{equation}
\label{eq:stronglip}
\$F(f);F(g)\$_{\bar{\gamma},\bar{\eta};\K}^{(\eps)}\lesssim (\$f;g\$_{\gamma,\eta;\K}^{(\eps)}+\$Z^\eps;\bar{Z}^\eps\$_{\gamma;\bar{\K}}^{(\eps)}).
\end{equation}
Here, $Z^\eps$ and $\bar{Z}^\eps$ and $f\in \CD^{\gamma,\eta}_{\eps}(Z^\eps)$, $g\in\CD^{\gamma,\eta}_{\eps}(\bar{Z}^\eps)$
are such that $\$Z^\eps\$_{\gamma;\bar\K}^{(\eps)} + \$\bar{Z}^\eps\$_{\gamma;\bar\K}^{(\eps)}\leq R$ and
$\$f\$_{\gamma,\eta;\K}^{(\eps)}+ \$g\$_{\gamma,\eta;\K}^{(\eps)}\leq R$.
We have the following result.
\begin{theorem}
\label{thm:fixedpointthm}
Let $V,\bar{V}$ be two sectors of a regularity structure  $\TT$ with regularities $\zeta,\bar{\zeta}\in\R$ with 
$\zeta\leq \bar{\zeta}+\beta$. Under all the assumptions stated in this section and in the setup described above, for some $\gamma\geq\bar\gamma>0$ and some $\eta\in\R$, let $F:\R^d\times V_{\gamma}\to\bar{V}_{\bar{\gamma}}$ be a smooth function such that, if $f\in\CD^{\gamma,\eta}_{\eps}$ is symmetric with respect to the action $\mathscr{S}$, then $F(f)$ defined above belongs to $\CD^{\bar{\gamma},\bar{\eta}}_{\eps}$ and is also symmetric with respect to $\mathscr{S}$.
Moreover, assume that there is an abstract integration map $\CI$ such that $\CQ_{\gamma}^{-1}\CI\bar{V}_{\bar{\gamma}}\subset V_{\gamma}$.
If $\eta< (\bar\eta\wedge \bar\zeta)+\beta$, $\gamma <\bar\gamma+\beta$, $(\bar\eta\wedge\bar\zeta)>-\s_1$, and $F$ is strongly locally Lipschitz continuous, then, for every $v\in \CD^{\gamma,\eta}_{\eps}$ which is symmetric with respect to $\mathscr{S}$, and for every symmetric model $Z^\eps=(\Pi^{\eps},\Gamma^{\eps})$ for the regularity structure $\TT$ such that $\CI$ is adapted to the kernel $K^{\eps}$, there exists $T^\eps>0$ such that the equation
\begin{equation}
u^\eps=(\CK_{\bar{\gamma}}^\eps-\CA^{\eps}+R_{\gamma}^\eps\CR^{\eps})R^{+}F(u^\eps)+v^\eps
\end{equation}
admits a unique solution $u^\eps\in\CD^{\gamma,\eta}_{\eps}(Z^\eps)$ on $(0,T^\eps)$. The solution map is $(v^\eps,Z^\eps)\mapsto u^\eps$ jointly Lipschitz continuous. Let $Z=(\Pi,\Gamma)$ be a continuous model and under the same assumptions as above let $u\in\CD^{\gamma,\eta}(Z)$ be the unique solution to the fixed point problem
\begin{equation}
u=(\CK_{\bar{\gamma}}+R_{\gamma}\CR)R^{+}F(u)+v
\end{equation}
on some time interval $(0,T)$. Then,
\begin{equation}
\$u;u^\eps\$_{\gamma,\eta;T^\eps\wedge T;\geq \eps}\lesssim \$Z;Z^\eps\$_{\gamma;O}+ C(\eps)\$Z\$_{\gamma;O}+ \$v;v^\eps\$_{\gamma,\eta;T^\eps\wedge T} +\$\CA^\eps\$_{\gamma,\eta;T^\eps\wedge T}^{(\eps)},
\end{equation}
and the proportionality constant only depends on $R$ defined above Equation~\eqref{eq:stronglip}.
\end{theorem}
\begin{remark}
In practice one would like to guarantee that $\liminf_{\eps\to 0}T^\eps >0$. For that it is enough to show that the bounds one gets on the various terms appearing in the theorem above are independent of $\eps$.
\end{remark}
\begin{remark}
Theorem~\ref{thm:fixedpointthm} does not allow for an abstract formulation of~\eqref{eq:generalpde} in case the non-linearity depends on a discrete version of the gradient. The reason for not including this case is that there are numerous ways to define a notion of the discrete derivative. Hence, it is not clear how to state a single theorem dealing with all these possibilities at once. Yet, the estimates on the various objects obtained in this article probably allow to adapt the abstract fixed point problem in most cases.
\end{remark}
\begin{proof}
With Assumptions~\ref{ass:R} and \ref{ass:Aeps}, and with Lemma~\ref{lem:Kgammakappa} at hand the proof follows in the same way as in~\cite[Thm~7.8]{Regularity}.
\end{proof}

\newpage
\appendix
\section{Norm index}
\label{A}
In this appendix we collect the various norms and metrics used in this article together with their meaning.
\begin{center}

\begin{tabular}{ll}\toprule
Norm & Meaning \\
\midrule
$d_\s(\cdot,\cdot)$ & Metric with scaling $\s$ on $\R^d$\\
$\|x-y\|_\s$ & Alternative expression for $d_\s(x,y)$\\
$\|\cdot\|_{\alpha;\K_\eps;z;\eps}$ & Extended seminorm at scale $\eps$ on $\CX_\eps$\\
$\$\cdot\$_{\gamma;\K;\eps}$ & Seminorm at scale $\eps$ on functions $f:\R^d\to \CT_{<\gamma}$\\
$\$\cdot,\cdot\$_{\gamma;\K;\eps}$ & Distance at scale $\eps$ on pairs $(f,g)$ with $f,g:\R^d\to\CT_{<\gamma}$\\
$\$\cdot\$_{\gamma,\eta;\K;\eps}$ & Seminorm at scale $\eps$ on functions $f:\R^d\to\CT_{<\gamma}$\\
$\$\cdot,\cdot\$_{\gamma,\eta;\K;\eps}$ & Distance at scale $\eps$ on pairs $(f,g)$ with $f,g:\R^d\to\CT_{<\gamma}$\\
$\$\cdot\$_{\gamma;\K}^{(\eps)}$ & Norm on the space $\CD_{\eps}^{\gamma}(\Gamma^\eps)$\\
$\$\cdot,\cdot\$_{\gamma;\K}^{(\eps)}$ & Distance between $f\in\CD_\eps^{\gamma}(\Gamma^\eps)$ and $g\in\CD_\eps^{\gamma}(\bar\Gamma^\eps)$\\
$\$\cdot\$_{\gamma,\eta;\K}^{(\eps)}$ & Norm on the space $\CD_{\eps}^{\gamma,\eta}(\Gamma^\eps)$\\
$\$\cdot,\cdot\$_{\gamma,\eta;\K}^{(\eps)}$ & Distance between $f\in\CD_\eps^{\gamma,\eta}(\Gamma^\eps)$ and $g\in\CD_\eps^{\gamma,\eta}(\bar\Gamma^\eps)$\\
$\|\cdot\|_{m}$ & Norm on $\CT_m$\\
$\|\Pi^\eps\|_{\gamma;\K}^{(\eps)}$ & Smallest proportionality constant such that~\eqref{eq:Pi} holds\\
$\|\Gamma^\eps\|_{\gamma;\K}^{(\eps)}$ & Smallest proportionality constant such that~\eqref{eq:Gamma} holds\\
$\$Z^\eps \$_{\gamma;\K}^{(\eps)}$ & Norm on models given by $\$Z^\eps \$_{\gamma;\K}^{(\eps)}= \|\Pi^\eps\|_{\gamma;\K}^{(\eps)}+ \|\Gamma^\eps\|_{\gamma;\K}^{(\eps)}$\\
$\$Z^\eps;\bar Z^\eps\$_{\gamma;\K}^{(\eps)}$ & Pseudo-metric to compare two different models\\
$\| \cdot\|_{P}$ & Distance to $P$ given by $\|z\|_P= (1\wedge d_\s(z,P))$\\
$\|\cdot,\cdot\|_P$ & Distance defined via $\|y,z\|_P= \|y\|_P\wedge \|z\|_P$\\
$\|\Pi-\Pi^\eps\|_{\gamma;\K;\geq \eps}$ & Distance between $\Pi$ and $\Pi^\eps$ at scales above $\eps$\\
$\|\Gamma-\Gamma^\eps\|_{\gamma;\K;\geq\eps}$ & Distance between $\Gamma$ and $\Gamma^\eps$ at scales above $\eps$\\
$\$Z;Z^\eps\$_{\gamma;\K}$ & Distance between limiting model $Z$ and $\eps$-model $Z^\eps$\\ 
$\$f\$_{\gamma;\K}$ & $\CD^\gamma$-norm for a map $f$ with respect to a continous model\\
$\$f;f^\eps\$_{\gamma;\K}$ & Distance between $f\in\CD^\gamma(\Gamma)$ and $f^\eps\in\CD_\eps^{\gamma}(\Gamma^\eps)$\\
$\$f\$_{\gamma,\eta;\K}$ & $\CD^{\gamma,\eta}$-norm for a map $f$ with respect to a continuous model\\
$\$f;f^\eps\$_{\gamma,\eta;\K}$ & Distance between $f\in\CD^{\gamma,\eta}(\Gamma)$ and $f^\eps\in\CD_\eps^{\gamma,\eta}(\Gamma^\eps)$\\
$\|\Pi;\Pi^\eps\|_{\gamma;\K}$ & Distance between $\Pi$ and $\Pi^\eps$\\
$\|\Gamma;\Gamma^\eps\|_{\gamma;\K}$ & Distance between $\Gamma$ and $\Gamma^\eps$\\
$\$f;f^{\eps}\$_{\gamma,\eta;\K;\geq\eps}$ & Distance in $\CD_{\eps}^{\gamma,\eta}$ between $f$ and $f^{\eps}$ at scales above $\eps$\\
\bottomrule
\end{tabular}

\end{center}

\endappendix
\newpage

\bibliographystyle{Martin}

\bibliography{refs}

\end{document}